\documentclass[a4paper]{article}
\usepackage{enumitem}
\usepackage{calc}

\usepackage{lmodern} 

\usepackage{graphicx} 

\usepackage{amsmath, accents}
\usepackage{amssymb,tikz}
\usepackage{pict2e}
\usepackage{amsthm}
\usepackage{amscd}
\usepackage{mathrsfs}
\usepackage{todonotes}
\usetikzlibrary{calc} 

\usepackage{yfonts}

\usepackage{marvosym}
\usepackage[percent]{overpic}

\newtheorem{theorem}{Theorem} [section]
\newtheorem{proposition}[theorem]{Proposition}	
\newtheorem{corollary}[theorem]{Corollary}
\newtheorem{lemma}[theorem]{Lemma}

\newtheorem{remark}[theorem]{Remark}

\theoremstyle{definition}

\newcommand{\C}{\mathbb{C}}

\newcommand{\R}{\mathbb{R}}

\newcommand{\N}{\mathbb{N}}

\newcommand{\re}{\text{\upshape Re\,}}
\newcommand{\im}{\text{\upshape Im\,}}

\newcommand{\CC}[1]{{\color{red}{#1}}}

\tikzset{
	master/.style={
		execute at end picture={
			\coordinate (lower right) at (current bounding box.south east);
			\coordinate (upper left) at (current bounding box.north west);
		}
	},
	slave/.style={
		execute at end picture={
			\pgfresetboundingbox
			\path (upper left) rectangle (lower right);
		}
	}
}

\let\oldbibliography\thebibliography
\renewcommand{\thebibliography}[1]{\oldbibliography{#1}
\setlength{\itemsep}{-0.5pt}}

\def\XXint#1#2#3{{\setbox0=\hbox{$#1{#2#3}{\int}$}
\vcenter{\hbox{$#2#3$}}\kern-.5\wd0}}

\usepackage{tikz}
\usetikzlibrary{arrows}
\usetikzlibrary{decorations.pathmorphing}
\usetikzlibrary{decorations.markings}
\usetikzlibrary{patterns}
\usetikzlibrary{automata}
\usetikzlibrary{positioning}
\usepackage{tikz-cd}
\tikzset{->-/.style={decoration={
				markings,
				mark=at position #1 with {\arrow{latex}}},postaction={decorate}}}
	
	\tikzset{-<-/.style={decoration={
				markings,
				mark=at position #1 with {\arrowreversed{latex}}},postaction={decorate}}}

\usetikzlibrary{shapes.misc}\tikzset{cross/.style={cross out, draw, 
         minimum size=2*(#1-\pgflinewidth), 
         inner sep=0pt, outer sep=0pt}}

\usepackage{pgfplots}
	
\allowdisplaybreaks	

\numberwithin{equation}{section}

\def\ds{\displaystyle}
\def\bigO{{\cal O}}

\makeatletter
\newcommand{\oset}[3][0ex]{%
  \mathrel{\mathop{#3}\limits^{
    \vbox to#1{\kern-2\ex@
    \hbox{$\scriptstyle#2$}\vss}}}}
\makeatother

\usepackage[colorlinks=true]{hyperref}
\hypersetup{urlcolor=blue, citecolor=red, linkcolor=blue}

\begin{document}
\title{Asymptotics for the number of domino tilings \\ of L-shaped Aztec domains}
\author{Christophe Charlier and Tom Claeys}

\maketitle

\begin{abstract}
We obtain precise asymptotics for the weighted number of domino tilings of an L-shaped subset of the Aztec diamond, obtained by removing an approximate rectangle in a corner of the Aztec diamond. 
By tuning the size of the removed corner, we observe different types of asymptotics. For a small removed corner, the number of tilings is close to that of the full Aztec diamond. Enlarging the removed corner to a critical size, a phase transition described in terms of the Tracy-Widom distribution occurs. Further increasing the size of the removed region, we observe a sharp decrease of the number of tilings, until it is finally approximated by the number of tilings of two smaller disjoint Aztec diamonds. We obtain uniform asymptotics for the number of domino tilings which fully describe these transitions.
\end{abstract}
\noindent
{\small{\sc AMS Subject Classification (2020)}: 41A60, 60F10, 60G55.}

\noindent
{\small{\sc Keywords}: Random tilings, determinantal point processes, asymptotic analysis, Riemann-Hilbert problems.}


\section{Introduction and main results}

\paragraph{Domino tilings of L-shaped domains.}
The Aztec diamond $A_N$ of order $N$ is a two-dimensional region consisting of the union of all $2N(N+1)$ squares $S_{j,k}:=[j,j+1]\times [k,k+1]\subset \mathbb R^2$ for which $j,k\in\mathbb Z$ and $S_{j,k}\subset\{(x,y)\in\mathbb R^2: |x|+|y|\leq N+1\}$. A domino is a rectangle covering two adjacent unit squares: vertical dominoes have the form $[j,j+1]\times [k,k+2]$, while horizontal dominoes have the form $[j,j+2]\times [k,k+1]$. A horizontal domino is called a north or south domino depending on whether $j+k+N$ is even or odd, respectively. Similarly, a vertical domino is called an east or west domino depending on whether $j+k+N$ is even or odd, respectively. A domino tiling $T$ of $A_N$ consists of a collection of dominoes covering exactly the region $A_N$ while the interiors of the dominoes are disjoint, see Figure \ref{fig:Aztec}.

Domino tilings of the Aztec diamond as well as lozenge tilings of hexagons have been investigated intensively over the past decades because of their intriguing large $N$ asymptotic behavior and their connections with statistical physics models and random matrix theory, see e.g.\ \cite{Aggarwal, BB, DK2021, Jptrf, GorinSurvey, HYZ2024}. Determinantal structures underlying the tiling models have provided a powerful framework to study their asymptotics and to identify phase transitions.

\begin{figure}
\begin{center}
\begin{tikzpicture}[master]
\node at (0,0) {\includegraphics[width=6cm]{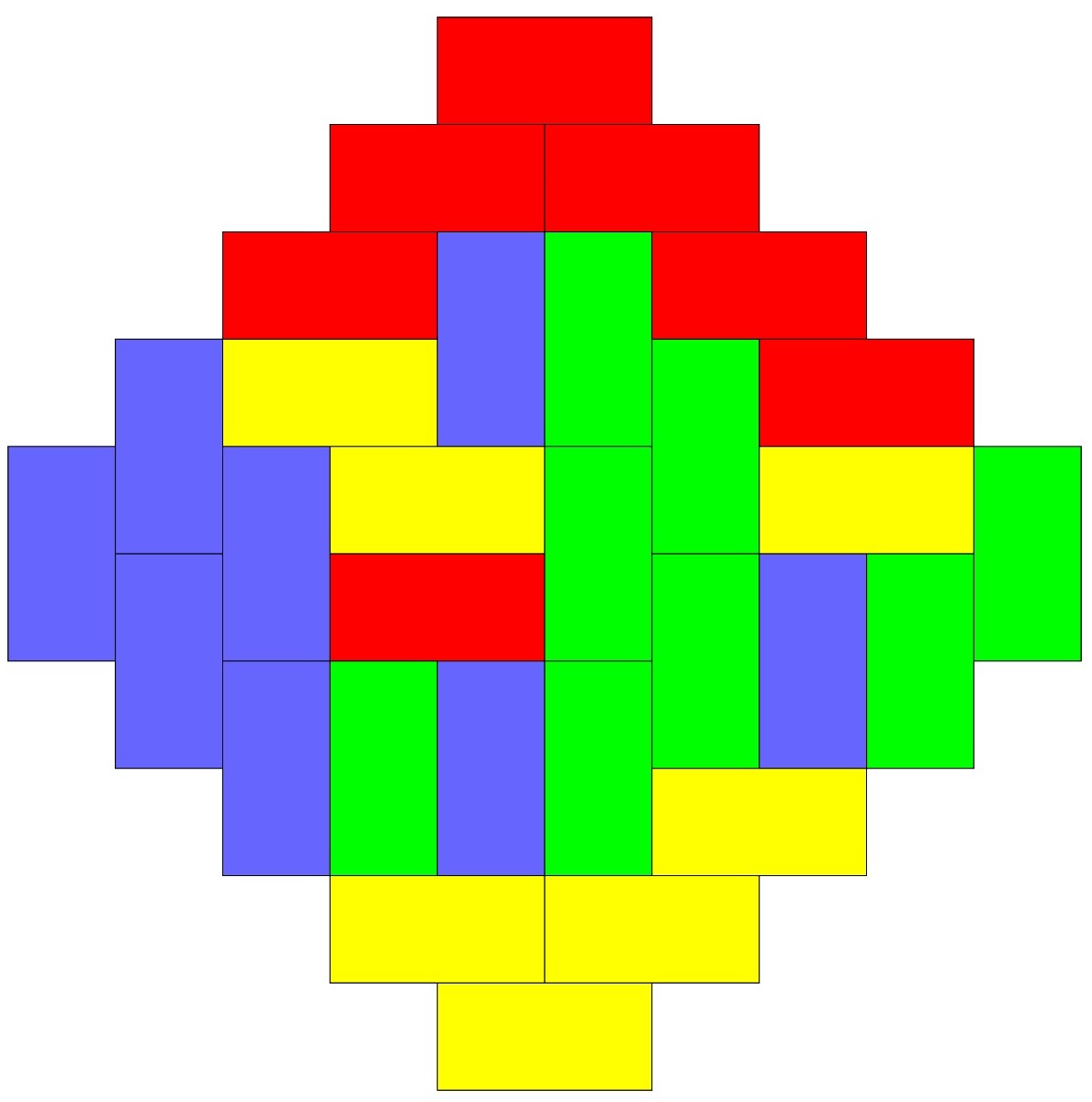}};
\end{tikzpicture}
\begin{tikzpicture}[slave]
\node at (0,0) {\includegraphics[width=6cm]{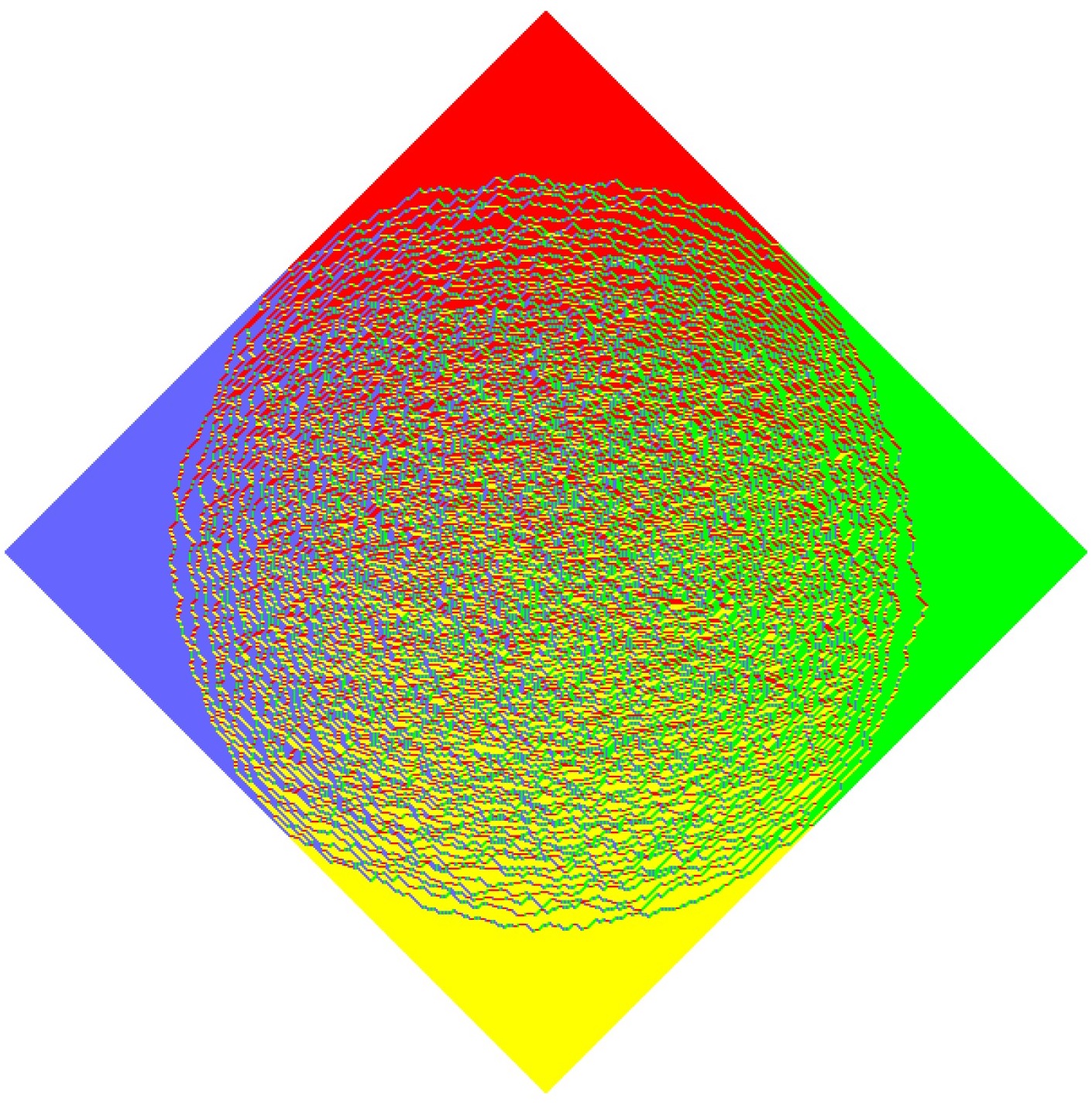}};
\end{tikzpicture}
\end{center}
\caption{\label{fig:Aztec}Left: a domino tiling of $A_{5}$. The north, south, east and west dominoes are shown in red, yellow, green and blue, respectively. Right: a domino tiling of $A_{300}$ chosen uniformly at random.}
\end{figure}

In the present work, we consider domino tilings of reduced, approximately L-shaped, Aztec diamonds, of the form illustrated in Figure \ref{fig:ADreduced2}. 
More precisely, for $m=1,\ldots, {N}$ and $k=1,\ldots, m+1$, $A_N^{m,k}$ is the union of all unit squares $S_{i,j}$ for which $i,j\in\mathbb Z$ and 
\[S_{i,j}\subset\{(x,y)\in\mathbb R^2: |x|+|y|\leq N+1\ \mbox{and}\ y\leq \max\{2m-1-N-x,x-2m-1+N+2k\}\}.\] 
We also consider a slightly modified domain $\tilde{A}_{N}^{m,k}$ as
\begin{align*}
\tilde{A}_{N}^{m,k}  = \begin{cases}
A_{N}^{m+1,k+1}\setminus v^{m,k}, & \mbox{if } m\in \{0,1,\ldots,N-1\}, \\
A_{N}, & \mbox{if } m = N,
\end{cases}
\end{align*}
where $v^{m,k}$ is the vertical segment from $(2m+k-N-1,k-1)$ to $(2m+k-N-1,k)$. 

A domino tiling of $A_N^{m,k}$ or $\tilde A_N^{m,k}$, i.e., a collection of dominoes covering exactly the domain while the interiors of the dominoes are disjoint, can be canonically extended to a domino tiling of $A_N$ by completing it with north dominoes.

Colomo and Pronko \cite{CP2013, CP2015} initiated the study of this model in order to compute the free energy of the fermionic six-vertex model on an L-shaped domain. Domino tilings of a different type of perturbations of Aztec diamonds have been studied in \cite{CHK,DF,Nico}.
Our objective is to understand the large $N$ asymptotics for 
\begin{equation}
F_N^{m,k}(a;\epsilon)
:=
\begin{cases}
\sum_{T\in\mathcal T(A_N^{m,k})}a^{v(T)},&\mbox{ if $\epsilon=1$,}\\
\sum_{T\in\mathcal T(\tilde A_N^{m,k})}a^{v(T)},&\mbox{ if $\epsilon=0$},\end{cases}\qquad 0<a\leq 1,
\end{equation}
when $m$ is proportional to $N$, uniformly for $k\in\{1,\ldots, m\}$. (For $k\leq 0$, the domains $A_N^{m,k}$ and $\tilde A_{N}^{m,k}$ are not tileable \cite[Remark 2.2]{CC2025}, whereas for $k\geq m+1$ the removed region is empty.)

For $a=1$, $F_N^{m,k}(1;\epsilon)$ is equal to the number of domino tilings of the domain $A_N^{m,k}$ if $\epsilon=1$ and of the domain $\tilde A_N^{m,k}$ if $\epsilon=0$. For $0<a<1$, it can be seen as the generating function for the number of domino tilings with a prescribed number of vertical dominoes.
For the full Aztec diamond $A_N$, a celebrated result of Elkies, Kuperberg, Larsen, and Propp \cite{EKLP} states that this generating function is explicitly given by
\begin{equation}\label{eq:ADT}F_N(a):=\sum_{T\in\mathcal T(A_N)}a^{v(T)}=(1+a^2)^{\frac{N(N+1)}{2}},\qquad 0<a\leq 1,\end{equation}
such that 
in particular, the number of domino tilings of $A_N$ is $2^{N(N+1)/2}$. Observe for $k=m+1$ that $F_N(a;\epsilon)=F_N^{m,m+1}(a;\epsilon)$ for any $m\in \{1,\ldots, N\}$ and $\epsilon\in\{0,1\}$.

Another degenerate situation arises when $k=1$:  as shown in \cite[Equation (1.3) and Remark 2.2]{CC2025}, we have the explicit expression
\begin{equation}\label{eq:formulamirror}
F_N^{m,1}(a;\epsilon)=F_{m-\epsilon}(a;\epsilon)F_{N-m}(a;\epsilon)=\left(1+a^2\right)^{\frac{N(N+1)}{2}-m(N+\epsilon-m)},
\end{equation}
which reveals that the weighted number of domino tilings is equal to that of two smaller disjoint Aztec diamonds.
It is worth recalling that the domains $A_N^{m,k}$ and $\tilde A_{N}^{m,k}$ corresponding to non-positive integers $k\leq 0$ cannot be tiled with dominoes, see also \cite[Remark 2.2]{CC2025}.

\begin{figure}
\begin{center}
\begin{tikzpicture}[master]
\node at (0,0) {\includegraphics[width=6cm]{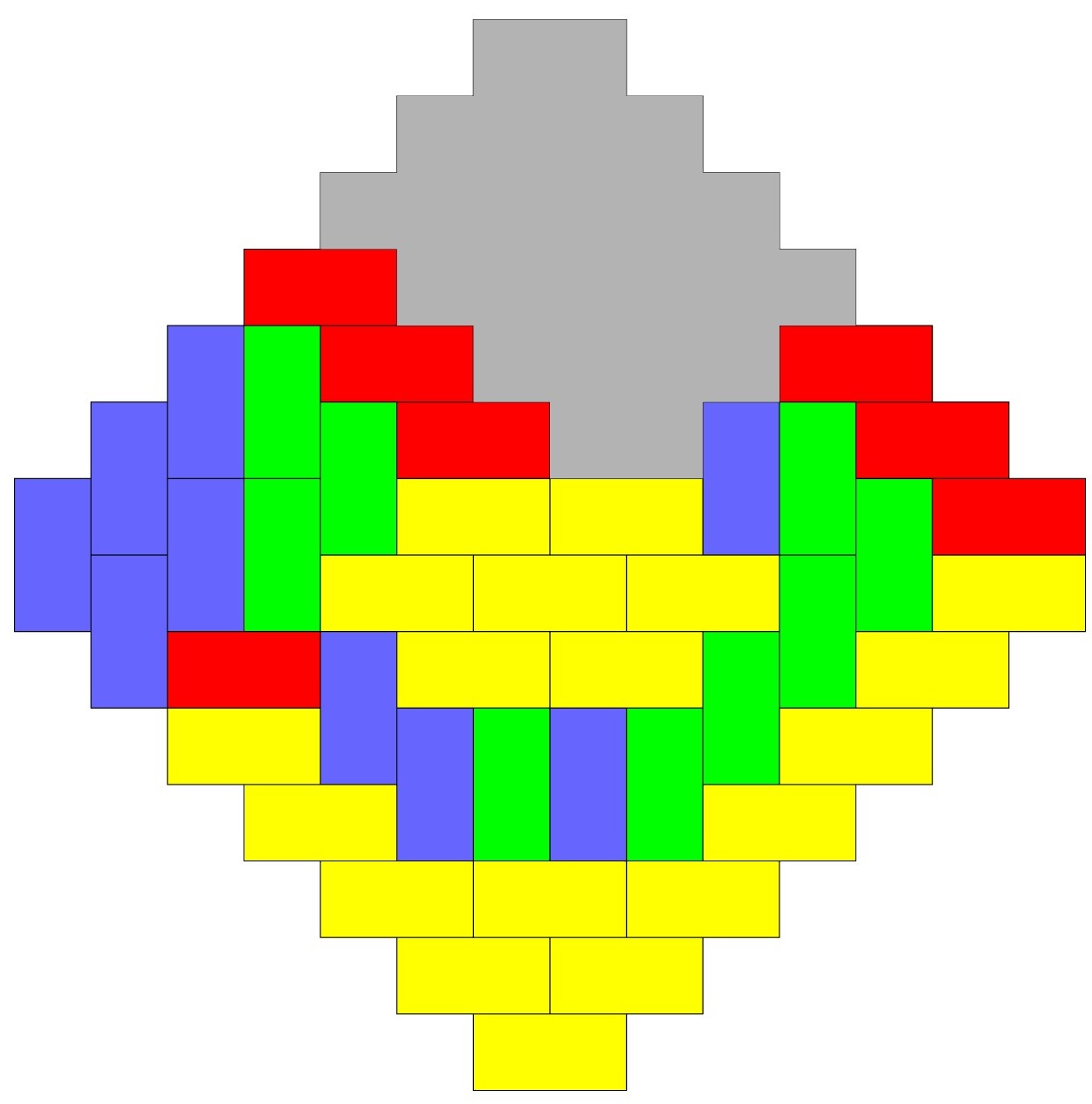}};
\draw[->-=0.9999, -<-=0] (-1.43,2.03)--(-0.43,3.03);
\node[rotate=45] at (-1.1,2.7) {\footnotesize $N+1-m$};
\draw[->-=0.9999, -<-=0] (0.5,3)--(1.8,1.7);
\node[rotate=-45] at (1.1,2.7) {\footnotesize $m-k+1$};
\end{tikzpicture}
\begin{tikzpicture}[slave]
\node at (0,0) {\includegraphics[width=6cm]{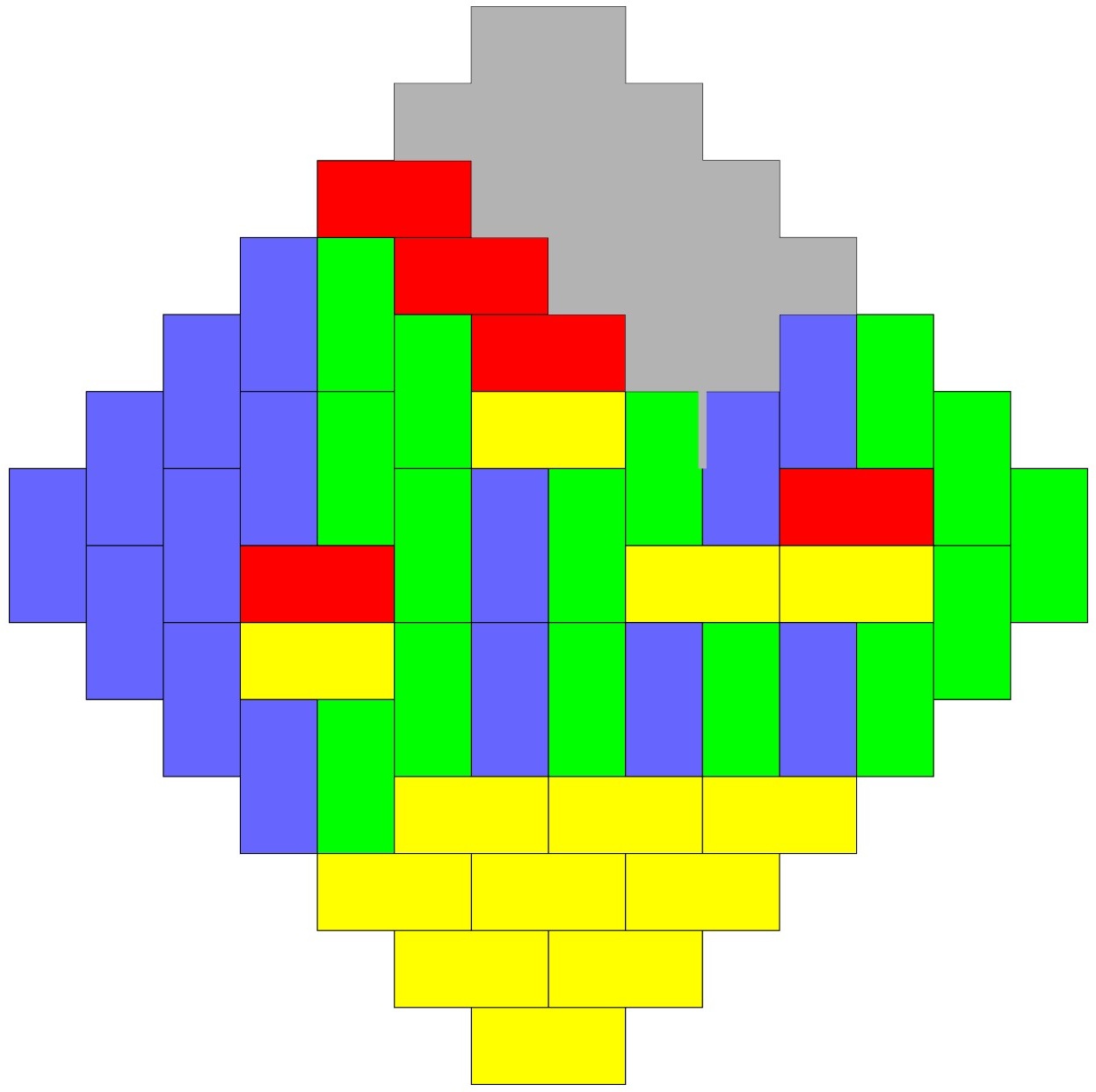}};
\draw[->-=0.9999, -<-=0] (-0.93,2.53)--(-0.43,3.03);
\node[rotate=45] at (-0.93,2.93) {\footnotesize $N-m$};
\draw[->-=0.9999, -<-=0] (0.5,3)--(1.8,1.7);
\node[rotate=-45] at (1.1,2.7) {\footnotesize $m-k+1$};
\end{tikzpicture}
\end{center}
\caption{\label{fig:ADreduced2}Left: a domino tiling of $A_{7}^{5,2}$.  Right: a domino tiling of $\tilde A_{7}^{5,2}$.}
\end{figure}
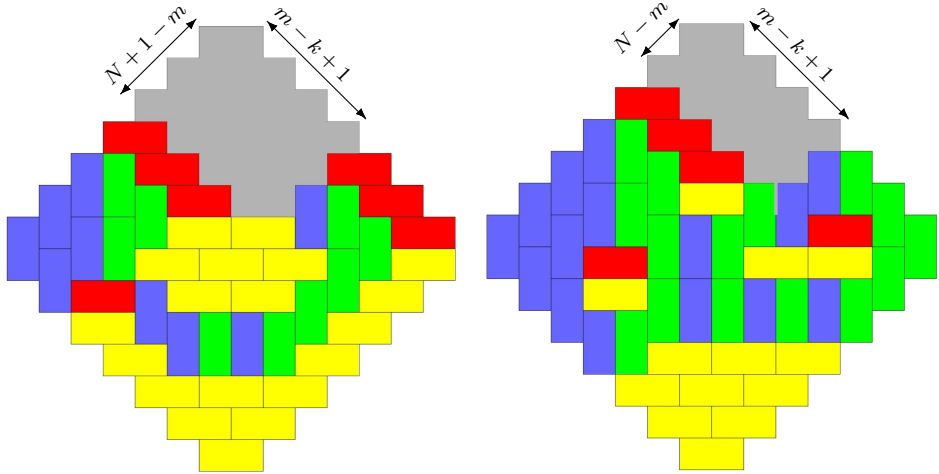

\paragraph{Probability of a large frozen region.}
Let us introduce a probability measure on the set of domino tilings of $A_N$: to a domino tiling $T\in \mathcal T(A_N)$, we assign the probability
\begin{align}\label{prob measure}
\mathbb P(T)=\frac{a^{v(T)}}{(1+a^2)^{\frac{N(N+1)}{2}}},\qquad 0<a\leq 1,
\end{align}
where $v(T)$ is the number of vertical dominoes in $T$. 
For $a=1$, this is the uniform measure, while for $a<1$, domino tilings with fewer vertical dominoes are more likely to occur than tilings with more vertical dominoes.
For large $N$, it is well-known that the vast majority of domino tilings of $A_N$ exhibits regions with qualitatively different behavior: near the four corners of the Aztec diamond, we observe a frozen region where only one type of dominoes is present, while a rough region containing a mixture of all types of dominoes is present in the middle of the Aztec diamond. 
For all values of $a$, the rough and frozen regions are separated by the so-called \textit{arctic curve}. When $a=1$, this curve is a circle (see Figure \ref{fig:Aztec} on the right), while for $a<1$, it is an ellipse \cite{JPS}.

For $\epsilon=1$, the quantity
\begin{equation}P_N^{m,k}(a;\epsilon):=\frac{F_N^{m,k}(a;\epsilon)}{F_N(a)},\qquad
m=1,\ldots, {N},\qquad k=1,\ldots, m+1,\end{equation}
is the probability that a random domino tiling of  the full Aztec diamond $A_N$ contains only north dominoes in the region $A_N\setminus A_N^{m,k}$, or equivalently that it has a frozen region containing $A_N\setminus A_N^{m,k}$.
For $\epsilon=0$, it is the probability that the region $A_N\setminus \tilde A_N^{m,k}$ contains only north dominoes, but with the additional constraint that the segment $v^{m,k}$ cannot be covered by a south domino. Henceforth, we will refer to $A_N\setminus A_N^{m,k}$ (for $\epsilon=1$) or $A_N\setminus \tilde A_N^{m,k}$ (for $\epsilon=0$) as the {\em removed corner}.

Naturally, for large $N$, one expects that $P_N^{m,k}(a;\epsilon)$ will be close to $1$ if the removed corner lies entirely outside the arctic curve. When the removed corner intrudes the arctic curve, one can expect however a dramatic decrease in the number of domino tilings of the domain, such that $P_N^{m,k}(a;\epsilon)$ will decay rapidly as $N\to\infty$. 
In order to capture this transition, we need to scale $m,k$ with $N$. We therefore introduce the variables
\begin{align}\label{def of mu and kappa}
\mu = \frac{m}{N} \in [0,1], \qquad \kappa = \frac{k}{N} \in [0,1].
\end{align}

Given $m\in\mathbb N\cap\big(\frac{a^2}{1+a^2}N,N\big)$, this transition where the removed region crosses the arctic curve will take place for $k\in\mathbb N$ around the value $\kappa_2 N$, with 
\begin{equation}\label{def:kappa2}
\kappa_2=\kappa_{2}(\mu) = \frac{a^{2} (2\mu-1) + 2a \sqrt{\mu(1-\mu)}}{1+a^{2}},\qquad \mu=m/N.
\end{equation}
Observe that for $\mu=\frac{a^2}{1+a^2}$, we have $\kappa_2=\mu$; this corresponds to the point where the arctic curve touches the boundary of the Aztec diamond. Accordingly, for $m\in\mathbb N \cap \big(0,\frac{a^2}{1+a^2}N\big)$, a removed corner would immediately cross the rough region, such that the transition described above does not take place. As $k/N$ decreases further, the number of domino tilings continues to decrease smoothly.
Another phase transition, caused by an almost frozen  region in the south corner, occurs when $k/N$ approaches $0$, or in other words when the removed corner hits the horizontal line dividing the Aztec diamond into halves.

\paragraph{Statement of results.}
 
As explained before, the L-shaped Aztec regions $A_N^{m,k+1}$ and $\tilde A_{N}^{m,k+1}$ can only be tiled by dominoes if $k$ is non-negative. For $k=0$, we have the largest possible removed corner such that our domain is still tileable. 
If $k$ is a fixed non-negative integer (independent of $N$), we will say that the removed corner is {\em almost maximal}. Our first result decribes the large $N$ asymptotics for $F_N^{m,k+1}(a;\epsilon)$ (or equivalently of $P_N^{m,k+1}(a;\epsilon)$) in such a situation.
 
\begin{theorem}{\bf (Almost maximal removed corner).}
\label{thm:birth of a cut}
Let $\epsilon\in\{0,1\}$, $a\in (0,1]$ be fixed.
As in \eqref{def of mu and kappa}, we write $\mu=m/N$ with $m\in\{1,\ldots, N\}$ and $N\in\mathbb N$.
Let $k$ be a non-negative integer. As $N\to\infty$, we have
\begin{align}\label{lol35}
& \log F_N^{m,k+1}(a;\epsilon) = F_{2}(\mu) N^{2} + F_{1}(k,\mu) N  - \frac{k^{2}}{2} \log N + F_{0}(k,\mu) + \bigO(N^{-1/2}),
\end{align}
uniformly
in $\delta \leq \mu \leq 1-\delta$ and in $0 \leq k\leq M$, for any fixed $M\in\mathbb N$ and $\delta>0$,
 where
\begin{align}
F_{2}(\mu) & = \bigg( \frac{1}{2} - \mu + \mu^{2} \bigg) \log(1+a^{2}), \label{def:F1}\\
F_{1}(k,\mu) & = k \varphi_{0}(x_{0}) + \bigg( \frac{1}{2} - \epsilon \mu \bigg) \log(1+a^{2}),\label{def:F2} \\
F_{0}(k,\mu) & = \log \mathcal G(k+1) - \frac{k^{2}}{2}\log\big(  |x_{0}|^{2}\varphi_{0}''(x_{0}) \big) +  \epsilon k\log(1-ax_{0}) - \frac{k}{2}\log (2\pi),\label{def:F4}
\end{align}
where $\mathcal G$ is Barnes' $G$-function and
\begin{align}\label{def:phi0x0}
x_0=\frac{-a-\sqrt{a^2+4\mu(1-\mu)}}{2(1-\mu)}<-a,\quad \varphi_{0}(z) = (1-\mu)\log(1-az)+\mu \log\Big(\frac{z}{z+a}\Big).
\end{align}
\end{theorem}

\begin{remark}
Observe that for $k=0$, we have $F_0({0},\mu)=0$ and $F_1({0},\mu)=( \frac{1}{2} - \epsilon \mu ) \log(1+a^{2})$, such that we have
\[F_N^{m,1}(a;\epsilon)=(1+a^2)^{\left( \frac{1}{2} - \mu + \mu^{2} \right)N^2+\left( \frac{1}{2} - \epsilon \mu \right) N}(1+\bigO(N^{-1/2})),\]
as $N\to\infty$,
which is in agreement with  \eqref{eq:formulamirror}.
For $k\geq 0$ fixed, we can write the above result alternatively as
\begin{multline*}F_N^{m,k+1}(a;\epsilon)=F_{m-\epsilon}(a;\epsilon)F_{N-m}(a;\epsilon)e^{k \varphi_{0}(x_{0}) N } N^{-\frac{k^2}{2}}\\\times \ \mathcal G(k+1) \big(  |x_{0}|^{2}\varphi_{0}''(x_{0}) \big)^{- \frac{k^{2}}{2}} (1-ax_0)^{ \epsilon k}(2\pi)^{- \frac{k}{2}}(1+\bigO(N^{-1/2})),\end{multline*}
as $N\to\infty$. In this expression, we see how the number of tilings, equal to the number of tilings of two disjoint Aztec diamonds for $k=0$, rapidly increases with $k$, since $\varphi_0(x_0)>0$.
The simulation in Figure \ref{fig:tilings large and almost maximal frozen regions} (left) suggests that the increasing number of tilings as $k$ grows, corresponds to the presence of $k$ apparent paths of unfrozen dominoes connecting the two smaller Aztec diamonds. To the best of our knowledge, it is the first time that such a phenomenon of unfreezing, via a small number of paths over a long distance, is observed and quantified, although there are some similarities with a transition in a model of overlapping Aztec diamonds studied in \cite{AJvM, AvM}. It would be interesting to understand not only the number of tilings, but also the locations and correlations of the unfrozen paths.
\end{remark}
\begin{remark}\label{remark:numerical check}
We verified Theorem \ref{thm:birth of a cut}  numerically for some choices of the parameters. In fact, the numerical results suggest that the $\bigO(N^{-1/2})$-term in \eqref{lol35} can be improved to $\bigO(N^{-1})$, see Figure \ref{fig:numerical confirmation} (left and middle).
\end{remark}

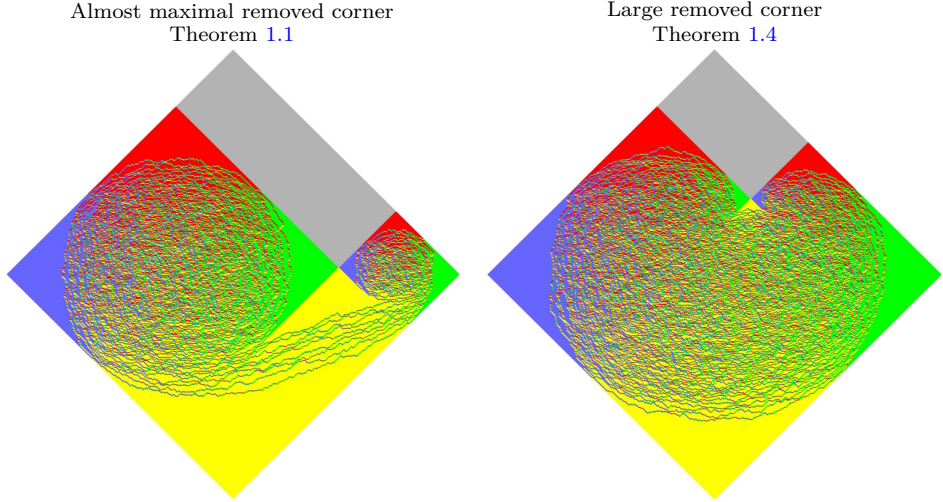
\begin{figure}
\begin{center}
\begin{tikzpicture}[master]
\node at (0,0) {\includegraphics[width=6cm]{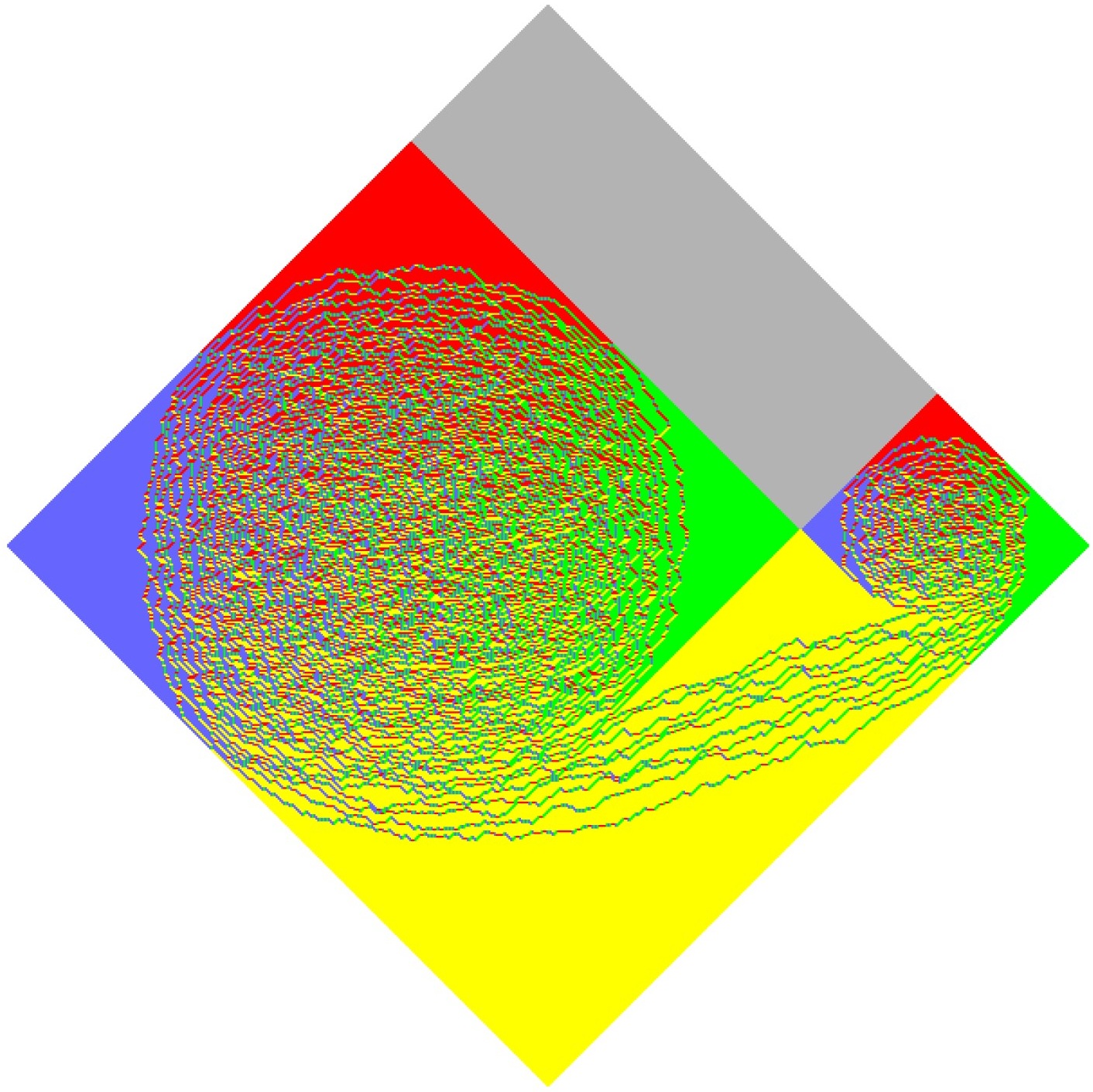}};
\node at (0,3.5) {\footnotesize Almost maximal {removed corner}};
\node at (0,3.2) {\footnotesize Theorem \ref{thm:birth of a cut}};
\end{tikzpicture}
\begin{tikzpicture}[slave]
\node at (0,0) {\includegraphics[width=6cm]{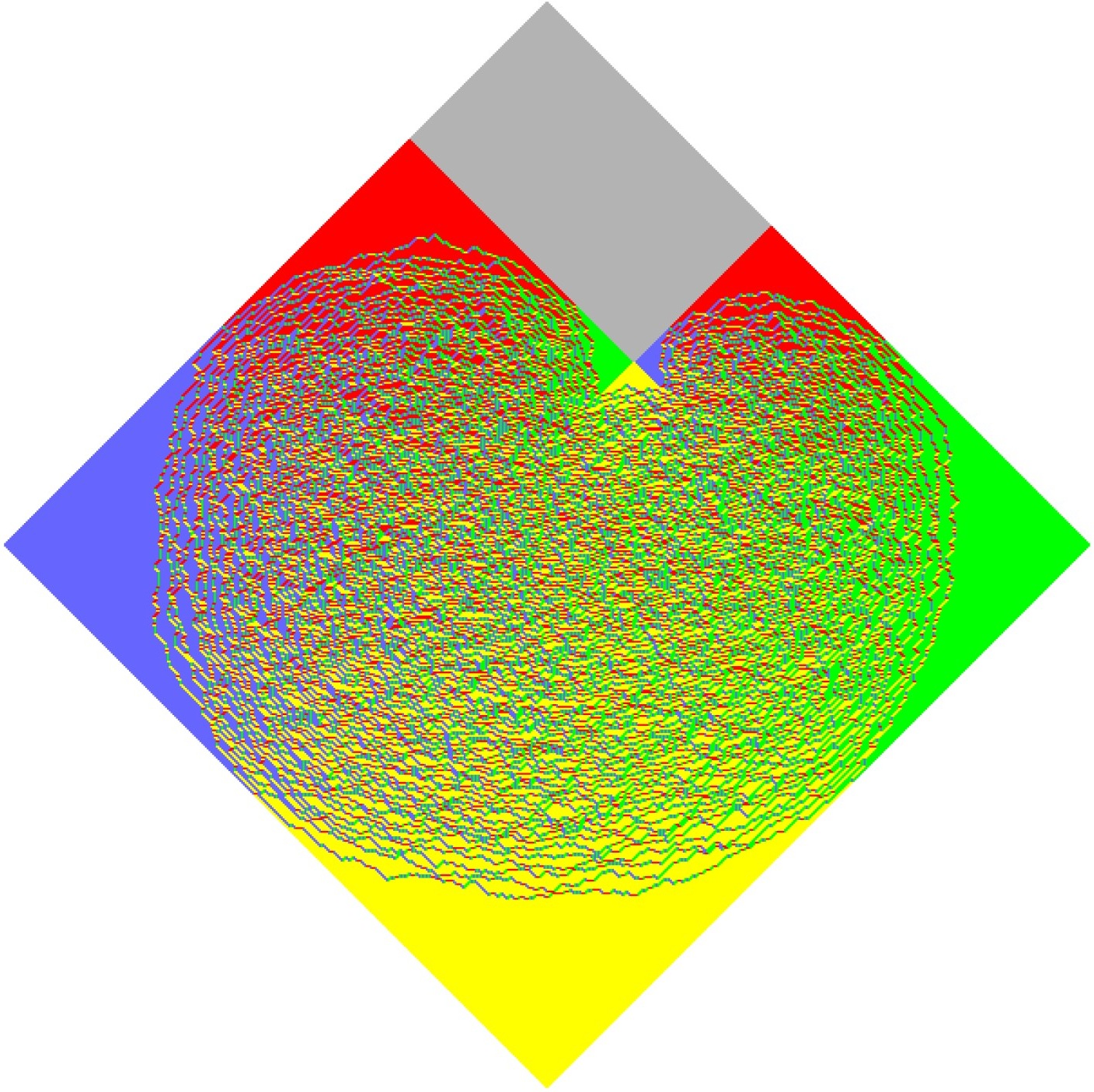}};
\node at (0,3.5) {\footnotesize Large {removed corner}};
\node at (0,3.2) {\footnotesize Theorem \ref{thm:rough}};
\end{tikzpicture}
\end{center}
\caption{\label{fig:tilings large and almost maximal frozen regions}Left: a domino tiling of $A_{300}^{225,10}$ taken uniformly at random. Right: a domino tiling of $A_{300}^{225,102}$ taken uniformly at random. }
\end{figure}

Our next result is concerned with a {removed} corner which is slightly smaller, but still has a large overlap with the rough region. {(This case is the most involved out of the four regimes considered in this paper.)} The number of tilings is then bigger than in the regime described above, but still much smaller than the number of tilings of the full Aztec diamond. This corresponds to $\mu=m/N\in (0,1)$ and $\kappa=k/N\in (0,\kappa_2)$, with $\kappa_2=\kappa_2(\mu)\in (0,\mu)$ given by \eqref{def:kappa2}. In this regime, Colomo and Pronko established the leading order term of order $N^2$ (the free energy) in the large $N$ asymptotic expansion of $\log F_N^{m,k+1}(a;\epsilon)$ \cite{CP2013, CP2015}. Moreover, when $\epsilon=0$ and the removed region satisfies $N-m=m-k+1$ (that is, the removed region has the same number of corners along each side; see Figure \ref{fig:ADreduced2}, right), the full expansion of $\log F_N^{m,2m-N}(a;0)$ was previously known \cite{KP2016}, see also Remark \ref{remark:comparison with Kitaev Pronko} below. We confirm and strengthen their result by computing the subleading terms proportional to $N$ and to $\log N$ as well as the term of order $1$, and by extending the uniformity range.

{The constants $C_2(\kappa,\mu), C_1(\kappa,\mu),C_0(\kappa,\mu)$ appearing in Theorem \ref{thm:rough} below are very involved, but in essence they only depend on $\kappa,\mu$ and on the unique $z_0=z_0(\kappa;\mu)$ in the complex upper half plane which satisfies
\begin{align}\label{def of z0 intro-0}
\begin{cases}
\ds \frac{1-\mu}{|z_{0}-\frac{1}{a}|} + \frac{\mu-\kappa}{|z_{0}|} - \frac{\mu}{|z_{0}+a|} = 0, \\
\ds \kappa+1-\mu - \frac{a\mu}{|z_{0}+a|} - \frac{1-\mu}{a |z_{0}-\frac{1}{a}|} = 0.
\end{cases}
\end{align}
This quantity is analyzed in Section \ref{sec:large} and Proposition \ref{prop:z0} below; in particular we will show that $z_0$ can be expressed in terms of a root of a degree four polynomial equation.}

\begin{figure}
\begin{center}
\begin{tikzpicture}[master]
\node at (0,0) {\includegraphics[width=4.5cm]{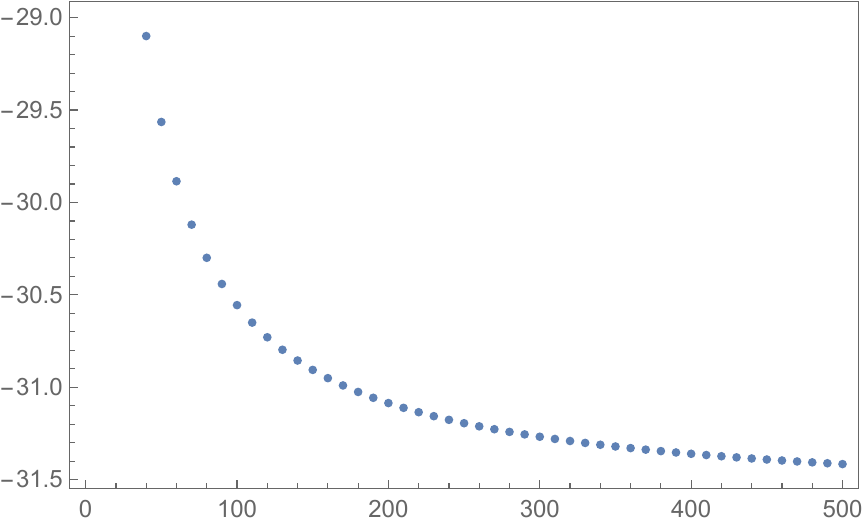}};
\node at (0,2.3) {\footnotesize Numerical confirmation of};
\node at (0,2.0) {\footnotesize Theorem \ref{thm:birth of a cut}};
\end{tikzpicture}
\begin{tikzpicture}[slave]
\node at (0,0) {\includegraphics[width=4.5cm]{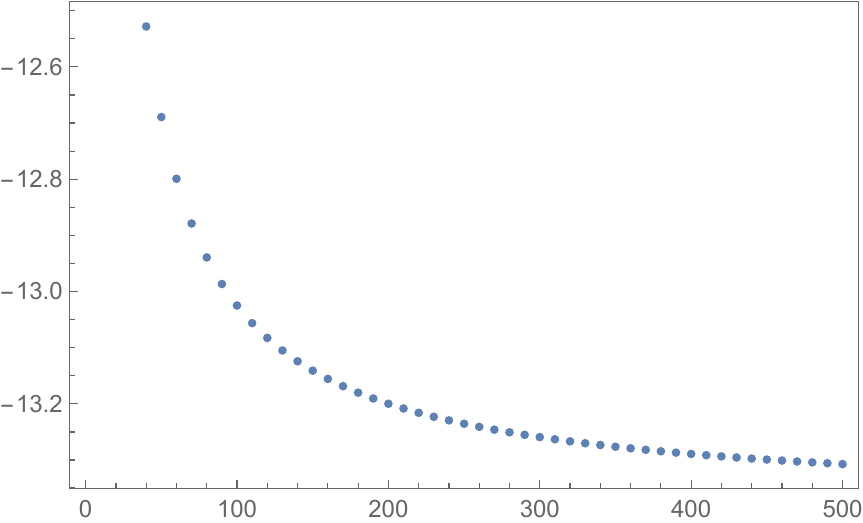}};
\node at (0,2.3) {\footnotesize Numerical confirmation of};
\node at (0,2.0) {\footnotesize Theorem \ref{thm:birth of a cut}};
\end{tikzpicture}
\begin{tikzpicture}[slave]
\node at (0,0) {\includegraphics[width=4.5cm]{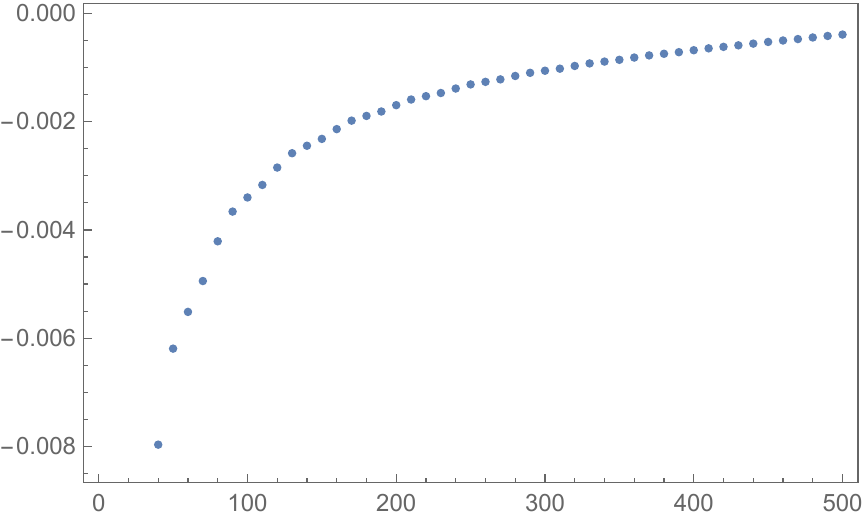}};
\node at (0,2.3) {\footnotesize Numerical confirmation of};
\node at (0,2.0) {\footnotesize Theorem \ref{thm:rough}};
\end{tikzpicture}
\end{center}
\caption{\label{fig:numerical confirmation} Let $\log \mathcal{F}_{N}^{1}$ be the right-hand side of \eqref{lol35} without the $\bigO$-term. The left and middle figures represent $N \mapsto N(\log F_N^{m,k+1}(a;\epsilon)-\log \mathcal{F}_{N}^{1})$ with $a=0.7845$, $k=3$, $m=\{0.7N\}$, $\mu = \mu_{N} = m/N\approx 0.7$, and with $\epsilon=1$ (left) and $\epsilon=0$ (middle). Here $\{x\}$ denotes the nearest integer to $x$ (in case $x$ is exactly halfway between two integers, then $\{x\}$ is the nearest even integer to $x$). \newline
Right: Let $\log \mathcal{F}_{N}^{2}$ be the right-hand side of \eqref{asymp main thm} without the $\bigO$-term and without the $o(1)$-term. The right figure represents $N \mapsto \log F_N^{m,k+1}(a;\epsilon)-\log \mathcal{F}_{N}^{2}$ with $a=0.7845$, $m=\{0.7N\}$, $\mu = \mu_{N} = m/N\approx 0.7$, $k=\{0.25N\}$, $\kappa=\kappa_{N} = k/N \approx 0.25$, and $\epsilon=1$.
}
\end{figure}
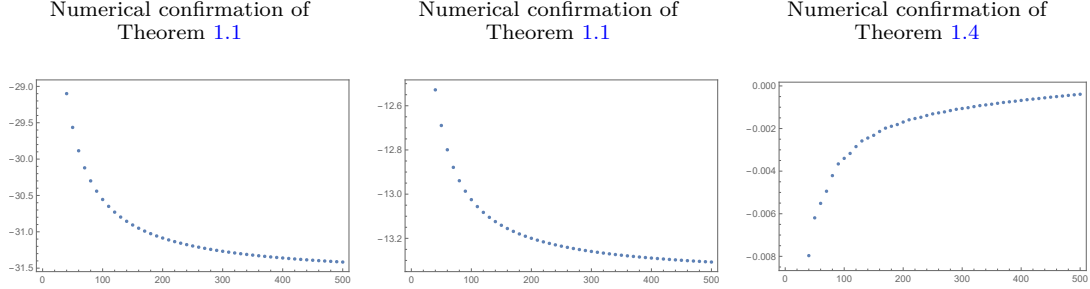
\begin{theorem}{\bf (Large removed corner).}
\label{thm:rough}
Let $\epsilon\in\{0,1\}$, $a\in (0,1]$ be fixed.
We write $\mu=m/N$ and $\kappa_+=\frac{2k+1}{2N}$ with $k\leq m\in\{1,\ldots, N\}$ and $N\in\mathbb N$.
There exists $M>0$ such that for any $\delta>0$, we have\begin{multline}\label{asymp main thm}
\log F_N^{m,k+1}(a;\epsilon) =  C_{2}(\kappa_+,\mu)N^{2}+C_{1}(\kappa_+,\mu) N -\frac{1}{12} \log N + C_{0}(\kappa_+,\mu)\\
  +  \bigO\bigg({\frac{1}{N(\kappa_2-\kappa)^2}}\bigg) +o(1) ,
\end{multline}
as $N\to\infty$, uniformly for
${\delta}\leq \mu \leq 1-\delta$ and for
 ${\frac{M}{N}\leq \kappa} \leq {\min\{\kappa_2 -MN^{-2/3},\mu-\delta\}}$.
The functions $C_2(\kappa_+,\mu)$, $C_1(\kappa_+,\mu)$ and $C_0(\kappa_+,\mu)$ are smooth for $\mu\in(0,1)$ and $\kappa_+\in\left(0,\kappa_2\right)$. {Their explicit forms are given in equations \eqref{def:C2}, \eqref{def:C1} and \eqref{def:C0} below, and depend on the functions $G(\kappa,\mu), H(\kappa,\mu), F(\kappa,\mu)$ given by \eqref{def:G}, \eqref{def:H} and \eqref{def:F}, respectively.}

\medskip As $\kappa_+\to 0$, they satisfy the asymptotics
\begin{subequations}\label{asympCj0}
\begin{align}
C_{2}(\kappa_+,\mu) &=\left(\frac{1}{2}-\mu+\mu^2\right)\log(1+a^2)
+ \kappa_+\varphi_0(x_0)+\frac{\kappa_+^2}{2}\log\kappa_+ \nonumber \\ 
& \quad -\frac{\kappa_+^2}{4}\left(3+2\log(|x_0|^2{\varphi_{0}''}(x_0))\right) +\bigO(\kappa_+^3),\\
C_{1}(\kappa_+,\mu) &= -\frac{-\varphi_0(x_0)}{2}+\left(\frac{1}{2}-\epsilon\mu\right)\log(1+a^2)-\frac{\kappa_+}{2}\log\kappa_+ \nonumber \\
& \quad +\frac{\kappa_+}{2}\left(1+\log(|x_0|^2{\varphi_{0}''}(x_{0}))+2\epsilon \log(1-ax_0)\right) {+\bigO(\kappa_{+}^{2})},\\
C_{0}(\kappa_+,\mu) &=  \frac{1}{24}\log\kappa_+-\frac{3}{24}\log(|x_0|^2{\varphi_{0}''}(x_0)) -\frac{\epsilon}{2}\log(1-ax_0)+\zeta'(-1)+\bigO(\kappa_+),
\end{align}
\end{subequations}
{uniformly in 
${\delta}\leq \mu \leq 1-\delta$}, where $x_{0}$ and $\varphi_{0}$ are defined in \eqref{def:phi0x0}, and $\zeta$ is Riemann's zeta function.
\end{theorem}

\begin{remark}
The condition $k \leq \min\{N\kappa_2 -MN^{1/3},N(\mu-\delta)\}$ means that the corner of the removed region lies inside the arctic ellipse, at least at a distance $MN^{1/3}$ away from the ellipse and at least at distance $\delta N$ away from the point where the ellipse touches the border of the Aztec diamond. Note also that the error terms in \eqref{asymp main thm} are small if $k\to\infty$ and $N^{-1/2}(N\kappa_2-k)\to \infty$.
\end{remark}

\begin{remark}
Barnes' $G$-function satisfies the asymptotic expansion
\begin{equation}\label{asBarnes}
\log \mathcal G(k+1)=\frac{k^2}{2}\log\left(k+\frac{1}{2}\right)-\frac{3k^2}{4}-\frac{k}{4}+\frac{k}{2}\log(2\pi)-\frac{1}{12}\log\left(k+\frac{1}{2}\right)+\frac{1}{16}+\zeta'(-1)+{\bigO\Big(\frac{1}{k}\Big)}
\end{equation}
as $k\to\infty$. Substituting this in Theorem \ref{thm:birth of a cut}, we obtain that the asymptotics from Theorem \ref{thm:birth of a cut} are compatible as $k\to\infty$ with those from
Theorem \ref{thm:rough}; more precisely, they match with the expansion \eqref{asymp main thm} with $\kappa_+=\frac{2k+1}{2N}$ and in which we substitute the small $\kappa_+$ asymptotic expansions \eqref{asympCj0}.
\end{remark}
\begin{remark}\label{remark:numerical check 2}
Figure \ref{fig:numerical confirmation} (right) supports the validity of Theorem \ref{thm:rough} with $\epsilon=1$. We also verified Theorem \ref{thm:rough} numerically for $\epsilon=0$, although we decided not to include a graph for that case.
\end{remark}

\begin{remark}
The simulation in Figure \ref{fig:tilings large and almost maximal frozen regions} (right) indicates several interesting asymptotic features, in particular it shows an arctic curve separating frozen and rough regions with two cusp points.
This arctic curve has been computed explicitly in \cite{CPS2019}.
Finer statistics like local correlation functions have not been analyzed so far. Like in the full Aztec diamond, one expects behavior described in terms of the Airy process near the arctic curve \cite{Joh2}, except near the cusp points. There, one would rather expect a behavior related to the cusp-Airy process \cite{DJM, Petrov}.
\end{remark}
\begin{remark}\label{remark:comparison with Kitaev Pronko}
Theorem \ref{thm:rough} was previously known in the case where $\epsilon=0$ and the removed region has the same number of corners along each side, that is, when $N-m=m-k+1$, see also Figure \ref{fig:ADreduced2} (right). In this particular situation, the expansion can be considerably simplified, as shown by Kitaev and Pronko in \cite{KP2016}. Indeed, by \cite[Theorem 1.2]{KP2016} (the parameters $\alpha$, $s$ and $r$ in \cite{KP2016} correspond here to $\frac{a^{2}}{1+a^{2}}$, $1-\mu$ and $\mu$, respectively), we have
\begin{multline}\label{asymp main thm Pronko}
\log \frac{F_N^{m,2m-N}(a;0)}{F_{N}(a)} =  \bigg( \frac{(1-2\mu)^{2}}{2(1-\mu)^{2}} \log \frac{1-v}{1-u} + \frac{1}{2(1-\mu)^{2}} \log \frac{1+v}{1+u} - \log \frac{v}{u} \bigg)(1-\mu)^{2}N^{2} \\
- \frac{1}{12}\log N - \frac{1}{12}\log(1-\mu)+\frac{1}{8}\log \frac{(1-u^{2})v^{2}}{v^{2}-u^{2}} - \frac{1}{12}\log \frac{1-v^{2}}{2} + \zeta'(-1) + \bigO(N^{-2}),
\end{multline}
as $N\to + \infty$ uniformly for $v$ in compact subsets of $(u,1)$, where
\begin{align*}
u = 1+2a^{2}-2a\sqrt{1+a^{2}}, \qquad v = \frac{1-\mu}{\mu}.
\end{align*}
Moreover, the next two terms were also given explicitly in \cite[Theorem 1.2]{KP2016}. It is worth noting that in \eqref{asymp main thm Pronko}, there is no term of order $N$. We have checked numerically that Theorem \ref{thm:rough} is consistent with \eqref{asymp main thm Pronko}.
\end{remark}

For $\mu>\frac{a^2}{1+a^2}$ and $k$ near $N\kappa_2$, ${P_N^{m,k}(a;\epsilon)=}F_N^{m,k}(a;\epsilon)/F_N(a)$ is the probability to have a frozen region that extends to a point near the arctic curve. It is well-known that the fluctuations near the arctic curve are described by the Airy process \cite{Joh2} and the Tracy-Widom distribution. 
We recall \cite{TW} that the Tracy-Widom distribution $F^{\rm TW}(s)$ is given by
\begin{equation}\label{def:TW}
F^{\rm TW}(s)=\exp\left(-\int_s^{+\infty}(x-s)u(x)^2 dx\right),
\end{equation}
where $u(x)$ is the Hastings-McLeod solution of the Painlev\'e II
equation, i.e.,
\[u''(x)=xu(x)+2u(x)^3,\qquad u(x)\sim {\rm Ai}(x)\mbox{ as $x\to +\infty$,}\]
where ${\rm Ai}$ is the Airy function. As $s\to +\infty$, $\log F^{\rm TW}(s)$ converges rapidly (faster than exponentially) to $0$,
\begin{align}\label{eq:TW easy tail}
\log F^{\rm TW}(s) = \frac{e^{-\frac{4}{3}s^{3/2}}}{16\pi s^{3/2}}(1+o(1)), \qquad \mbox{as } s\to + \infty,
\end{align}
while its tail as $s\to -\infty$ is described by \cite{BBdF, DIK}
\begin{equation}\label{eq:TWtail}
\log F^{\rm TW}(s)=-\frac{|s|^3}{12}-\frac{1}{8}\log|s|+\frac{1}{24}\log 2+\zeta'(-1)+\bigO(|s|^{-3/2}), \qquad \mbox{as } s\to - \infty,
\end{equation}
where $\zeta$ is the Riemann zeta function.

\begin{figure}
\begin{center}
\begin{tikzpicture}[master]
\node at (0,0) {\includegraphics[width=6cm]{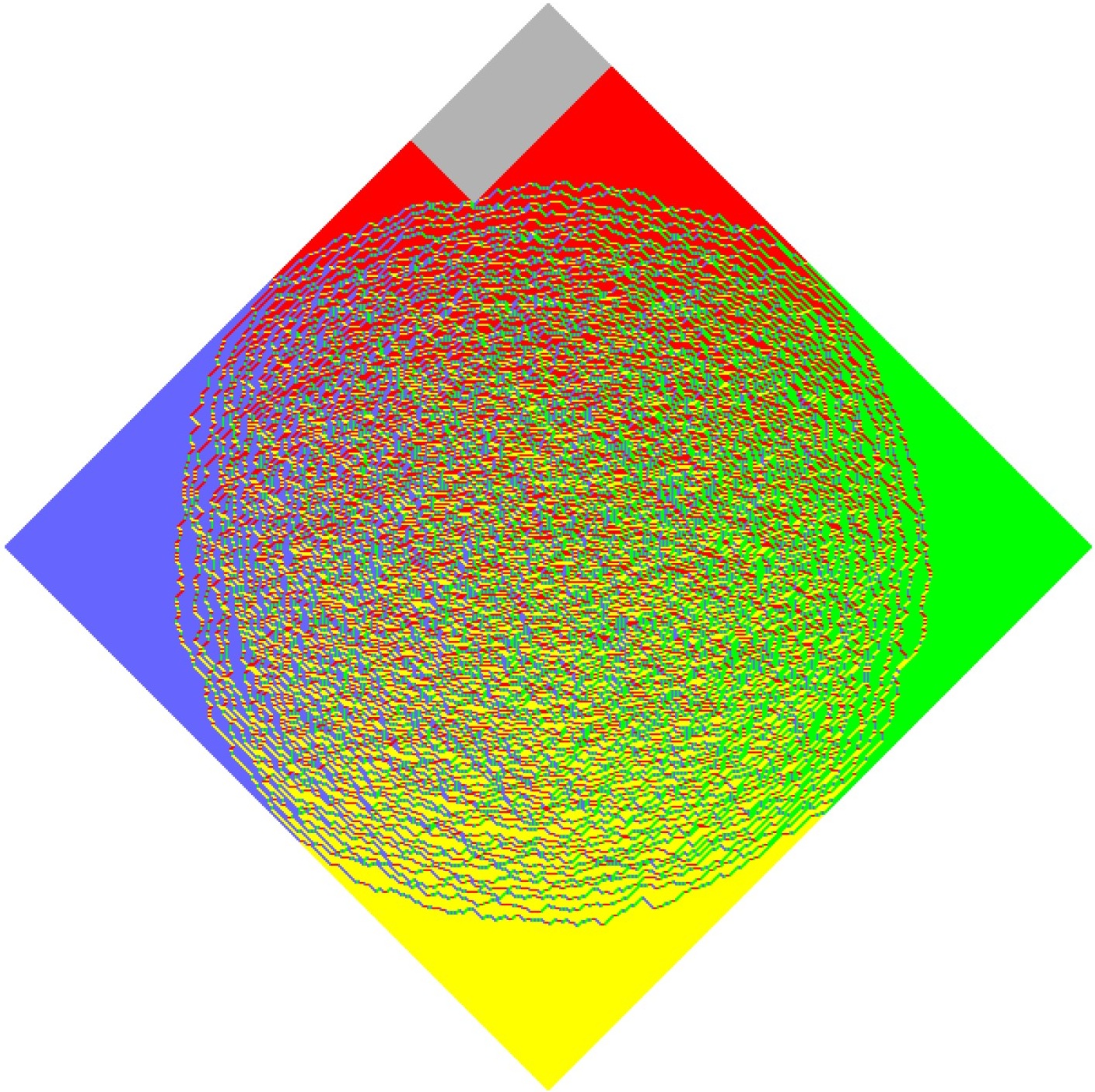}};
\node at (0,3.5) {\footnotesize Critical removed corner};
\node at (0,3.2) {\footnotesize Theorem \ref{thm:boundary}};
\end{tikzpicture}
\begin{tikzpicture}[slave]
\node at (0,0) {\includegraphics[width=6cm]{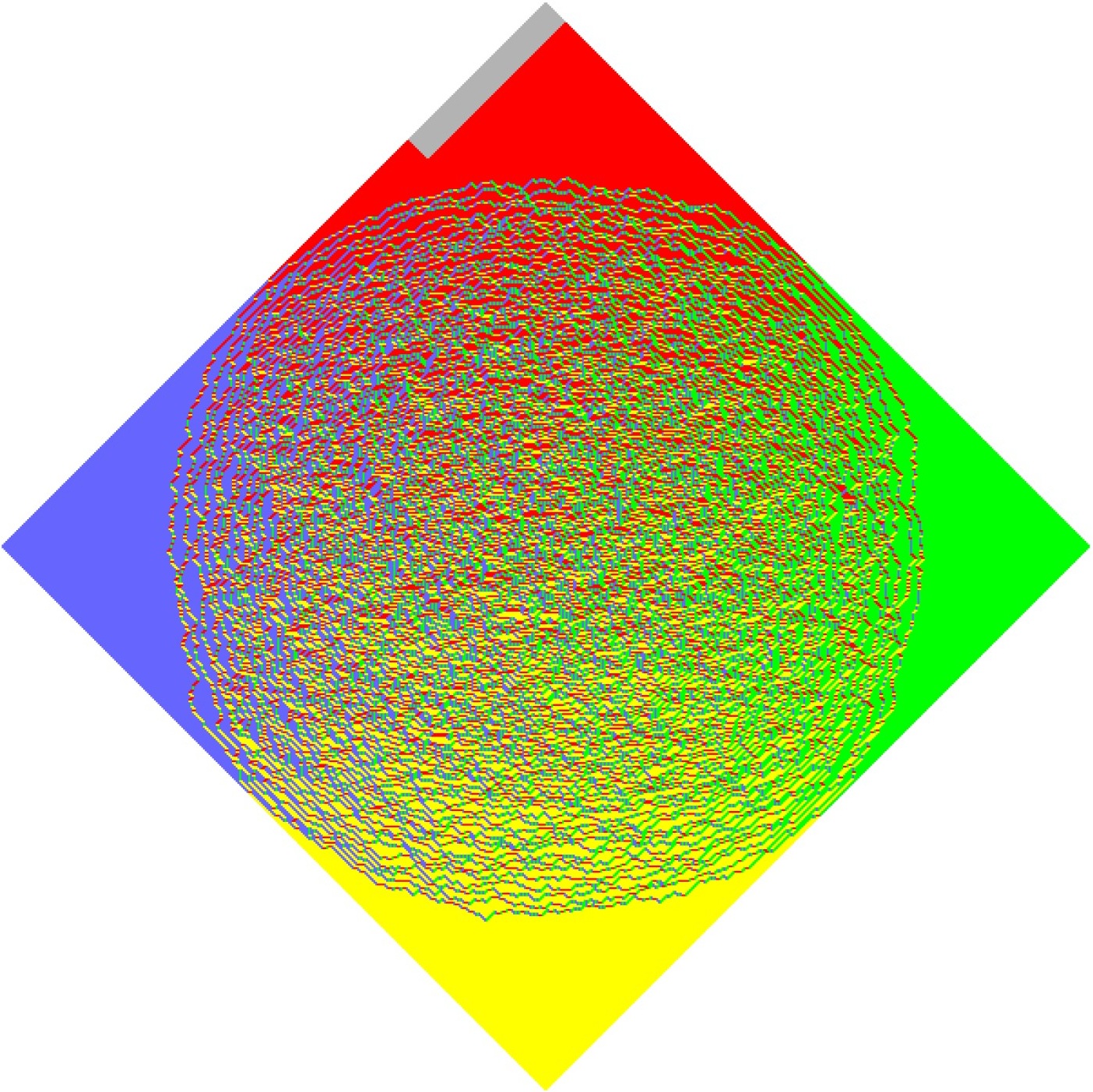}};
\node at (0,3.5) {\footnotesize Small removed corner};
\node at (0,3.2) {\footnotesize Theorem \ref{thm:frozen}};
\end{tikzpicture}
\end{center}
\caption{\label{fig:tilings small and critical frozen regions}Left: a domino tiling of $A_{300}^{225,191}$ taken uniformly at random.  Right: a domino tiling of $A_{300}^{225,215}$ taken uniformly at random.}
\end{figure}
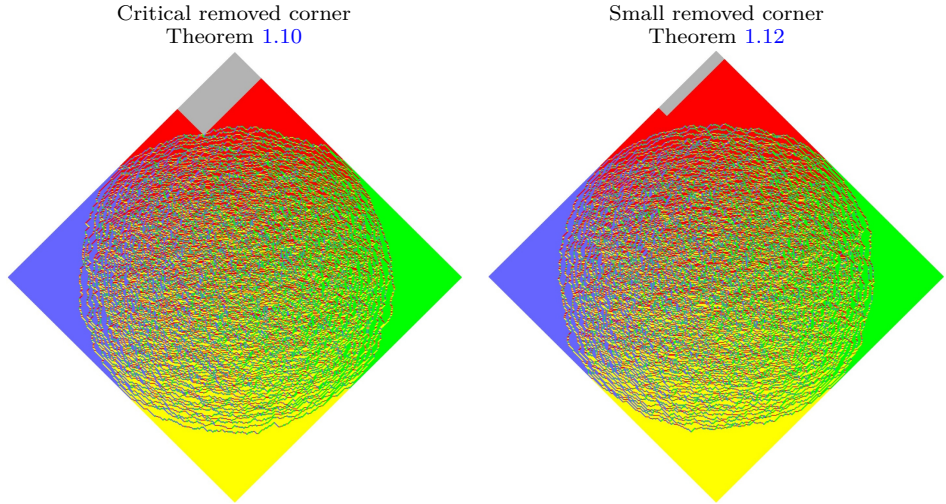

The next result is an improvement of existing results \cite{Joh2} describing convergence to the Tracy-Widom distribution. Compared to \cite{Joh2}, we obtain uniform convergence for $(m,k)$ in a slightly bigger neighborhood of the arctic curve, and with more control over the error term.
\begin{theorem}{\bf (Critical removed corner).}
\label{thm:boundary}
Let $\epsilon\in\{0,1\}$, $a\in (0,1]$ be fixed. For any $\delta>0$ and $c<2/5$, we have 
\begin{align*}
F_N^{m,k}(a;\epsilon)=F_N(a)F^{\rm TW}\left({\frac{N^{2/3}(\kappa-\kappa_2)}{(c^{*})^{1/3}}}\right)\left(1+{\mathcal{E}_{N}}\right),\qquad N\to\infty,
\end{align*}
uniformly for $\frac{a^2}{1+a^2}(1+\delta)N\leq m\leq (1-\delta)N$ and for $N\kappa_2-N^{c}\leq k\leq N(\kappa_2+\delta)$. Here $c^*=c^*(\mu,a)>0$, with $\mu=\frac{m}{N}$, is given by
\begin{align*}
{c^{*}}=\mu{-}\kappa_2 {+}\frac{a^3(1-\mu)}{(1/{x}^*-a)^3} {-} \frac{\mu}{(1+a/{x}^*)^3}>0,\qquad {x}^*=  \frac{a-(1+a^{2})\sqrt{\mu(1-\mu)}}{(1+a^{2})\mu-1}\in(0,1/a),
\end{align*}
and $\mathcal{E}_{N}$ satisfies the following asymptotics as $N\to \infty$:
\begin{align*}
\mathcal{E}_{N} = \begin{cases}
\bigO(N^{-1/3}), & \mbox{if } {N^{\frac{2}{3}}(\kappa-\kappa_{2})} \mbox{ is bounded}, \\
\bigO\Big({\frac{e^{-c'N(\kappa-\kappa_{2})^{3/2}}}{N^{1/2}(\kappa-\kappa_{2})^{1/4}}}\Big), & \mbox{if } {N^{\frac{2}{3}}(\kappa-\kappa_{2})} \to + \infty, \\
\bigO({N^{3}(\kappa-\kappa_{2})^{5}}), & \mbox{if } {N^{\frac{2}{3}}(\kappa-\kappa_{2})} \to - \infty,
\end{cases}
\end{align*}
for any {fixed} $c' \in (0,\frac{2}{3}(c^{*})^{-1/2})$.
In particular, by \eqref{eq:TWtail}, as $N\to \infty$ such that $M:={N^{\frac{2}{3}}(\kappa_{2}-\kappa)} \to + \infty$ and {$M/N^{1/15}\to 0$}, 
\begin{align}\label{lol22}
\log F_N^{m,k}(a;\epsilon)=\log F_N(a)-\frac{M^3}{12c^{*}}-\frac{1}{8}\log M+\frac{1}{24}\log (2c^{*})+\zeta'(-1)+\bigO\bigg(\frac{1}{M^{3/2}} + \frac{M^{2}}{N^{1/3}}\bigg).
\end{align}
\end{theorem}

\begin{remark}
Note that there is an interval $k\in \left[N\kappa_2-N^{c_1},N\kappa_2-N^{c_2})\right]$, for any $\frac{5}{15}<c_2<c_1<\frac{6}{15}$, in which the results of Theorem \ref{thm:rough} and Theorem \ref{thm:boundary} both hold. The compatibility of both results in that case suggest the following asymptotics as $\kappa_+\to \kappa_2$ for $C_2$ and $C_1$,
\begin{subequations}\label{asympCjkappa2}
\begin{align}
C_{2}(\kappa_+,\mu) &= 
\frac{1}{2}\log(1+a^{2})-\frac{1}{12c^*}(\kappa_2-\kappa_+)^3+\bigO((\kappa_2-\kappa_+)^4),\\
 C_{1}(\kappa_+,\mu) &=
\frac{1}{2} \log(1+a^{2}) 
+\bigO((\kappa_2-\kappa_+)^{{2}}).  
\end{align}
\end{subequations}
We will prove these asymptotics in Subsection \ref{subsection:end}.
\end{remark}

Finally, for $\mu>\frac{a^2}{1+a^2}(1+\delta)$ and ${\kappa}>\kappa_2+\delta$, the removed {corner} lies entirely north of the arctic curve. Hence, we expect $P_N^{m,k}(a;\epsilon)$ to be very close to $1$. The next theorem confirms this expectation.
\begin{theorem}{\bf (Small removed corner).}
\label{thm:frozen}
Let $\epsilon\in\{0,1\}$, $a\in (0,1]$ be fixed. There exists $c>0$ such that, for any $\delta>0$,
\begin{align*}
F_N^{m,k}(a;\epsilon)=\big(1+O(e^{-c{N}(\kappa-\kappa_2)^{3/2}})\big)F_N(a),\qquad \mbox{as } N\to +\infty,
\end{align*}
uniformly for $\frac{a^2}{1+a^2}(1+\delta)\leq {\mu} \leq 1-\delta$ and for $\kappa_2+\delta \leq {\kappa}\leq {\mu}$.
\end{theorem}

\paragraph{Outline.}
The starting point of our asymptotic analysis is a characterization of the quantity
$\frac{F_{N}^{m,k+1}(a;\epsilon)}{F_{N}^{m,k}(a;\epsilon)}$
in terms of a $2\times 2$ Riemann-Hilbert (RH) problem established in \cite{CC2025} using a 
discrete version of Fourier techniques developed in \cite{BertolaCafasso} combined with classical Its-Izergin-Korepin-Slavnov theory \cite{IIKS}. We recall this characterization in Section \ref{section:RHP for Y}.
In Section \ref{sec:small frozen region}, we perform the large $N$ asymptotic analysis of the RH problem in the situation of a small removed corner. This asymptotic analysis requires the inspection of a phase function and a suitable deformation of jump contours, but no construction of parametrices, and it will allow us to prove Theorem \ref{thm:frozen}.
In Section \ref{sec:critical}, we continue with the asymptotic RH analysis in the critical regime where the removed corner hits the arctic curve. Here, we will need to construct a local parametrix related to the Painlev\'e II equation, which will lie at the basis of the proof of the Tracy-Widom asymptotics in Theorem \ref{thm:boundary}.
In Section \ref{sec:max}, we will complete the asymptotic analysis for an almost maximal removed corner and prove Theorem \ref{thm:birth of a cut}. Here we will need to construct a local parametrix in terms of Hermite polynomials. 
In Section \ref{sec:large}, we finally analyze the RH problem asymptotically for a large removed corner. In that case, we will need to construct a suitable $g$-function and at a later stage local Airy parametrices and a global parametrix. This will lead us to the proof of Theorem \ref{thm:rough}.

\section{Ratio identity and RH problem}\label{section:RHP for Y}

The starting point of our analysis is a characterization of $F_N^{m,k}(a;\epsilon)$ in terms of a $2\times 2$ matrix RH problem, obtained recently in \cite{CC2025}.
By  \cite[Proposition 3.2, Remark 4.1 and Theorem 4.2]{CC2025}, we have
\begin{equation}\label{eq:ratioid}
\frac{F_{N}^{m,k+1}(a;\epsilon)}{F_{N}^{m,k}(a;\epsilon)}
 =U_{11}(0),
\end{equation}
where $U$ is the unique solution of the following RH problem, which depends on the parameters $\epsilon\in\{0,1\}$, $a\in(0,1]$, $N\in\mathbb N$, $m\in\{1,\ldots, N\}$, $k\in\{1,\ldots, m\}$. %

\subsubsection*{RH problem for $U$}
\begin{itemize}
\item[(a)] $U:\mathbb C\setminus\left(\gamma_0\cup\gamma_1\right)\to \mathbb C^{2\times 2}$ is analytic, where $\gamma_0$ is a simple positively oriented closed contour around $a$ but not enclosing $-1/a$, and $\gamma_1$ is a simple positively oriented contour encircling $\gamma_0$ and $0$.
\item[(b)] $U$ admits continuous boundary values $U_\pm(z)$ as $z\in\gamma_0\cup\gamma_1$ is approached from the left ($+$) or right ($-$) according to the orientation of the contour, and they are related by
$U_+(s)=U_-(z)J_U(z)$ for $z\in \gamma_0\cup\gamma_1$, with
\begin{align*}
J_U(z)=\begin{cases}
\begin{pmatrix}
1& \frac{z^{N-m+k+\epsilon-1}}{(z-a)^{N-m+\epsilon}(1+az)^m}\\ 0& 1
\end{pmatrix},&z\in\gamma_0,\\
\begin{pmatrix}
1& 0\\ -\frac{(z-a)^{N-m+\epsilon}(1+az)^m}{z^{N-m+k+\epsilon-1}}
& 1
\end{pmatrix},&z\in\gamma_1.\end{cases}
\end{align*}\item[(c)]{As $z\to \infty$, $U(z)={I}+\bigO(z^{-1})$}, where $I$ is the $2\times 2$ identity matrix.
\end{itemize}
It follows from standard complex analysis arguments that $\det U(z)\equiv 1$. Moreover, the RH solution $U$ depends on the choice of contours $\gamma_0,\gamma_1$, but the value $U_{11}(0)$ does not, provided that 
$\gamma_0$ and $\gamma_1$ satisfy the conditions described in the above RH problem.
\begin{remark}It is well known that domino tilings of the Aztec diamond are connected to discrete orthogonal polynomials with respect to a Krawtchouk weight, see \cite{Jptrf}, and this connection can be used to characterize $\frac{F_{N}^{m,k+1}(a;\epsilon)}{F_{N}^{m,k}(a;\epsilon)}$ in terms of a discrete RH problem, characterizing orthogonal polynomials with respect to a truncated Krawtchouk weight. In principle, one could attempt to analyze this discrete RH problem asymptotically using the techniques of \cite{Baiketal}. It would however be highly challenging to obtain subleading terms and uniform error terms in this way. The above RH problem for $U$ is simpler in the sense that it is a continuous RH problem with simple (rational) jump matrices, which greatly simplifies asymptotic analysis. We emphasize that the RH problem for $U$ is not directly related to orthogonal polynomials; it can in fact be interpreted as a RH problem in Fourier space rather than in position space, and its solution can be expressed in terms of Padé approximants, see \cite{CC2025} for details.
\end{remark}

In view of asymptotic analysis, it will be convenient to change variable $z\mapsto 1/z$ in this RH problem.
To that end, we define 
\begin{equation}\label{def:Y}
Y(z)=\sigma_3U(0)^{-1}U(1/z)\sigma_3,\qquad\sigma_3=\begin{pmatrix}1&0\\0&-1\end{pmatrix}.
\end{equation}
Under the map $z\mapsto 1/z$, the contour $\gamma_1$ is sent to a closed negatively oriented loop $\Sigma_0$ enclosing $0$ but not enclosing $1/a$, while $\gamma_0$ is sent to a closed loop $\Sigma_1$ lying in the exterior of $\Sigma_0$ and not passing through $1/a$. 
If $\gamma_{0}$ encloses $0$, then $\Sigma_{1}$ is oriented negatively and encloses $-a,0$ but not $1/a$; if $\gamma_{0}$ does not enclose $0$, then $\Sigma_{1}$ is oriented positively and encloses $1/a$ but not $-a,0$; if $\gamma_0$ passes through $0$, then $\Sigma_1$ closes at infinity. While all these choices are admissible, below we will always assume that $\gamma_{0}$ surrounds $0$, so that $\Sigma_{1}$ surrounds $-a,0$ but not $1/a$.
After changing the orientation of the contours $\Sigma_0$ and $\Sigma_1$, we obtain the following RH problem for $Y$.

\subsubsection*{RH problem for $Y$}
\begin{itemize}
\item[(a)] $Y:\mathbb C\setminus\left(\Sigma_0\cup\Sigma_1\right)\to \mathbb C^{2\times 2}$ is analytic, with $\Sigma_0$ positively oriented and enclosing $0$, and $\Sigma_1$ either positively oriented and enclosing $\Sigma_0$ and $-a$ without enclosing $1/a$, or negatively oriented and enclosing $1/a$ without enclosing nor intersecting $\Sigma_0$ and without enclosing $-a$.
\item[(b)] $Y_+(z)=Y_-(z)J_Y(z)$ for $z\in \Sigma_0\cup\Sigma_1$, with
\begin{align}\label{eq:JY}
J_Y(z)=\begin{cases}
\begin{pmatrix}
1& \frac{z^{m-k+1}}{(1-az)^{N-m+\epsilon}(z+a)^m}\\ 0& 1
\end{pmatrix},&z\in\Sigma_1,\\
\begin{pmatrix}
1& 0\\ - \frac{(1-az)^{N-m+\epsilon}(z+a)^m}{z^{m-k+1}}
& 1
\end{pmatrix},&z\in\Sigma_0.\end{cases}
\end{align}
\item[(c)]{As $z\to \infty$, $Y(z) = I +\bigO(z^{-1})$}.
\end{itemize}
Using \eqref{def:Y}, \eqref{eq:ratioid} and the fact that $U(z)=I+\bigO(z^{-1})$ as $z\to \infty$, the ratio identity in terms of $Y$ reads
\begin{equation}
\label{eq:ratioid2}
\frac{F_{N}^{m,k+1}(a;\epsilon)}{F_{N}^{m,k}(a;\epsilon)}=Y_{22}(0)\det U(0)=Y_{22}(0).
\end{equation}
Let us parametrize $m\in\{1,\ldots, N\}$ and $k\in\{1,\ldots, m\}$ as follows:
\begin{equation}
\label{def:parametr}
k=\kappa N,\quad m=\mu N,\qquad {\mbox{such that}\quad \mu\in\left(0,1\right],\quad \kappa\in (0,\mu].}
\end{equation}
In a large part of our analysis, we will fix $\mu\in \big(\frac{a^2}{1+a^2},1\big]$. Later on, we will argue that our results are valid uniformly for $\mu\in\big(\frac{a^2}{1+a^2}(1+\delta),1-\delta\big]$.
On the other hand, we will treat $\kappa\in\left(0,\mu\right]$ as a varying parameter throughout our analysis. This is needed in particular because the ratio identity \eqref{eq:ratioid2} involves two different values of $\kappa$. For this reason, in our notations below, we will omit the dependence on $\mu$ of several quantities, while we do indicate the dependence on $\kappa$.

For $z\in \C\setminus \big( (-\infty,0]\cup [\frac{1}{a},+\infty) \big)$, we write 
\begin{equation}\label{def:phi}
\phi(z;\kappa)=(\kappa - \mu)  \log z + (1-\mu)\log(1-az) + \mu\log(z+a) ,
\end{equation}
where the principal branch is used for the logarithms, such that the jump matrix $J_Y$ can be rewritten as
\begin{align}\label{eq:JY2}
J_Y(z)=\begin{cases}
\begin{pmatrix}
1& \frac{z}{(1-az)^\epsilon}e^{-N\phi(z;\kappa)}\\ 0& 1
\end{pmatrix},&z\in\Sigma_1,\\
\begin{pmatrix}
1& 0\\ -\frac{(1-az)^\epsilon}{z} e^{N\phi(z;\kappa)}
& 1
\end{pmatrix},&z\in\Sigma_0.\end{cases}
\end{align}

\section{Small removed corner}\label{sec:small frozen region}
We assume in this section that $\mu\in\big(\frac{a^2}{1+a^2},1\big)$ and that $\kappa\in[\kappa_2,\mu]$, with $\kappa_2$ given by \eqref{def:kappa2}.
For $\kappa\in(\kappa_2,\mu]$ and with $\phi$ given by \eqref{def:phi}, the discriminant of the saddle point equation $\phi'(z;\kappa)=0$ is positive, so that we have two real saddle points 
\begin{align*}
& x_{1}(\kappa) = \frac{\kappa + a^{2}(2\mu-\kappa-1) - \sqrt{-a^{2}(1+4\kappa\mu -2\kappa - 2 \kappa^{2}-(2\mu-1)^{2})+\kappa^{2}+a^{4}(\kappa - 2\mu+1)^{2}}}{2a(1+\kappa - \mu)}, \\
& x_{2}(\kappa) = \frac{\kappa + a^{2}(2\mu-\kappa-1) + \sqrt{-a^{2}(1+4\kappa\mu-2\kappa- 2 \kappa^{2}-(2\mu-1)^{2})+\kappa^{2}+a^{4}(\kappa - 2\mu+1)^{2}}}{2a(1+\kappa - \mu)}.
\end{align*}
For $\kappa=\kappa_2$, these two saddle points coincide, such that there is a double saddle point at
\begin{align}\label{def of xstar}
x^*:=\frac{\kappa_2 + a^{2}(2\mu-\kappa_2-1)}{2a(1+\kappa_2 - \mu)}.
\end{align}
Note also that $x_{1}(\mu)=0$ and 
\begin{align*}
0 < x_{1}(\kappa) < x_{2}(\kappa)< \frac{1}{a}, \qquad \mbox{for all } \mu \in \Big(\frac{a^{2}}{1+a^{2}},1\Big), \; \kappa \in (\kappa_{2},\mu).
\end{align*} 
By \eqref{def:phi}, we can rewrite $\phi'(z;\kappa)$ as follows,
\[\phi'(z;\kappa)=\left(\kappa+1-\mu\right)\frac{(z-x_1(\kappa))(z-x_2(\kappa))}{z(z+a)(z-1/a)}.\]

\medskip

We can use the theory of critical trajectories of quadratic differentials \cite{Strebel} to understand the behavior of $\Re \phi(z;\kappa)$ as a function of $z\in\mathbb C$.
An object of the form $\phi'(z;\kappa)^2 dz^2$ with $\phi'(z;\kappa)^2$  rational, is called a quadratic differential. A curve is called a (vertical) trajectory of this quadratic differential if  $\phi'(z;\kappa)^2dz^2<0$ for $z$ on the curve, or equivalently if $\Re \big(\phi'(z;\kappa) dz\big)=0$, i.e.
\[\Re\left(\frac{(z-x_1(\kappa))(z-x_2(\kappa))}{z(z+a)(z-1/a)}dz\right)=0.\]
By the fundamental theorem of calculus, the real part of $\phi(z;\kappa)$ is constant along such trajectories.
There is an extensive literature on the local and global structure of trajectories of quadratic differentials, see e.g.\ \cite{Strebel}. We limit ourselves to recalling the main results that we will need for our specific quadratic differential which has two double zeros $x_1(\kappa),x_2(\kappa)$ if $\kappa>\kappa_2$ and three double poles $-a,0,1/a$.
 Trajectories going through one of the zeros $x_1(\kappa),x_2(\kappa)$ are called critical trajectories. 
 These critical trajectories consist locally near that zero of two orthogonal curves making angles $\pm\pi/4$ with the real line. 
Other, non-critical trajectories, are smooth curves without self-intersections.  
 If $\kappa=\kappa_2$ such that $x_1(\kappa)=x_2(\kappa)$, the critical trajectory through this point consists of three curves making angles $-\pi/6,\pi/6, \pi/2$ with the real line. 
Moreover, a trajectory of $\phi'(z;\kappa)^2dz^2$ can only contain a closed loop if it encloses at least one of the poles $-a,0,1/a$; otherwise the maximum and minimum principles applied to the harmonic function $\Re\phi(z;\kappa)$ would imply that $\Re\phi(z;\kappa)$ is constant in the interior of the closed loop, which is impossible.
We denote the trajectory of level $h$ as
\[\mathcal T_h:=\{z\in\mathbb C:\Re\phi(z;\kappa)=h\}.\] Of course, trajectories corresponding to different levels cannot intersect, and they are invariant under complex conjugation since $\overline{\phi'(\overline{z};\kappa)}=\phi'(z;\kappa)$. It will be useful to observe that
\begin{align}\label{re phi at the poles}
\Re \phi(\infty,\kappa) = \Re \phi(0,\kappa) = +\infty, \qquad \Re \phi(-a,\kappa) = \Re \phi(\tfrac{1}{a},\kappa)=-\infty.
\end{align} The next result summarizes the global structure of the critical trajectories, both for $\kappa>\kappa_2$ and for $\kappa=\kappa_2$. We don't need the latter in this section, but we will need it later when studying the transition where $\kappa\approx\kappa_2$.

\begin{figure}
\begin{center}
\begin{tikzpicture}[master]
\node at (0,0) {\includegraphics[width=7cm]{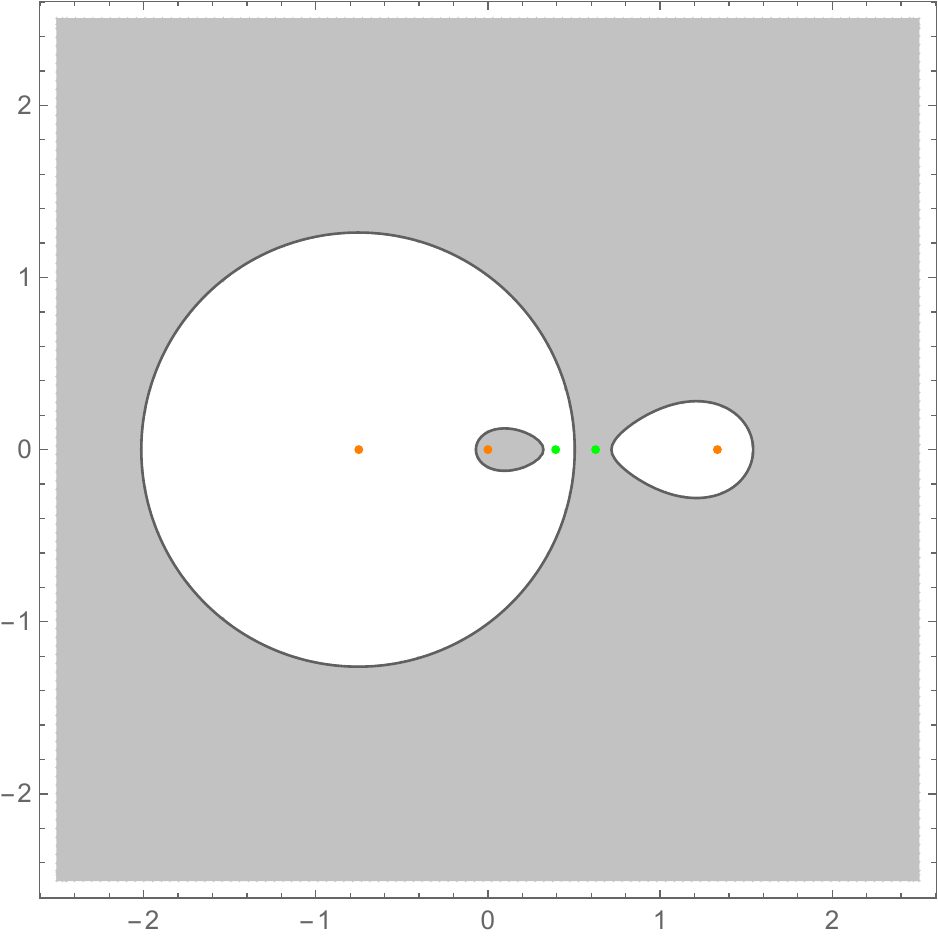}};
\node at (0.25,0.42) {\tiny $\Gamma_{\hspace{-0.025cm}1}\hspace{-0.025cm}(\hspace{-0.025cm}h\hspace{-0.025cm})$};
\node at (-0.75,1.9) {\tiny $\Gamma_{\hspace{-0.025cm}2}\hspace{-0.025cm}(\hspace{-0.025cm}h\hspace{-0.025cm})$};
\node at (1.75,0.65) {\tiny $\Gamma_{\hspace{-0.025cm}3}\hspace{-0.025cm}(\hspace{-0.025cm}h\hspace{-0.025cm})$};
\node at (0,3.8) {$\kappa>\kappa_{2}$};
\end{tikzpicture}
\begin{tikzpicture}[slave]
\node at (0,0) {\includegraphics[width=7cm]{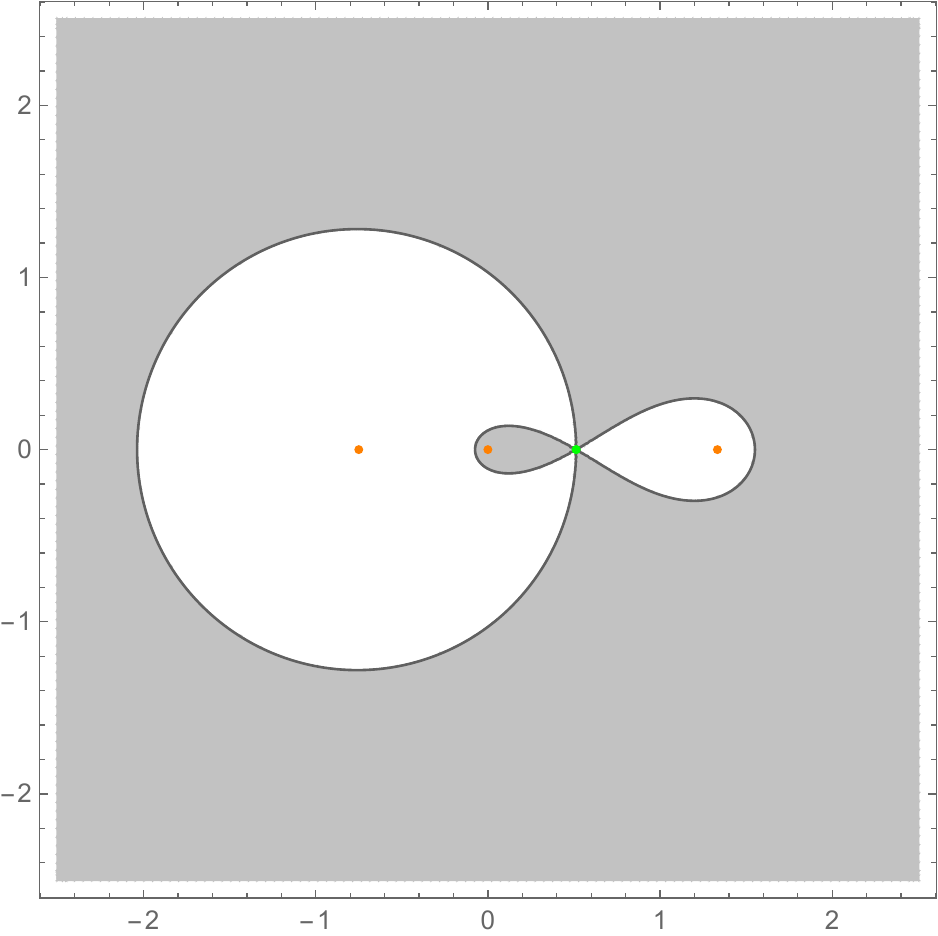}};
\node at (0.25,0.44) {\tiny $\Gamma_{\hspace{-0.025cm}1}$};
\node at (-0.75,1.92) {\tiny $\Gamma_{\hspace{-0.025cm}2}$};
\node at (1.75,0.65) {\tiny $\Gamma_{\hspace{-0.025cm}3}$};
\node at (0,3.8) {$\kappa=\kappa_{2}$};
\end{tikzpicture}
\end{center}
\caption{{In all images, the shaded regions correspond to $\{z:\Re \phi(z;\kappa)>h\}$ and the white regions to $\{z:\Re \phi(z;\kappa)<h\}$. The orange dots are $-a,0,1/a$, and the green dots are $x_{1}(\kappa),x_{2}(\kappa)$. Left: $a=0.75$, $\mu=0.808 \in (\frac{a^{2}}{1+a^{2}},1)$, $\kappa = 0.95\kappa_{2}+0.05\mu > \kappa_{2}$, $h=(h_{1}+h_{2})/2$. Right: $a=0.75$, $\mu=0.808 \in (\frac{a^{2}}{1+a^{2}},1)$, $\kappa = \kappa_{2}$, $h=\Re \phi(x^{*};\kappa_{2})$.} \label{fig:Re phi}}
\end{figure}
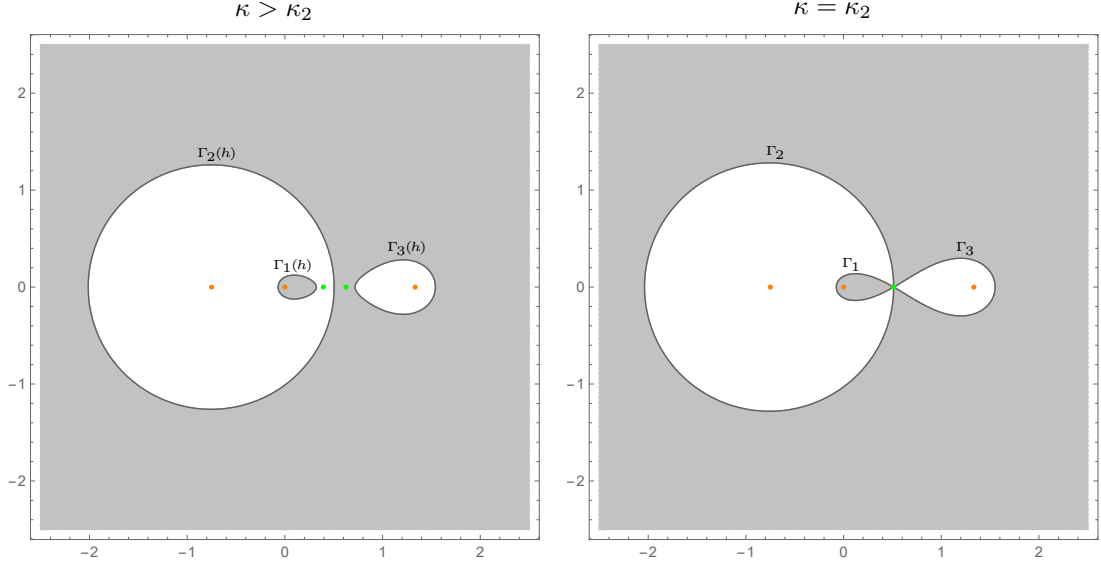

\begin{proposition}\label{prop:crittraj1}
\begin{enumerate} 
\item Let $\kappa>\kappa_2$. Denote $h_1=\Re \phi(x_1(\kappa);\kappa)$ and $h_2=\Re \phi(x_2(\kappa);\kappa)$. For $h_1<h<h_2$, the trajectory $\mathcal T_{h}=\Gamma_1(h)\cup\Gamma_2(h)\cup\Gamma_3(h)$ consists of three disjoint simple closed curves $\Gamma_1(h),\Gamma_2(h),\Gamma_3(h)$, symmetric under complex conjugation; $\Gamma_1(h)$ encloses $0$ but neither of the points $-a,x_1(\kappa),x_2(\kappa)$, $1/a$; $\Gamma_2(h)$ encloses $-a,0,x_1(\kappa)$ but does not enclose  $x_2(\kappa), 1/a$; $\Gamma_3(h)$ encloses $1/a$ but neither of the points $-a,0,x_1(\kappa),x_2(\kappa)$. See Figure \ref{fig:Re phi} (left).
\item Let $\kappa=\kappa_2$. Denote $h=\Re\phi(x^*;\kappa_2)$. The critical trajectory $\mathcal T_{h}=\Gamma_1\cup\Gamma_2\cup\Gamma_3$ is the union of three simple closed curves intersecting only at $x^*$; $\Gamma_1$ encloses $0$ but neither of the points $-a, 1/a$; $\Gamma_2$ encloses $\Gamma_1$ and $-a,0$ but not $1/a$; $\Gamma_3$ encloses $1/a$ but neither of the points $-a,0$. See Figure \ref{fig:Re phi} (right).
\end{enumerate}
\end{proposition}
\begin{proof}
For $h$ large and positive, we easily conclude from \eqref{re phi at the poles} that $\mathcal T_h$ consists of a 
small closed curve $\mathcal C_1(h)$ around $0$ and a large closed curve $\mathcal C_2(h)$ enclosing all of the points $-a,0,x_1(\kappa),x_2(\kappa),1/a$. 
As $h$ decreases,
$\mathcal C_1(h),\mathcal C_2(h)$ are homotopically deformed in $\mathbb C\setminus\{-a,0,x_1(\kappa),x_2(\kappa),1/a\}$, until $h=\Re\phi(x_2(\kappa);\kappa)$.

Suppose first that $\kappa>\kappa_2$, such that $h_1<h_2$.
Since $\mathcal T_{h_2}$ contains two orthogonal curves going through $x_2(\kappa)$ and since it can only contain a closed loop if it encloses at least one of the poles $-a,0,1/a$, necessarily $x_2(\kappa)$ belongs to $\mathcal C_2(h_2)$ which turns from a simple closed loop into an eight-shaped contour as $h\to h_2$. As $h$ decreases further, $\mathcal C_2(h)$ splits into two disjoint simple closed loops, the right one going around $1/a$ and the left one going around $\mathcal C_1(h)$ and $-a$. (As $h\to h_1$, the left one connects with $\mathcal C_1(h)$, and for $h<h_1$, they become one closed curve around $-a$ not enclosing $0$.)

Next, suppose that $\kappa=\kappa_2$, such that $h_1=h_2$.
Then, the critical trajectory $\mathcal T_{h}$ consists locally near $x^*$ of three curves. Since closed loops in $\mathcal T(h)$ need to enclose a pole, this can only happen if $x^*$ belongs to both of the components $\mathcal C_1(h),\mathcal C_2(h)$. In other words, $\mathcal C_2(h)=\Gamma_2\cup\Gamma_3$ is eight-shaped with intersection at $x^*$, and $\mathcal C_1(h)=\Gamma_1$ also contains $x^*$ while it lies inside the leftmost loop $\Gamma_2$ of $\mathcal C_2(\kappa_2)$.
\end{proof}

\begin{corollary}\label{prop:contours}
\begin{enumerate}
\item There exists $c>0$  
such that for any $\kappa\in\left(\kappa_2,\mu\right]$, there exist $h\in\mathbb R$ and disjoint contours $\Sigma_0$ and $\Sigma_1$,
with $\Sigma_0,\Sigma_1$ as specified in the RH problem for $Y$, such that
\begin{align*}
&\Re\phi(z;\kappa)-h\geq c(\kappa-\kappa_2)^{3/2},&z\in\Sigma_1,\\
&\Re\phi(z;\kappa)-h\leq -c(\kappa-\kappa_2)^{3/2},&z\in\Sigma_0.
\end{align*}
\item Let $\kappa=\kappa_2$. There exist contours $\Sigma_0,\Sigma_1$ intersecting only at $x^*$,
with $\Sigma_0,\Sigma_1$ as specified in the RH problem for $Y$, such that the following holds: for any $\delta>0$, there exists $c>0$ such that
\begin{align*}
&\Re\phi(z;\kappa_2)-\Re\phi(x^*;\kappa_2)>c,&z\in\Sigma_1\setminus\{|z-x^*|\leq \delta\},\\
&\Re\phi(z;\kappa_2)-\Re\phi(x^*;\kappa_2)<-c,&z\in\Sigma_0\setminus\{|z-x^*|\leq\delta\}.
\end{align*}
\end{enumerate}
\end{corollary}
\begin{proof}
For $\kappa=\kappa_2$, we take $\Sigma_0$ in the region between $\Gamma_1$ and $\Gamma_2$ (thus $\Sigma_0$ passes through $x^{*}$). Since $-a$ lies in that region and since $\Re\phi(-a;\kappa_{2})=-\infty$, we then  have $\Re\phi(z;\kappa_2)< \Re\phi(x^*;\kappa_2)$ on $\Sigma_0\setminus \{x^{*}\}$. Further, we take $\Sigma_1$ outside $\Gamma_2$ and going around $\Gamma_3$ (thus $\Sigma_1$ also passes through $x^{*}$). Since $\Re\phi(\infty;\kappa_2)=+\infty$, we have $ \Re\phi(z;\kappa_2)>\Re\phi(x^*;\kappa_2)$ on $\Sigma_1\setminus\{x^*\}$.

For $\kappa>\kappa_2$, we use the notations of Proposition \ref{prop:crittraj1} and let $h_1<h'<h<h''<h_2$.
Since $h_2-h_1\geq 3c(\kappa-\kappa_2)^{3/2}$ for some $c>0$, we can take $h',h''$ moreover such that $h-h',h''-h\geq c(\kappa-\kappa_2)^{3/2}$.
We take $\Sigma_0=\Gamma_1(h')$ and $\Sigma_1= \Gamma_3(h'')$. Then we have 
\begin{align*}
&\Re\phi(z;\kappa)-h=h''-h\geq c(\kappa-\kappa_2)^{3/2},&z\in\Sigma_1,\\
&\Re\phi(z;\kappa)-h=h'-h\leq -c(\kappa-\kappa_2)^{3/2},&z\in\Sigma_0.
\end{align*}
This proves the result.
\end{proof}

We will now prove that the RH problem for $Y$ is equivalent to a small-norm RH problem for $\kappa\in(\kappa_2,\mu]$, with jump matrices exponentially close to identity as $N\to\infty$.

In the remaining part of this section, we let $\kappa>\kappa_2+\delta$ for some $\delta>0$ and we choose $\Sigma_1,\Sigma_0$ such that they satisfy the inequalities in Corollary \ref{prop:contours}, part 1, for some $h \in \R$. 
We define
\begin{equation}
\label{def:T}
T(z)=e^{\frac{N}{2}h\sigma_3}
Y(z)
e^{-\frac{N}{2}h\sigma_3}.
\end{equation}
Then $T$ solves the following RH problem, whose jumps are exponentially close to the identity.

\subsubsection*{RH problem for $T$}
\begin{itemize}
\item[(a)] $T:\mathbb C\setminus\left(\Sigma_0\cup\Sigma_1\right)\to \mathbb C^{2\times 2}$ is analytic.
\item[(b)] $T_+(z)=T_-(z)J_T(z)$ for $z\in \Sigma_0\cup\Sigma_1$, with
\begin{align}\label{eq:JT}
J_T(z)=\begin{cases}
\begin{pmatrix}
1& \frac{z}{(1-az)^\epsilon}e^{-N(\phi(z;\kappa)-h)}\\ 0& 1
\end{pmatrix},&z\in\Sigma_1,\\
\begin{pmatrix}
1& 0\\ -\frac{(1-az)^\epsilon}{z} e^{N(\phi(z;\kappa)-h)}
& 1
\end{pmatrix},&z\in\Sigma_0.\end{cases}
\end{align}\item[(c)]{As $z\to \infty$, $T(z) = I +\bigO(z^{-1})$}.
\end{itemize}

We obtain from {Corollary} \ref{prop:contours} that, for sufficiently small $c>0$, 
\begin{align*}
J_T(z)=I+\mathcal O(e^{-c N(\kappa-\kappa_2)^{3/2}}),\qquad \mbox{as } N\to+\infty,
\end{align*}
uniformly for $z\in\Sigma_0\cup\Sigma_1$. It then follows from general RH theory \cite{DKMVZ,DeiftZhou} that
\begin{align*}
T(z)=I+\mathcal O(e^{-c N(\kappa-\kappa_2)^{3/2}}),\qquad \mbox{as } N\to+\infty,
\end{align*}
uniformly for $z\in\mathbb C\setminus(\Sigma_0\cup\Sigma_1)$.
In particular, we obtain that
\begin{align*}
T_{22}(0)=Y_{22}(0)=1+\mathcal O(e^{-c N(\kappa-\kappa_2)^{3/2}}),\qquad \mbox{as } N\to+\infty,
\end{align*}
such that \eqref{eq:ratioid2} implies the following result.

\begin{proposition}
There exists $c>0$ such that, for any $\delta>0$,\begin{align*}
\frac{F_N^{m,k+1}(a;\epsilon)}{F_N^{m,k}(a;\epsilon)}=1+O(e^{-c{N(\kappa-\kappa_2)^{3/2}}}),\qquad N\to\infty,
\end{align*}
uniformly for $N(\kappa_2+\delta)<k\leq N$.
\end{proposition}
It is straightforward to verify that all the error terms appearing above are uniform for $\frac{a^2}{1+a^2}(1+\delta)\leq {\mu} \leq 1-\delta$, for any {fixed} $\delta>0$.
We now easily prove Theorem \ref{thm:frozen} by taking a telescoping product of the ratios $\frac{F_N^{m,\ell+1}(a;\epsilon)}{F_N^{m,\ell}(a;\epsilon)}$ for $\ell=k,\ldots, m$, since $F_N^{m,m+1}(a;\epsilon)=F_N(a;\epsilon)=(1+a^2)^{\frac{N(N+1)}{2}}$.

\section{Critical removed corner}
\label{sec:critical}
In this section, we let $\mu=m/N\in\big(\frac{a^2}{1+a^2},1\big)$ and $\kappa=k/N \in [\kappa_{2}-N^{c-1},\kappa_{2}+\delta]$ for some $c\in (0,2/5)$ and for some small but fixed $\delta>0$.
\subsection{Transformation $Y\mapsto T$}
We denote 
\begin{align*}
& x^*:= \frac{\kappa_2 + a^{2}(2\mu-\kappa_2-1)}{2a(1+\kappa_2 - \mu)}, \qquad \phi^* := \phi(x^*;\kappa_2).
\end{align*}
In other words, $x^*$ is the double saddle point of $\phi$ for $\kappa=\kappa_2$ (recall the discussion around \eqref{def of xstar}), and $\phi^*$ is the value of $\phi$ at this double saddle point.
We define $T$ as in \eqref{def:T} with $h=\phi(x^*;\kappa)$, such that it satisfies the following RH problem.
\subsubsection*{RH problem for $T$}
\begin{itemize}
\item[(a)] $T:\C\setminus\left(\Sigma_0\cup\Sigma_1\right)\to \mathbb C^{2\times 2}$ is analytic.
\item[(b)] $T_+(z)=T_-(z)J_T(z)$ for $z\in \Sigma_0\cup\Sigma_1$, with
\begin{align}\label{eq:JT PII}
J_T(z)=\begin{cases}
\begin{pmatrix}
1& \frac{z}{(1-az)^{\epsilon}}e^{-N(\phi(z;\kappa)-\phi(x^*;\kappa))}  \\ 0& 1
\end{pmatrix},&z\in\Sigma_1,\\
\begin{pmatrix}
1& 0\\ 
-\frac{(1-az)^{\epsilon}}{z}e^{N(\phi(z;\kappa)-\phi(x^*;\kappa))}  & 1
\end{pmatrix},&z\in\Sigma_0.
\end{cases}
\end{align}
\item[(c)]{As $z\to \infty$, $T(z)=I +\bigO(z^{-1})$}.
\end{itemize}

Recall from Corollary \ref{prop:contours} (part 2) that for $\kappa=\kappa_2$, we can choose the jump
contours $\Sigma_0,\Sigma_1$ such that they intersect at $x^*$, with $\Sigma_0$ enclosing $0$, $\Sigma_1$ enclosing $1/a$ but not enclosing $-a$ nor $\Sigma_0$, and moreover such that 
\begin{equation}\label{eq:estimateJYÏI}
J_T(z)=I+\mathcal O(e^{-c N}),\qquad z\in(\Sigma_0\cup\Sigma_1)\setminus\{|z-x^*|\leq \delta\},
\end{equation}
as $N\to+\infty$.
For $\kappa=\kappa_2$, we have
\begin{equation}\label{eq:expandphi}
\phi(z;\kappa_2)=\phi^*+\frac{\phi'''(x^*;\kappa_2)}{6}(z-x^*)^3+\mathcal O((z-x^*)^4),\qquad z\to x^*,
\end{equation} which implies that 
we can choose these curves $\Sigma_0,\Sigma_1$ such that they intersect each other and the real line at $x^*$ with angles, for instance, $\pm\pi/3$.
If we take $\kappa$ sufficiently close to $\kappa_2$, then we easily see that there exists $\delta>0$ such that the estimate \eqref{eq:estimateJYÏI} remains valid, uniformly in $z$ and $\kappa$.

By \eqref{def:phi} and \eqref{eq:expandphi}, we have
\begin{equation}\label{eq:expandphi2}
\phi(z;\kappa)-\phi(x^*;\kappa)=\frac{\phi'''(x^*;\kappa_2)}{6}(z-x^*)^3+ (\kappa-\kappa_2)\log\frac{z}{x^*}+\mathcal O((z-x^*)^4),\qquad z\to x^*.
\end{equation}
{(Note that $x^{*}$ remains bounded away from $0$ when $\mu$ remains bounded away from $\frac{a^{2}}{1+a^{2}}$.)}
This implies cubic exponential behavior of the jump matrix $J_T$ near $x^*$ for $\kappa$ close to $\kappa_2$. Analogies with similar situations in e.g.\ \cite{BDJ, BleherIts, ClaeysKuijlaars, ClaeysMauersberger}  suggest that we need to construct a local parametrix near $x^*$ built out of a model RH problem associated to the Painlev\'e II equation, which we describe next.

\subsection{Painlev\'{e} II model RH problem}

The general RH problem characterizing solutions to the Painlev\'e II equation 
\begin{equation}\label{eq:PII}
u''(s)=su(s)+2u^3(s)
\end{equation} 
depends on four parameters $s,a_{1},a_{2},a_{3}$ and is as follows, see \cite{FlaschkaNewell, FIKN}.

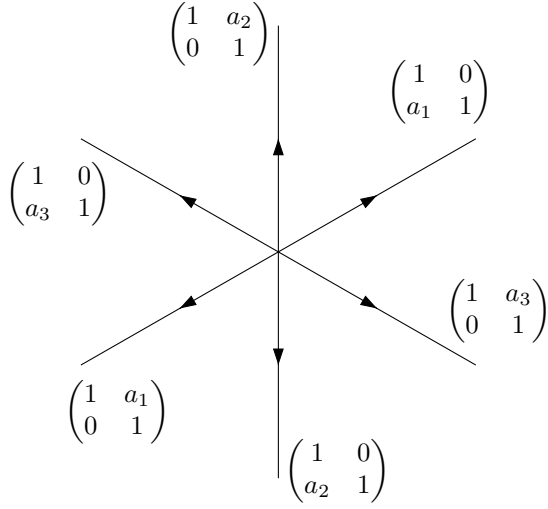
\begin{figure}[h]
\begin{center}
\begin{tikzpicture}
\node (A) at (0,0) {};

\node (a1) at (45:3) {$\begin{pmatrix} 1 & 0 \\ a_{1} & 1 \end{pmatrix}$};

\node (a2) at (105:3) {$\begin{pmatrix} 1 & a_{2} \\ 0 & 1 \end{pmatrix}$};

\node (a3) at (165:3) {$\begin{pmatrix} 1 & 0 \\ a_{3} & 1 \end{pmatrix}$};

\node (a4) at (-15:3) {$\begin{pmatrix} 1 & a_{3} \\ 0 & 1 \end{pmatrix}$};

\node (a5) at (-75:3) {$\begin{pmatrix} 1 & 0 \\ a_{2} & 1 \end{pmatrix}$};

\node (a6) at (-135:3) {$\begin{pmatrix} 1 & a_{1} \\ 0 & 1 \end{pmatrix}$};

\draw (0,0) -- (30:3);
\draw[black,arrows={-Triangle[length=0.22cm,width=0.14cm]}]
($(30:1.5)$) --  ++(30:0.001);

\draw (0,0) -- (90:3);
\draw[black,arrows={-Triangle[length=0.22cm,width=0.14cm]}]
($(90:1.5)$) --  ++(90:0.001);

\draw (0,0) -- (150:3);
\draw[black,arrows={-Triangle[length=0.22cm,width=0.14cm]}]
($(150:1.5)$) --  ++(150:0.001);

\draw (0,0) -- (-30:3);
\draw[black,arrows={-Triangle[length=0.22cm,width=0.14cm]}]
($(-30:1.5)$) --  ++(-30:0.001);

\draw (0,0) -- (-90:3);
\draw[black,arrows={-Triangle[length=0.22cm,width=0.14cm]}]
($(-90:1.5)$) --  ++(-90:0.001);

\draw (0,0) -- (-150:3);
\draw[black,arrows={-Triangle[length=0.22cm,width=0.14cm]}]
($(-150:1.5)$) --  ++(-150:0.001);
\end{tikzpicture}
\end{center}
\caption{\label{figII}The jump contour $\Sigma = \bigcup_{j=0}^{5} \{ xe^{i(\frac{\pi}{6}+j \frac{\pi}{3})}:x > 0 \}$, together with the corresponding jump matrices.}
\end{figure}
\subsubsection*{RH problem for $\Phi=\Phi(\cdot;s,a_{1},a_{2},a_{3})$}
\begin{itemize}
\item[(a)] $\Phi : \mathbb{C} \setminus \Sigma \rightarrow \mathbb{C}^{2 \times 2}$ is analytic, and $\Sigma$ consists of $6$ semi-infinite rays starting at the origin, as shown in Figure \ref{figII}.
\item[(b)] $\Phi$ has the jump relation $\Phi_+=\Phi_-J_\Phi$ on $\Sigma$, with $J_\Phi$ as shown in Figure \ref{figII}. 
\item[(c)] As $\zeta \to \infty$, we have
\begin{equation}
\Phi(\zeta) =  \left( I + \frac{1}{\zeta}\begin{pmatrix}
-\frac{1}{2i}q(s) & \frac{1}{2i}u(s) \\[0.1cm] - \frac{1}{2i}u(s) & \frac{1}{2i}q(s)
\end{pmatrix} + \bigO(\zeta^{-2}) \right) e^{-i( \frac{4}{3}\zeta^{3}+s\zeta )\sigma_{3}},
\end{equation}
where
\begin{equation}\label{defl:qP}
q(s)=\int_{\infty}^{s}u^{2}(s')ds',\end{equation}
and where 
$u$ is a solution to the Painlev\'e II equation \eqref{eq:PII}.

\smallskip \noindent As $\zeta \to 0$, we have
\begin{equation}
\Phi(\zeta) = \bigO(1).
\end{equation} 
\end{itemize}
If
\begin{equation}
a_{1}a_{2}a_{3} + a_{1}+a_{2}+a_{3} = 0,
\end{equation}
this RH problem is solvable for $s\in\mathbb C$, except at isolated singularities which are poles of the relevant Painlev\'e II solution $u$. The Painlev\'e II solution depends on the values of the Stokes parameters $a_1,a_2,a_3$.

\begin{figure}
\begin{center}
\begin{tikzpicture}
\node (A) at (0,0) {};

\node (a1) at (45:3) {$\begin{pmatrix} 1 & 1 \\ 0 & 1 \end{pmatrix}$};

\node (a2) at (135:3) {$\begin{pmatrix} 1 & 0 \\ 1 & 1 \end{pmatrix}$};

\node (a5) at (-45:3) {$\begin{pmatrix} 1 & 1 \\ 0 & 1 \end{pmatrix}$};

\node (a6) at (-135:3) {$\begin{pmatrix} 1 & 0 \\ 1 & 1 \end{pmatrix}$};

\draw (0,0) -- (60:3);
\draw[black,arrows={-Triangle[length=0.22cm,width=0.14cm]}]
($(60:1.5)$) --  ++(60:0.001);

\draw (0,0) -- (120:3);
\draw[black,arrows={-Triangle[length=0.22cm,width=0.14cm]}]
($(120:1.3)$) --  ++(120+180:0.001);

\draw (0,0) -- (-60:3);
\draw[black,arrows={-Triangle[length=0.22cm,width=0.14cm]}]
($(-60:1.3)$) --  ++(-60+180:0.001);

\draw (0,0) -- (-120:3);
\draw[black,arrows={-Triangle[length=0.22cm,width=0.14cm]}]
($(-120:1.5)$) --  ++(-120:0.001);
\end{tikzpicture}
\end{center}
\caption{\label{fig:Psi}The jump contour $\Sigma_{\Psi} = \bigcup_{\alpha_{j} = \frac{\pi}{3},\frac{2\pi}{3}} \{ xe^{i\alpha_{j}}:x \in \mathbb{R} \}$. The orientations are indicated by arrows, and the matrices along the rays are the corresponding jump matrices.}
\end{figure}
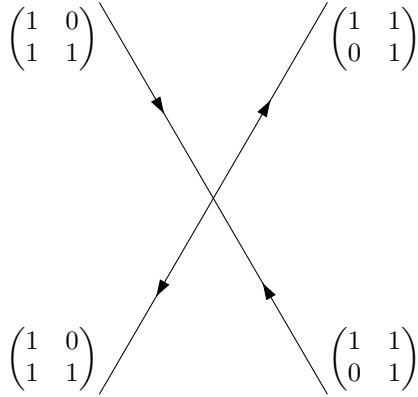

The Hastings-McLeod solution $u(s)$ of the Painlev\'{e} II equation is characterized by its 
rapid decay $u(s)\sim {\rm Ai}(s)$ as $s\to+\infty$, and it corresponds to $a_{1} = 1, a_{2} = 0, a_{3}=-1$. It has no poles on the real line. It is the RH problem corresponding to this solution that we will use to construct a local parametrix.
We define
\begin{equation}\label{def of Psi}
\Psi(\zeta;s) = \sigma_{3}\sigma_{1}\Phi(\zeta e^{\frac{i\pi}{2}};s,a_1=1,a_2=0,a_3=-1)\sigma_{1}\sigma_{3},
\end{equation}
where $\sigma_{1}=\begin{pmatrix}
0 & 1 \\ 1 & 0
\end{pmatrix}$ is the first Pauli matrix.
Then, $\Psi$ solves the following RH problem.
\subsubsection*{RH problem for $\Psi$}
\begin{itemize}
\item[(a)] $\Psi : \mathbb{C} \setminus \Sigma_{\Psi} \rightarrow \mathbb{C}^{2 \times 2}$ is analytic, and $\Sigma_{\Psi}$ consists of $4$ semi-infinite rays as shown in Figure \ref{fig:Psi}.
\item[(b)] $\Psi$ has the jump relation $\Psi_+=\Psi_-J_\Psi$ on $\Sigma_{\Psi}$, with $J_\Psi$ as shown in Figure \ref{fig:Psi} (note that we have reversed the orientation of two rays).
\item[(c)] As $\zeta \to \infty$, we have
\begin{equation}\label{eq:Psiasymp}
\Psi(\zeta) =  \left( I + \frac{1}{\zeta}\begin{pmatrix}
-\frac{1}{2}q(s) & -\frac{1}{2}u(s) \\[0.1cm] \frac{1}{2}u(s) & \frac{1}{2}q(s)
\end{pmatrix} + \Psi_2(s) \zeta^{-2}+\bigO(\zeta^{-3}) \right) e^{ (\frac{4}{3}\zeta^{3}-s\zeta )\sigma_{3}},
\end{equation}
where $\Psi_2(s)$ is a $2\times 2$ matrix depending on $s$.

\smallskip \noindent As $\zeta \to 0$, we have
\begin{equation}
\Psi(\zeta) = \bigO(1).
\end{equation} 
\end{itemize}
Observe that, by \eqref{def:TW} and \eqref{defl:qP}, we have
\begin{equation}\label{derTW}
\partial_s \log F^{\rm TW}(s)=-q(s).
\end{equation}

The asymptotic analysis of the RH problem for $\Phi$ as $s\to \pm\infty$ has been carried out in \cite{FIKN} (with \cite[Chapter 11]{FIKN} being particularly relevant to our setting). However, the specific estimates required for our purposes do not appear to be directly available in the literature. We will therefore perform two nonlinear steepest descent analyses on $\Psi$, inspired by \cite{FIKN}: first in the regime $s \to +\infty$, and then in the regime $s \to -\infty$.

\subsubsection{Asymptotic analysis for $\Psi$ as $s\to + \infty$}
As $s\to +\infty$, not only the Painlev\'e II solution $u(s)$ decays rapidly (faster than exponentially), but the same is true for the functions $q(s)$ and $\Psi_2(s)$ in the expansion \eqref{eq:Psiasymp}. To show this, we first consider $\Psi_{\mathrm{b}}(\zeta) = \Psi(\sqrt{s}\zeta)e^{-s^{3/2}(\frac{4}{3}\zeta^{3}-\zeta)\sigma_{3}}$. We then apply the transformation
\begin{align}\label{deformation of contour s pos}
\Psi_{\mathrm{c}}(\zeta) = \Psi_{\mathrm{b}}(\zeta) \begin{cases}
\begin{pmatrix} 1 & e^{2s^{3/2}(\frac{4}{3}\zeta^{3}-\zeta)} \\ 0 & 1 \end{pmatrix}, & \mbox{if } \zeta \in \Omega_{\mathrm{r}}, \\
\begin{pmatrix} 1 & 0 \\ e^{-2s^{3/2}(\frac{4}{3}\zeta^{3}-\zeta)} & 1 \end{pmatrix}, & \mbox{if } \zeta \in \Omega_{\mathrm{l}}, \\
I, & \mbox{otherwise},
\end{cases}
\end{align}
where 
\begin{align*}
\Omega_{\mathrm{r}} = \{z=r e^{i\theta}: r\in (0,1), \theta\in (-\tfrac{\pi}{3},\tfrac{\pi}{3}), \Re z \in (0,\tfrac{1}{2}) \}, \qquad \Omega_{\mathrm{l}} = -\Omega_{\mathrm{r}}.
\end{align*}
The jump contour for $\Psi_{\mathrm{c}}$ is shown in Figure \ref{fig:Psi s to +inf} (left). 
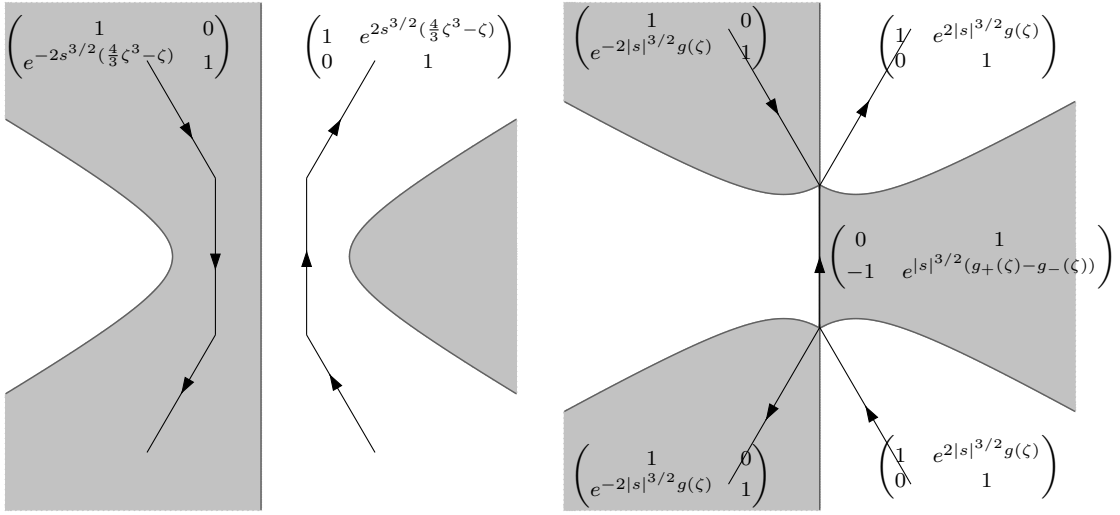
\begin{figure}
\begin{center}
\begin{tikzpicture}[master]
\node at (0,0) {\includegraphics[width=7cm]{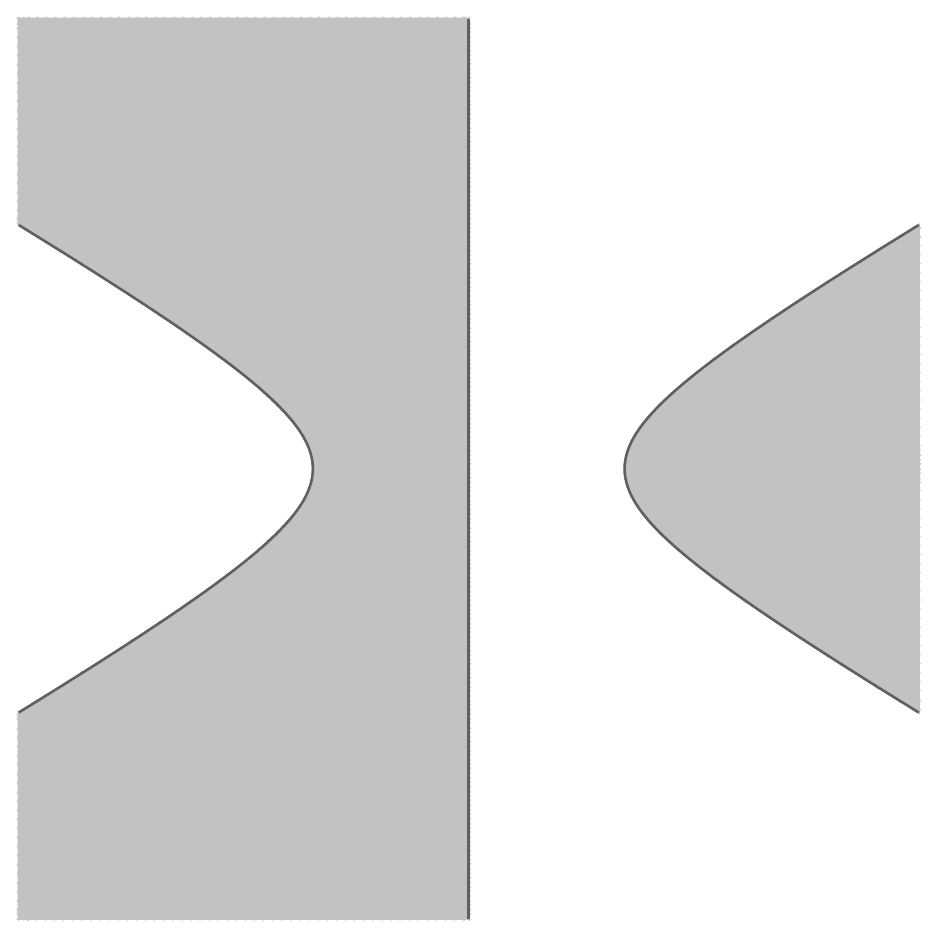}};

\node (a1) at (55:3.4) {\footnotesize $\begin{pmatrix} 1 & e^{2s^{3/2}(\frac{4}{3}\zeta^{3}-\zeta)} \\ 0 & 1 \end{pmatrix}$};

\node (a2) at (124:3.35) {\footnotesize $\begin{pmatrix} 1 & 0 \\ e^{-2s^{3/2}(\frac{4}{3}\zeta^{3}-\zeta)} & 1 \end{pmatrix}$};

\draw (60:1.2) -- (-60:1.2);
\draw[black,arrows={-Triangle[length=0.22cm,width=0.14cm]}]
($(0.6,0.08)$) --  ++(90:0.001);
\draw (120:1.2) -- (-120:1.2);
\draw[black,arrows={-Triangle[length=0.22cm,width=0.14cm]}]
($(-0.6,-0.18)$) --  ++(-90:0.001);

\draw (60:1.2) -- (60:3);
\draw[black,arrows={-Triangle[length=0.22cm,width=0.14cm]}]
($(60:2.1)$) --  ++(60:0.001);

\draw (120:1.2) -- (120:3);
\draw[black,arrows={-Triangle[length=0.22cm,width=0.14cm]}]
($(120:1.8)$) --  ++(120+180:0.001);

\draw (-60:1.2) -- (-60:3);
\draw[black,arrows={-Triangle[length=0.22cm,width=0.14cm]}]
($(-60:1.8)$) --  ++(-60+180:0.001);

\draw (-120:1.2) -- (-120:3);
\draw[black,arrows={-Triangle[length=0.22cm,width=0.14cm]}]
($(-120:2.1)$) --  ++(-120:0.001);


\end{tikzpicture} \begin{tikzpicture}[slave]
\node at (0,0) {\includegraphics[width=7cm]{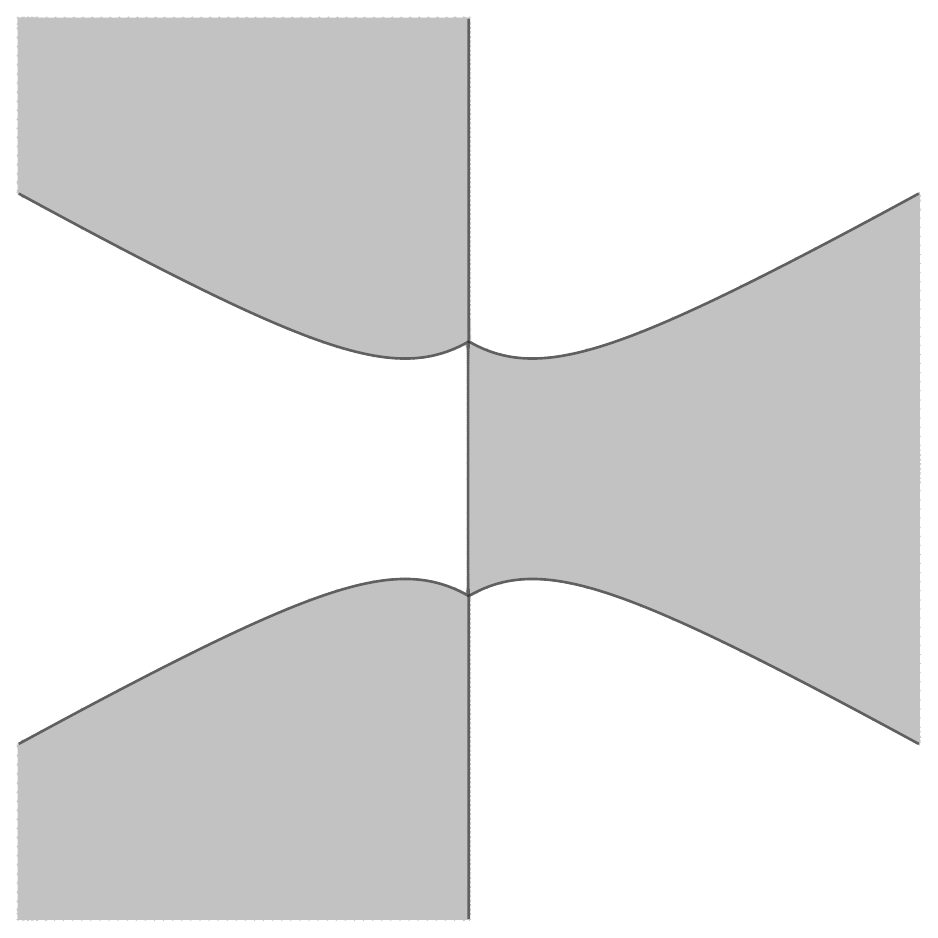}};

\node at (55:3.4) {\footnotesize $\begin{pmatrix} 1 & e^{2|s|^{3/2}g(\zeta)} \\ 0 & 1 \end{pmatrix}$};
\node at (-55:3.4) {\footnotesize $\begin{pmatrix} 1 & e^{2|s|^{3/2}g(\zeta)} \\ 0 & 1 \end{pmatrix}$};

\node at (124:3.5) {\footnotesize $\begin{pmatrix} 1 & 0 \\ e^{-2|s|^{3/2}g(\zeta)} & 1 \end{pmatrix}$};
\node at (-124:3.5) {\footnotesize $\begin{pmatrix} 1 & 0 \\ e^{-2|s|^{3/2}g(\zeta)} & 1 \end{pmatrix}$};

\node at (2,0) {\footnotesize $\begin{pmatrix} 0 & 1 \\ -1 & e^{|s|^{3/2}(g_{+}(\zeta)-g_{-}(\zeta))} \end{pmatrix}$};

\draw (0,0.94) -- ($(0,0.94)+(60:2.4)$);
\draw[black,arrows={-Triangle[length=0.22cm,width=0.14cm]}]
($(0,0.94)+(60:1.3)$) --  ++(60:0.001);

\draw (0,-0.94) -- ($(0,-0.94)+(-60:2.4)$);
\draw[black,arrows={-Triangle[length=0.22cm,width=0.14cm]}]
($(0,-0.94)+(-60:1.2)$) --  ++(120:0.001);

\draw (0,0.94) -- ($(0,0.94)+(120:2.4)$);
\draw[black,arrows={-Triangle[length=0.22cm,width=0.14cm]}]
($(0,0.94)+(120:1.08)$) --  ++(-60:0.001);

\draw (0,-0.94) -- ($(0,-0.94)+(-120:2.4)$);
\draw[black,arrows={-Triangle[length=0.22cm,width=0.14cm]}]
($(0,-0.94)+(-120:1.42)$) --  ++(-120:0.001);

\draw (0,-0.94) -- (0,0.94);
\draw[black,arrows={-Triangle[length=0.22cm,width=0.14cm]}]
($(0,0)$) --  ++(90:0.001);


\end{tikzpicture}
\end{center}
\caption{\label{fig:Psi s to +inf}
Left: The shaded regions correspond to $\{\zeta:\Re (\frac{4}{3}\zeta^{3}-\zeta)>0\}$ and the white regions to $\{\zeta:\Re (\frac{4}{3}\zeta^{3}-\zeta)<0\}$. The jump matrices for $\Psi_{\mathrm{c}}$ are uniformly close to $I$ on the contour as $s\to + \infty$. The vertical segments pass through $\pm \frac{1}{2}$. Right: The shaded regions correspond to $\{\zeta:\Re g(\zeta)>0\}$ and the white regions to $\{\zeta:g(\zeta)<0\}$. The jump matrices for $\tilde{\Psi}_{\mathrm{c}}$ are uniformly close to $I$ on the contour  as $s\to - \infty$, except on the vertical segment $[\frac{-i}{\sqrt{2}},\frac{i}{\sqrt{2}}]$ and in small neighborhoods of $\pm \frac{i}{\sqrt{2}}$.}
\end{figure}
The jump matrices for $\Psi_{\mathrm{c}}$ satisfy 
\begin{align*}
I+\bigO(e^{-2s^{3/2} |\Re (\frac{4}{3}\zeta^{3}-\zeta)|}), \qquad  \mbox{as } s \to +\infty
\end{align*}
uniformly for $\zeta$ on the contour $\Sigma_{\Psi_{\mathrm{c}}}$. It is also easy to check that $|\Re (\frac{4}{3}\zeta^{3}-\zeta)| \geq \frac{1}{3}$ for all $\zeta \in \Sigma_{\Psi_{c}}$. It then follows from standard theory for small norm RH problems \cite{DeiftZhou, DKMVZ} that
\begin{align*}
\Psi_{\mathrm{c}}(\zeta) = I+\bigO\Bigg(\frac{\int_{\Sigma_{\Psi_{\mathrm{c}}}}e^{-2s^{3/2} |\Re (\frac{4}{3}z^{3}-z)|}|dz|}{1+|\zeta|}\Bigg)=I+\bigO\Big(\frac{s^{-3/4}e^{-\frac{2}{3}s^{3/2}}}{1+|\zeta|}\Big), \qquad \mbox{as } s \to +\infty,
\end{align*} 
uniformly for $\zeta \in \C\setminus \Sigma_{\Psi_{\mathrm{c}}}$.
 Since $\Psi(\zeta) = \Psi_{c}(\zeta/\sqrt{s})e^{(\frac{4}{3}\zeta^{3}-s\zeta )\sigma_{3}}$ for all $|\zeta|>\sqrt{s}$, we conclude that
\begin{equation}\label{eq:Psiasymp2}
\Psi(\zeta) =  \bigg( I + \bigO\Big(\frac{s^{-1/4}}{\zeta}e^{-\frac{2}{3}s^{3/2}}\Big) \bigg) e^{ (\frac{4}{3}\zeta^{3}-s\zeta )\sigma_{3}}, \qquad \mbox{as } s\to + \infty
\end{equation}
uniformly for $|\zeta|>\sqrt{s}$. Comparing \eqref{eq:Psiasymp2} with \eqref{eq:Psiasymp}, we conclude that
\begin{align}\label{lol5}
u(s) = \bigO\Big(s^{-1/4}e^{-\frac{2}{3}s^{3/2}}\Big), \qquad q(s) = \bigO\Big(s^{-1/4}e^{-\frac{2}{3}s^{3/2}}\Big), \qquad \Psi_2(s)=\bigO\Big(s^{1/4}e^{-\frac{2}{3}s^{3/2}}\Big),
\end{align}
as $s\to+\infty$.

\subsubsection{Asymptotic analysis for $\Psi$ as $s\to - \infty$}
The expansion \eqref{eq:Psiasymp} does not hold uniformly as $s\to -\infty$. This is connected to the fact that $q,u, \Psi_2$ and further terms in the expansion blow up in this limit. The goal of this subsection is to obtain an expansion for $\Psi(\zeta)$ as $\zeta\to \infty$ that is valid uniformly for large negative $s$. We first define $\tilde{\Psi}(\zeta) = \Psi(\sqrt{-s}\zeta)$. Then, we analytically continue $\tilde\Psi$ from each of the $4$ sectors, in order to construct a new function $\tilde{\Psi}_{\mathrm{b}}$ with jumps on the contour shown in Figure \ref{fig:Psi s to +inf} (right - note that the jump matrices shown in the figure are not the ones for $\tilde{\Psi}_{\mathrm{b}}$; they correspond to $\tilde{\Psi}_{\mathrm{c}}$ defined below). The transformation $\tilde{\Psi}\mapsto \tilde{\Psi}_{\mathrm{b}}$ can be made explicit in a similar way as in \eqref{deformation of contour s pos}, but since this explicit transformation will not play any role for us, we omit it here. By \eqref{eq:Psiasymp}, as $\zeta\to \infty$ we have
\begin{align}\label{lol1}
\tilde{\Psi}_{\mathrm{b}}(\zeta) =  \left( I + \frac{1}{\zeta\sqrt{-s}}\begin{pmatrix}
-\frac{1}{2}q(s) & -\frac{1}{2}u(s) \\[0.1cm] \frac{1}{2}u(s) & \frac{1}{2}q(s)
\end{pmatrix} + \Psi_2(s) \zeta^{-2}|s|^{-1}+\bigO(\zeta^{-3}) \right) e^{ (-s)^{3/2}(\frac{4}{3}\zeta^{3}+\zeta )\sigma_{3}}.
\end{align}
Define $g(\zeta) = \frac{4}{3}(\zeta^{2}+\tfrac{1}{2})^{3/2}$, where the branch cut is chosen so that $g$ is analytic in $\C$ except on the vertical segment $[-\frac{i}{\sqrt{2}},\frac{i}{\sqrt{2}}]$, and such that
\begin{align}\label{lol2}
g(\zeta) = \frac{4}{3}\zeta^{3} + \zeta + \frac{1}{8\zeta} + \bigO(\zeta^{-3}), \qquad \mbox{as } \zeta \to \infty.
\end{align}
Let us orient $[-\frac{i}{\sqrt{2}},\frac{i}{\sqrt{2}}]$ upwards. We directly verify that
\begin{align*}
g_{+}(\zeta)+g_{-}(\zeta)=0, \qquad g_{+}(\zeta)-g_{-}(\zeta)<0, \qquad \mbox{for all } \zeta \in (-\tfrac{i}{\sqrt{2}},\tfrac{i}{\sqrt{2}}).
\end{align*}
We next apply the transformation 
\begin{align*}
\tilde{\Psi}_{\mathrm{c}}(\zeta) = \tilde{\Psi}_{\mathrm{b}}(\zeta)e^{-(-s)^{3/2}g(\zeta)\sigma_{3}}.
\end{align*}
By \eqref{lol1}--\eqref{lol2}, we have
\begin{align}\label{lol3}
\tilde{\Psi}_{\mathrm{c}}(\zeta) = I + \frac{1}{\zeta}\Bigg( -\frac{|s|^{3/2}}{8}\sigma_{3} + \frac{1}{|s|^{1/2}}\begin{pmatrix}
-\frac{q(s)}{2} & -\frac{u(s)}{2} \\[0.1cm] \frac{u(s)}{2} & \frac{q(s)}{2}
\end{pmatrix} \Bigg) + \bigO(\zeta^{-2}) \qquad \mbox{as } \zeta \to \infty.
\end{align}
The jump matrices for $\tilde{\Psi}_{\mathrm{c}}$ are shown in Figure \ref{fig:Psi s to +inf} (right); they are uniformly close to $I$ on the contour as $s\to - \infty$, except on the vertical segment $[\frac{-i}{\sqrt{2}},\frac{i}{\sqrt{2}}]$ and in small neighborhoods of $\pm \frac{i}{\sqrt{2}}$. To handle the jump matrix on $(-\tfrac{i}{\sqrt{2}},\tfrac{i}{\sqrt{2}})$, we use the global parametrix
\begin{align}\label{def:Qinf}
Q^{\infty}(\zeta)=\begin{pmatrix}1&-1\\-i&-i\end{pmatrix}
\left(\frac{\zeta+\frac{i}{\sqrt{2}}}{\zeta-\frac{i}{\sqrt{2}}}\right)^{\sigma_3 /4}
\begin{pmatrix}1&-1\\-i&-i\end{pmatrix}^{-1}, \qquad \mbox{for } \zeta \in \C \setminus [-\tfrac{i}{\sqrt{2}},\tfrac{i}{\sqrt{2}}],
\end{align}
where the branchs are chosen so that $Q^{\infty}$ is analytic on $\C \setminus [-\tfrac{i}{\sqrt{2}},\tfrac{i}{\sqrt{2}}]$ and behaves as 
\begin{equation}\label{Qinftyexpansion}Q^{\infty}(\zeta) = I +\frac{1}{2\sqrt{2}\zeta}\begin{pmatrix}0&-1\\1&0\end{pmatrix}+ \bigO(\zeta^{-2})\quad\mbox{ as $\zeta \to \infty$}.\end{equation}
It is easy to verify that
\begin{align*}
Q^{\infty}_{+}(\zeta) = Q^{\infty}_{-}(\zeta)\begin{pmatrix}
0 & 1 \\ -1 & 0
\end{pmatrix}, \qquad \zeta \in (-\tfrac{i}{\sqrt{2}},\tfrac{i}{\sqrt{2}}),
\end{align*}
so that $\tilde{\Psi}_{\mathrm{c}}(\zeta)Q^{\infty}(\zeta)^{-1}$ has uniformly small jumps, except near $\pm \frac{i}{\sqrt{2}}$. 
To handle the jumps of $\tilde{\Psi}_{\mathrm{c}}$ for $\zeta$ close to $\frac{i}{\sqrt{2}}$, we construct a local parametrix $Q^{(\frac{i}{\sqrt{2}})}(\zeta)$, defined for $|\zeta - \frac{i}{\sqrt{2}}|\leq \frac{1}{10}$ such that $P^{(\frac{i}{\sqrt{2}})}$ has the same jumps as $\tilde{\Psi}_{c}$ on this disk, and such that
\begin{align*}
Q^{(\frac{i}{\sqrt{2}})}(\zeta)Q^{\infty}(\zeta)^{-1} = I + \bigO(|s|^{-3/2}), \qquad \mbox{as } s \to - \infty
\end{align*}
uniformly for $|\zeta - \frac{i}{\sqrt{2}}| = \frac{1}{10}$. $Q^{(\frac{i}{\sqrt{2}})}$ can be constructed in terms of Airy function as in \cite{DKMVZ}, but since this expression will not play any role for us, we do not provide details here. We then define
\begin{align*}
\tilde{R}(\zeta) = \begin{cases}
\tilde{\Psi}_{\mathrm{c}}(\zeta) Q^{\infty}(\zeta)^{-1}, & \mbox{if } |\zeta-\frac{i}{\sqrt{2}}|>\frac{1}{10} \mbox{ and } |\zeta+\frac{i}{\sqrt{2}}|>\frac{1}{10}, \\
\tilde{\Psi}_{\mathrm{c}}(\zeta) Q^{(\frac{i}{\sqrt{2}})}(\zeta)^{-1}, & \mbox{if } |\zeta-\frac{i}{\sqrt{2}}|<\frac{1}{10}, \\
\tilde{\Psi}_{\mathrm{c}}(\zeta) \overline{Q^{(\frac{i}{\sqrt{2}})}(\overline{\zeta})}^{-1}, & \mbox{if } |\zeta+\frac{i}{\sqrt{2}}|<\frac{1}{10}.
\end{cases}
\end{align*}
The jumps for $\tilde{R}$ are $I+ \bigO(|s|^{-3/2}e^{-|\zeta|})$ as $s\to -\infty$ uniformly for $\zeta$ on the contour, and therefore, by standard theory for small norms RH problems \cite{DeiftZhou, DKMVZ},
\begin{align}\label{lol4}
\tilde{R}(\zeta) = I + \bigO\bigg( \frac{1}{(1+|\zeta|)|s|^{3/2}} \bigg), \qquad \mbox{as } s\to -\infty,
\end{align}
uniformly for $\zeta\in \C$. Since $\tilde{\Psi}_{\mathrm{c}}(\zeta) = \tilde{R}(\zeta)Q^{\infty}(\zeta)$ for all $|\zeta| \geq 1$, \eqref{lol3}, \eqref{Qinftyexpansion} and \eqref{lol4} imply that
\begin{align}\label{asymp of u and q}
u(s)=\sqrt{-s/2}+\bigO(s^{-1}),\qquad q(s)=-\frac{s^2}{4}+\mathcal O(1/s), \qquad \mbox{as } s \to - \infty.
\end{align}
In fact, by \cite[Theorem 10.2 (part 3 with $s_{2}=0$)]{FIKN}, the first formula in \eqref{asymp of u and q} can be improved, while by \eqref{derTW} and \eqref{eq:TWtail}, the second formula in \eqref{asymp of u and q}
can be improved (more precisely, we need to use here the fact that the $s$-derivative of the $\bigO$-term in \eqref{eq:TWtail} is $\bigO(s^{-5/2})$, which follows from the analysis in \cite{BBdF, DIK}). This gives
\begin{align}\label{asymp of u and q update}
u(s)=\sqrt{-s/2}+\bigO(s^{-5/2}), \qquad q(s)=-\frac{s^2}{4}+\frac{1}{8s}+\mathcal O(s^{-5/2}),\qquad \mbox{as } s \to - \infty.
\end{align}
Inverting the transformations $\Psi \to \tilde{\Psi} \to \tilde{\Psi}_{\mathrm{b}} \to \tilde{\Psi}_{\mathrm{c}} \to \tilde{R}$, we obtain 
\begin{align*}
\Psi(\zeta) & = \tilde{\Psi}_{\mathrm{c}}(\tfrac{\zeta	}{|s|^{1/2}})e^{|s|^{3/2}g(\zeta/|s|^{1/2})\sigma_{3}} = \tilde{R}(\tfrac{\zeta	}{|s|^{1/2}}) Q^{\infty}(\tfrac{\zeta}{|s|^{1/2}})e^{|s|^{3/2}g(\zeta/|s|^{1/2})\sigma_{3}}
\end{align*}
for all $|\zeta|\geq \sqrt{|s|}$. By \eqref{lol4}, we thus have
\begin{equation}\label{eq:Psi32}
\Psi(\zeta)  =\left(I +\bigO\left(\frac{1}{|s|^{3/2}\zeta}\right)\right)Q^\infty(\tfrac{\zeta}{|s|^{1/2}})e^{\frac{4}{3} (\zeta^2-s/2)^{3/2}\sigma_3},\qquad \mbox{as } s\to -\infty
\end{equation} 
uniformly for $|\zeta|\geq \sqrt{|s|}$. By expanding $(\zeta^2-s/2)^{3/2}$ and $Q^{\infty}(\frac{\zeta}{|s|^{1/2}})$ as $\zeta/|s|^{1/2} \to \infty$, and by using the asymptotics \eqref{asymp of u and q update} for $u(s)$ and $q(s)$, we can also rewrite \eqref{eq:Psi32} as
\begin{equation}\label{eq:Psi33}
\Psi(\zeta)  =\Bigg(I + \frac{1}{\zeta}\begin{pmatrix}
-\frac{1}{2}q(s) & -\frac{1}{2}u(s) \\[0.1cm] \frac{1}{2}u(s) & \frac{1}{2}q(s)
\end{pmatrix} +\bigO\left(\frac{1}{|s|^{3/2}\zeta}\right) +\mathcal O\left(\frac{s^4}{\zeta^2}\right)\Bigg) e^{ (\frac{4}{3}\zeta^{3}-s\zeta )\sigma_{3}},
\end{equation} 
uniformly as $\zeta\to\infty$ and $s\to -\infty$ such that $s^2/\zeta\to 0$. In fact, the above asymptotics can be shown to hold as $s\to \infty$ such as $|\arg s-\pi|<\delta$, see e.g. \cite[Chapter 11]{FIKN}.

\subsection{Local parametrix near $x^*$}
We need to construct a local parametrix $P$ near $x^*$ satisfying the following conditions. We denote $\mathcal D$ for the open disk centered at $x^*$ with sufficiently small radius $\eta>0$.

\subsubsection*{RH problem for $P$}
\begin{itemize}
\item[(a)] $P:\mathcal{D}\setminus\left(\Sigma_0\cup\Sigma_1\right)\to \mathbb C^{2\times 2}$ is analytic.
\item[(b)] $P_+(z)=P_-(z)J_{T}(z)$ for $z\in (\Sigma_0\cup\Sigma_1)\cap\mathcal D$.
\item[(c)]$P(z)=I+\bigO(N^{-1/3})$ as $N\to +\infty$ uniformly for $z\in \partial{\mathcal D}$.
\end{itemize}
We construct $P$ of the form
\begin{align}\label{def of P PII}
P(z) = \left(\frac{(1-az)^{\epsilon}}{z}\right)^{-\frac{\sigma_{3}}{2}} \Psi \big(N^{\frac{1}{3}} f(z); N^{\frac{2}{3}}s(z;\kappa)\big) e^{\frac{N}{2}(\phi(z;\kappa)-\phi(x^*;\kappa))\sigma_{3}}\left(\frac{(1-az)^{\epsilon}}{z}\right)^{\frac{\sigma_{3}}{2}}.
\end{align}
Here $f$ is a conformal map in a sufficiently small disk $\mathcal D:=\{z\in\mathbb C:|z-x^*|<\eta\}$, given by
\begin{align*}
f(z) & = \bigg(\frac{3}{8}\bigg)^{1/3}(\phi^*- \phi(z;\kappa_2))^{1/3}. 
\end{align*}
It satisfies
\begin{align}\label{expansion of f PII}
f(z) = f^*(z-x^*)\big(1+\bigO(z-x^*)\big), \qquad \mbox{as } z\to x^*,
\end{align}
by \eqref{eq:expandphi}, with 
\begin{equation}\label{def:f0}
f^*=\left(-\frac{1}{16}\phi'''(x^*;\kappa_2)\right)^{1/3}>0.
\end{equation} 
Note that $f$ (and therefore also $f^*$) is independent of $\kappa$. We choose, a posteriori, the jump contours $\Sigma_0$ and $\Sigma_1$ for $T$ locally near $x^*$ in such a way that $z\mapsto N^{\frac{1}{3}}f(z)$ maps $(\Sigma_0\cup\Sigma_1)\cap\mathcal D$ to a subset of the jump contour $\Sigma_\Psi$ for $\Psi$. By this choice, $P$ has exactly the same jump relation as $T$ on $(\Sigma_0\cup\Sigma_1)\cap\mathcal D$, such that condition (b) of the RH problem for $P$ is satisfied.
The function $s$ is defined by
\begin{align}\label{def:s}
s(z;\kappa) = \frac{\phi(z;\kappa)-\phi(x^*;\kappa)-(\phi(z;\kappa_{2})-\phi^*)}{2f(z)}=(\kappa-\kappa_2)\frac{\log z - \log x^*}{2f(z)}.
\end{align}
Note that $\frac{s(z;\kappa)}{\kappa-\kappa_{2}}$ is independent of $\kappa$ and satisfies
\begin{align}\label{asymp:s}
\frac{s(z;\kappa)}{\kappa-\kappa_{2}} = s^* + \bigO\big( z-x^* \big), \qquad \mbox{as } z \to x^*,\qquad s^*=\frac{1}{2x^* f^*}.
\end{align}
Hence, using also \eqref{def:s}, we can (and do) choose $\eta$ sufficiently small so that \[\frac{9}{10}{s^*}|\kappa-\kappa_{2}| \leq |s(z;\kappa)|\leq \frac{11}{10}{s^*}|\kappa-\kappa_{2}|\quad\mbox{ for all $z\in \partial \mathcal{D}$.}\]

As $N\to+\infty$, the RH problem for $\Psi$, more specifically \eqref{eq:Psiasymp} together with the strengthened versions \eqref{eq:Psiasymp2}-\eqref{lol5} as $s\to +\infty$ and \eqref{eq:Psi33} as $s\to -\infty$, implies that
\begin{multline*}
P(z) = I + \frac{1}{N^{\frac{1}{3}} f(z)} \left(\frac{(1-az)^{\epsilon}}{z}\right)^{-\frac{\sigma_{3}}{2}} \begin{pmatrix}
-\frac{1}{2}q(N^{\frac{2}{3}}s(z;\kappa)) & -\frac{1}{2}u(N^{\frac{2}{3}}s(z;\kappa)) \\[0.1cm] \frac{1}{2}u(N^{\frac{2}{3}}s(z;\kappa)) & \frac{1}{2}q(N^{\frac{2}{3}}s(z;\kappa))
\end{pmatrix} \left(\frac{(1-az)^{\epsilon}}{z}\right)^{\frac{\sigma_{3}}{2}} \\ 
+1_{\{|\kappa-\kappa_2|\leq N^{-2/3}\}}\bigO(N^{-2/3}) +1_{\{\kappa<\kappa_2\}}\bigO\Big(N^2 s(z;\kappa)^{4}\Big)+1_{\{\kappa\geq\kappa_2\}}\mathcal O(N^{-\frac{1}{2}}s(z;\kappa)^{1/4}e^{-\frac{2}{3}Ns(z;\kappa)^{3/2}}),
\end{multline*}
as $N\to\infty$, which can be rewritten as
\begin{multline}\label{expansion on dDz0 PII}
P(z) = I + \frac{1}{N^{\frac{1}{3}} f(z)} \left(\frac{(1-az)^{\epsilon}}{z}\right)^{-\frac{\sigma_{3}}{2}} \begin{pmatrix}
-\frac{1}{2}q(N^{\frac{2}{3}}s(z;\kappa)) & -\frac{1}{2}u(N^{\frac{2}{3}}s(z;\kappa)) \\[0.1cm] \frac{1}{2}u(N^{\frac{2}{3}}s(z;\kappa)) & \frac{1}{2}q(N^{\frac{2}{3}}s(z;\kappa))
\end{pmatrix} \left(\frac{(1-az)^{\epsilon}}{z}\right)^{\frac{\sigma_{3}}{2}} \\ 
+1_{\{|\kappa-\kappa_2|\leq N^{-2/3}\}}\bigO(N^{-2/3})
+1_{\{\kappa<\kappa_2\}}\bigO\Big(N^2 (\kappa_2-\kappa)^4\Big)+1_{\{\kappa\geq\kappa_2\}}\mathcal O\bigg(\frac{(\kappa-\kappa_{2})^{1/4}}{N^{1/2}}e^{-c'N(\kappa-\kappa_{2})^{3/2}}\bigg).
\end{multline}
for any fixed $0<c'<\frac{2}{3}(s^*)^{3/2}$. This expansion is uniform in $z\in \partial \mathcal{D}$ and the error terms are uniformly small for  $\frac{a^2}{1+a^2}(1+\delta)\leq {\mu}\leq 1-\delta$ and for $\kappa\in\left[\kappa_2-\frac{1}{N^{\mathfrak{c}}},\kappa_2+\delta\right]$ for $\delta>0$ sufficiently small and for $\mathfrak{c}>1/2$.

\begin{figure}
\begin{center}
\begin{tikzpicture}[master]
\node at (0,0) {\includegraphics[width=7cm]{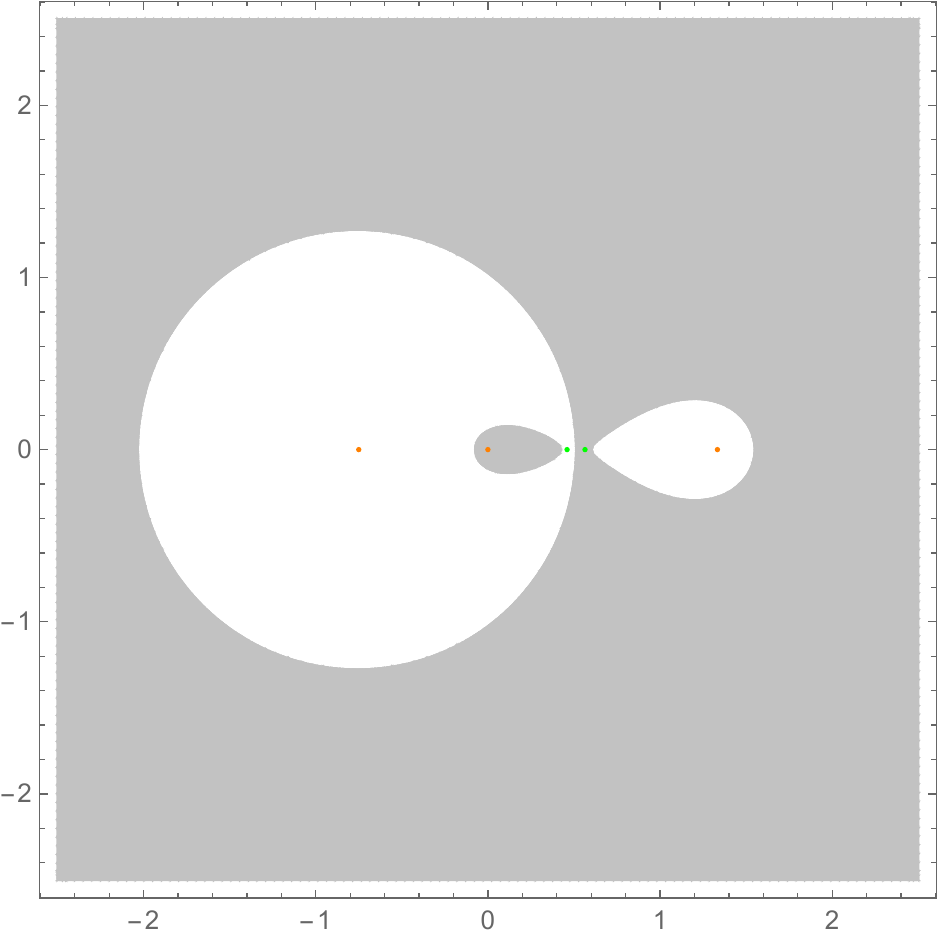}};
\node at (0,3.8) {$\kappa>\kappa_{2}$ ($s>0$)};
\draw (0.8,0.13) circle (0.3);
\draw[black,arrows={-Triangle[length=0.12cm,width=0.08cm]}]
(0.88,0.428) --  ++(00:0.001);

\draw ($(0.8,0.13)+(60:0.3)$) to [out=60, in=0]
(-1,2) to [out=180, in=90] (-2.8,0.13) to [out=-90, in=180] (-1,-1.74) to [out=0, in=-60] ($(0.8,0.13)+(-60:0.3)$);
\draw[black,arrows={-Triangle[length=0.12cm,width=0.08cm]}]
(-1,2) --  ++(180:0.001);

\draw ($(0.8,0.13)+(120:0.3)$) to [out=120, in=0]
(-1,1) to [out=180, in=90] (-1.8,0.13) to [out=-90, in=180] (-1,-0.74) to [out=0, in=-120] ($(0.8,0.13)+(-120:0.3)$);
\draw[black,arrows={-Triangle[length=0.12cm,width=0.08cm]}]
(-1,1) --  ++(180:0.001);
\end{tikzpicture}
\begin{tikzpicture}[slave]
\node at (0,0) {\includegraphics[width=7cm]{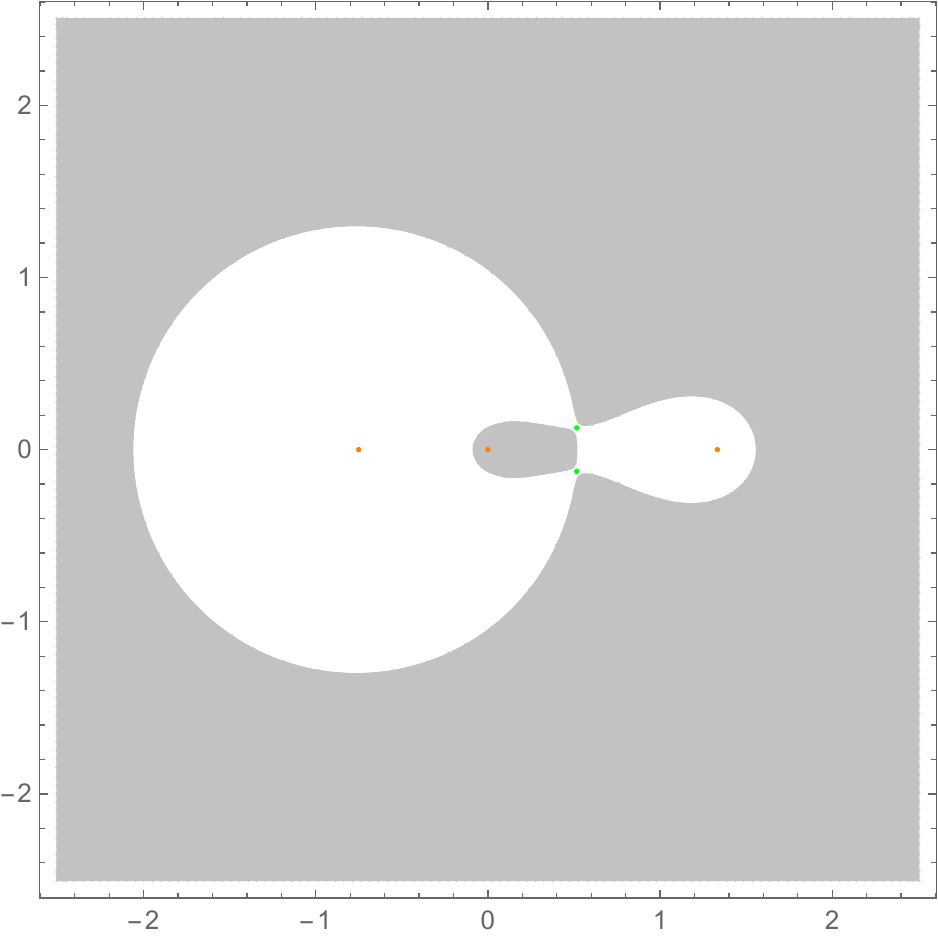}};
\node at (0,3.8) {$\kappa<\kappa_{2}$ ($s<0$)};
\draw (0.8,0.13) circle (0.3);
\draw[black,arrows={-Triangle[length=0.12cm,width=0.08cm]}]
(0.88,0.428) --  ++(00:0.001);

\draw ($(0.8,0.13)+(60:0.3)$) to [out=60, in=0]
(-1,2) to [out=180, in=90] (-2.8,0.13) to [out=-90, in=180] (-1,-1.74) to [out=0, in=-60] ($(0.8,0.13)+(-60:0.3)$);
\draw[black,arrows={-Triangle[length=0.12cm,width=0.08cm]}]
(-1,2) --  ++(180:0.001);

\draw ($(0.8,0.13)+(120:0.3)$) to [out=120, in=0]
(-1,1) to [out=180, in=90] (-1.8,0.13) to [out=-90, in=180] (-1,-0.74) to [out=0, in=-120] ($(0.8,0.13)+(-120:0.3)$);
\draw[black,arrows={-Triangle[length=0.12cm,width=0.08cm]}]
(-1,1) --  ++(180:0.001);
\end{tikzpicture}
\end{center}
\caption{In all images, $\Sigma_{R}$ is the black oriented contour. The shaded regions correspond to $\{z:\Re \phi(z;\kappa)>h\}$ and the white regions to $\{z:\Re \phi(z;\kappa)<h\}$, where $h=\phi(x^{*};\kappa)$. The orange dots are $-a,0,1/a$, and the green dots are $x_{1}(\kappa),x_{2}(\kappa)$. The jump contour $\Sigma_{R}$ is independent of $s$ and is therefore the same in both images. \label{fig:Small norm PII}}
\end{figure}
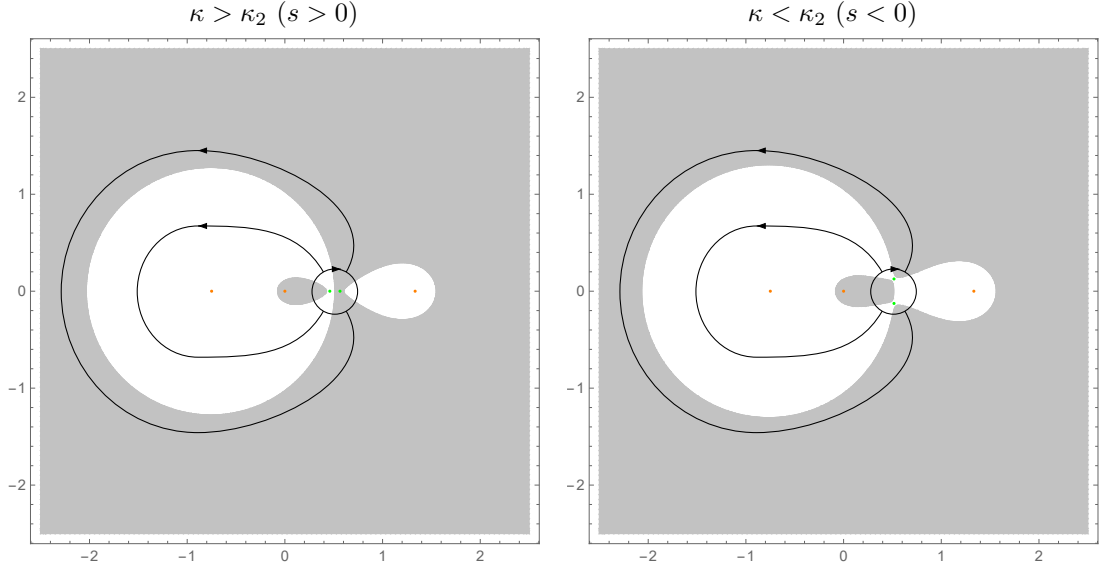

\subsection{Small norm RH problem}
Define
\begin{align}\label{def of R PII}
R(z) = \begin{cases}
T(z), & z \in \C\setminus \overline{\mathcal D}, \\
T(z) P(z)^{-1}, & z \in \mathcal{D}.
\end{cases}
\end{align}
$R$ satisfies the following RH problem.
\subsection*{RH problem for $R$}
\begin{itemize}
\item[(a)] $R:\C\setminus \Sigma_{R}\to \C^{2\times 2}$ is analytic, where
\begin{align*}
\Sigma_{R} = \partial \mathcal{D} \cup \big( (\Sigma_{0}\cup \Sigma_{1}) \setminus \mathcal{D} \big).
\end{align*}
The contour $\Sigma_{R}$ is oriented as shown in Figure \ref{fig:Small norm PII}. In particular, we orient the circle $\partial \mathcal{D}$ in the clockwise direction.
\item[(b)] For $z \in \Sigma_{R}$, we have $R_{+}(z) = R_{-}(z)J_{R}(z)$, where
\begin{align*}
& J_{R}(z) = P(z), & & z \in \partial \mathcal{D}, \\
& J_{R}(z) = J_{T}(z), & & z \in \Sigma_{R} \setminus  \partial \mathcal{D}.
\end{align*}
\item[(c)] $R(z)$ remains bounded as $z$ tends to the points of self-intersection of $\Sigma_{R}$.

\noindent As $z\to \infty$, $R(z) = I + \frac{R_{1}}{z} + \bigO(z^{-2})$ for some $R_{1}$ independent of $z$.
\end{itemize}

Observe that we have the symmetry $R(z) = \overline{R(\overline{z})}$, $z\in \C \setminus \Sigma_{R}$. 

On $\partial \mathcal{D}$, by \eqref{expansion on dDz0 PII} and \eqref{def:s}, we have 
\begin{multline}\label{eq:expansionJR}
J_{R}(z) = I + \frac{J_{R}^{(1)}(z;N^{2/3}s(z;\kappa))}{N^{\frac{1}{3}}}+1_{|\kappa-\kappa_2|\leq N^{-2/3}}\bigO(N^{-2/3})
\\ +1_{\{\kappa<\kappa_2\}}\bigO\Big(N^2(\kappa_2-\kappa)^4\Big)+1_{\{\kappa\geq\kappa_2\}}\mathcal O\bigg(\frac{(\kappa-\kappa_{2})^{1/4}}{N^{1/2}}e^{-c'N (\kappa-\kappa_2)^{3/2}}\bigg), 
\end{multline}
as $N\to\infty$ uniformly for $z\in \partial \mathcal{D}$ and $\kappa\in\left[\kappa_2-\frac{1}{N^{\mathfrak{c}}},\kappa_2+\delta\right]$, for some fixed constants $\delta>0$, $\mathfrak{c}>1/2$ and $c'\in (0,\frac{2}{3}(s^*)^{3/2})$, and for some matrix $J_R^{(1)}(z;s)$. On the rest of the (bounded) contour, by Corollary \ref{prop:contours} (part 2, using also that $\phi(z;\kappa)$ vary smoothly with $\kappa$) and \eqref{eq:JT PII}, we have
\begin{align*}
J_{R}(z) = I+\bigO(e^{-cN}), \qquad \mbox{as } N \to + \infty
\end{align*}
uniformly for $z\in \Sigma_{R}\setminus \partial \mathcal{D}$ and $\kappa\in\left[\kappa_2-\frac{1}{N^{\mathfrak{c}}},\kappa_2+\delta\right]$.
By small-norm theory of RH problems \cite{DeiftZhou, DKMVZ}, we then obtain a similar expansion for $R$:
\begin{multline}\label{large N exp of R PII}
R(z) = I + \frac{R^{(1)}(z;N^{2/3}s(z;\kappa))}{N^{\frac{1}{3}}} 
+1_{|\kappa-\kappa_2|\leq N^{-2/3}}\bigO\bigg(\frac{1}{N^{2/3}(1+|z|)}\bigg)
\\
+1_{\{\kappa<\kappa_2\}}\bigO\bigg(\frac{N^2(\kappa_2-\kappa)^4}{1+|z|}\bigg) +1_{\{\kappa\geq\kappa_2\}}\mathcal O\bigg(\frac{(\kappa-\kappa_{2})^{1/4} e^{-c'N (\kappa-\kappa_2)^{3/2}}}{N^{1/2}(1+|z|)}\bigg),
\end{multline}
as $N\to\infty$ uniformly for $z\in \C \setminus  \Sigma_{R}$ and $\kappa\in\left[\kappa_2-\frac{1}{N^{\mathfrak{c}}},\kappa_2+\delta\right]$.
By substituting \eqref{eq:expansionJR} and \eqref{large N exp of R PII} in the jump relation $R_+=R_-J_R$ valid on $\partial D$  we obtain the relation (omitting the dependence on $N^{2/3}s(z;\kappa)$ in our notations)
\[R_+^{(1)}(z)-R_-^{(1)}(z)=J_R^{(1)}(z),\qquad z\in\partial\mathcal D.\]
Since $R^{(1)}(z)\to 0$ as $z\to\infty$, this relation allows us to compute $R^{(1)}$ explicitly as the Cauchy integral
\begin{align*}
R^{(1)}(z) = \frac{1}{2\pi i}\int_{\partial \mathcal{D} } \frac{J_{R}^{(1)}(s)}{s-z}ds,
\end{align*}
where we recall that the orientation of the disk is clockwise. Therefore, for $z \in \C \setminus \overline{\mathcal{D}}$, we have
\begin{align*}
R^{(1)}(z) = & \, \frac{1}{z-x^*} \mbox{Res}\bigg( J_{R}^{(1)}(s); s=x^* \bigg).
\end{align*}
The above residue can be computed explicitly using \eqref{asymp:s}, \eqref{expansion on dDz0 PII} and \eqref{expansion of f PII}:
\begin{align}\label{eq:R1}
& \mbox{Res}\bigg( J_{R}^{(1)}(s); s=x^* \bigg) = \frac{1}{f_{0}}\begin{pmatrix}
-\frac{1}{2}q(N^{\frac{2}{3}}(\kappa-\kappa_{2}){s}^*) & -\frac{1}{2}u(N^{\frac{2}{3}}(\kappa-\kappa_{2}){s}^*) \frac{x^*}{(1-ax^*)^{\epsilon}} \\[0.1cm] \frac{1}{2}u(N^{\frac{2}{3}}(\kappa-\kappa_{2}){s}^*)\frac{(1-ax^*)^{\epsilon}}{x^*} & \frac{1}{2}q(N^{\frac{2}{3}}(\kappa-\kappa_{2}){s}^*)
\end{pmatrix}.
\end{align}
We now have all the ingredients needed to compute the large $N$ asymptotics for $F_N^{m,k}(a;\epsilon)$.

\begin{proposition}[Ratio asymptotics in the Tracy-Widom region]\label{prop:ratio asymp in TW region}
As $N\to\infty$, we have
\begin{multline*}
\log F_N^{m,k}(a;\epsilon)-\log F_N^{m,k+1}(a;\epsilon)
=\frac{s^*}{N^{\frac{1}{3}}} q(N^{\frac{2}{3}}(\kappa-\kappa_{2})s^*)
+1_{\{|\kappa-\kappa_2|\leq N^{-2/3}\}}\bigO(N^{-2/3})\\
+1_{\{{\kappa<\kappa_2}\}}\bigO\Big({N^{2}(\kappa_2-\kappa)^4}\Big) +1_{\{{\kappa\geq \kappa_2}\}}\mathcal O(N^{-1/2}e^{-c'{N(\kappa-\kappa_2)^{3/2}}}),
\end{multline*}
uniformly for  $\frac{a^2}{1+a^2}(1+\delta)\leq {\mu}\leq 1-\delta$ and for $\kappa_{2}-N^{-\mathfrak{c}}\leq {\kappa}\leq \kappa_{2}+\delta$ with $\mathfrak{c}>1/2$, $c' \in (0,\frac{2}{3}(s^*)^{3/2})$ fixed and $\delta>0$ fixed and sufficiently small.
\end{proposition}
\begin{proof}
By \eqref{def:T} (with $h=\phi(x^{*};\kappa)$) and \eqref{def of R PII}, for $z=0$, we have
\begin{align*}
Y(z) = Y(z;k) = e^{-\frac{N\phi(x^*;\kappa)}{2}\sigma_{3}}R(z) e^{\frac{N\phi(x^*;\kappa)}{2}\sigma_{3}}.
\end{align*} 
Using also \eqref{large N exp of R PII}, \eqref{eq:R1} and the fact that for any $0<c''<c'$, there exists $N_{0}$ such that
\begin{align}\label{lol6}
x^{1/4}e^{-c'Nx^{3/2}} \leq e^{-c''Nx^{3/2}}, \qquad \mbox{for all } x \geq 0, \; N \geq N_{0},
\end{align}
we obtain (after renaming $c''$ as $c'$)
\begin{multline*}
Y_{22}(0;k)  = R_{22}(0;k) = 1 - \frac{1}{N^{\frac{1}{3}}} \frac{1}{2f_{0}x^*}q(N^{-\frac{1}{3}}(k-N\kappa_{2})s^*) +1_{\{|\kappa-\kappa_2|\leq N^{-2/3}\}}\bigO(N^{-2/3})\\
+1_{\{{\kappa<}\kappa_2\}}\bigO\Big({N^{2}(\kappa_2-\kappa})^4\Big) +1_{\{{\kappa\geq }\kappa_2\}}\mathcal O\bigg(\frac{e^{-c'{N(\kappa-\kappa_2)}^{3/2}}}{N^{1/2}}\bigg).
\end{multline*}
By \eqref{eq:ratioid2} and \eqref{asymp:s}, the result follows.
\end{proof}
\begin{corollary}
As $N\to + \infty$, we have
\begin{align*}
\log F_N^{m,k}(a;\epsilon)-\log F_N(a) &= \log F^{\rm TW}\left(s^*{N^{2/3}(\kappa-\kappa_2) }\right)+\mathcal{E}_{N},
\end{align*}
uniformly for  $\frac{a^2}{1+a^2}(1+\delta)\leq {\mu}\leq 1-\delta$ and for ${\kappa_{2}-N^{\mathfrak{c}-1}\leq \kappa\leq \kappa_{2}+\delta}$ with $\mathfrak{c}<2/5$ fixed and $\delta>0$ fixed and sufficiently small, where $\mathcal{E}_{N}$ satisfies
\begin{align}\label{estimate for Epsilon N}
\mathcal{E}_{N} = \begin{cases}
\bigO(N^{-1/3}), & \mbox{if } {N^{\frac{2}{3}}(\kappa-\kappa_{2})} \mbox{ is bounded}, \\
\bigO\Big(\frac{e^{-c'{N(\kappa-\kappa_{2})^{3/2}}}}{N^{1/2}(\kappa-\kappa_{2})^{1/4}}\Big), & \mbox{if } {N^{\frac{2}{3}}(\kappa-\kappa_{2})} \to + \infty, \\
\bigO({N^{3}(\kappa_{2}-\kappa)}^{5}), & \mbox{if } {N^{\frac{2}{3}}(\kappa-\kappa_{2})} \to - \infty,
\end{cases}
\end{align}
and where $c' \in (0,\frac{2}{3}(s^*)^{3/2})$ is arbitrary but fixed.
\end{corollary}
\begin{proof}
We take the sum of $\log F_N^{m,j}(a;\epsilon)-\log F_N^{m,j+1}(a;\epsilon)$ for $j=k, k+1,\ldots, k_{\rm max}-1$, where we choose $N(\kappa_2+\delta/2)<k_{\rm max}<N(\kappa_2+\delta)$. By Proposition \ref{prop:ratio asymp in TW region}, we get
\begin{multline*}
\log F_N^{m,k}(a;\epsilon)-\log F_N^{m,k_{\rm max}}(a;\epsilon)= \sum_{j=k}^{k_{\rm max}-1} \bigg(\frac{s^*}{N^{\frac{1}{3}}}q(N^{-\frac{1}{3}}(j-N\kappa_{2})s^*) +  1_{\{|k-\kappa_2 N|\leq N^{1/3}\}}\bigO(N^{-2/3}) \\
+ 1_{\{k<N\kappa_2\}}\bigO\Big(N^{-2}(N\kappa_2-k)^4\Big) + 1_{\{k\geq N\kappa_2\}}\mathcal O(N^{-1/2}e^{-c'N^{-1/2}(k-N\kappa_2)^{3/2}})\bigg),
\end{multline*}
as $N\to\infty$.
We recognize the first part of the sum as a Riemann sum for the (smooth and fast decaying at $+\infty$) function $q$:
\begin{align*}
& \frac{s^*}{N^{\frac{1}{3}}} \sum_{j=k}^{k_{\rm max}-1} q(N^{-\frac{1}{3}}(j-N\kappa_{2})s^*) \\
& = \int_{s^*N^{-1/3}(k-N\kappa_2)}^{s^*N^{-1/3}(k_{\rm max}-N\kappa_2)} q(s)ds + \bigO \bigg( \frac{q({N^{\frac{2}{3}}(\kappa-\kappa_{2})}s^*)}{N^{\frac{1}{3}}} \bigg), \qquad \mbox{as } N \to + \infty.
\end{align*}
By \eqref{lol5} and \eqref{asymp of u and q}, as $N\to \infty$, we have
\begin{align*}
\frac{q({N^{\frac{2}{3}}(\kappa-\kappa_{2})}s^*)}{N^{\frac{1}{3}}} = \begin{cases}
\bigO(N^{-1/3}), & \mbox{if } {N^{\frac{2}{3}}(\kappa-\kappa_{2})} \mbox{ is bounded}, \\
\bigO\Big(\frac{e^{-\frac{2}{3}(s^*)^{3/2}{N(\kappa-\kappa_{2})}^{3/2}}}{{N^{1/2}(\kappa-\kappa_{2})^{1/4}}}\Big), & \mbox{if } {N^{\frac{2}{3}}(\kappa-\kappa_{2})} \to + \infty, \\
\bigO({N (\kappa-\kappa_{2})^{2}}), & \mbox{if } {N^{\frac{2}{3}}(\kappa-\kappa_{2})} \to - \infty.
\end{cases}
\end{align*}
Observe also that as $N\to + \infty$, 
\begin{align*}
& \sum_{j=k}^{k_{\rm max}-1} 1_{\{k<N\kappa_2\}}\; N^{-2}(N\kappa_2-k)^4 = \begin{cases}
\bigO(N^{-1/3}), & \mbox{if } {N^{\frac{2}{3}}(\kappa-\kappa_{2})} \mbox{ is bounded}, \\
\bigO({N^{3}(\kappa_{2}-\kappa)^{5}}), & \mbox{if } {N^{\frac{2}{3}}(\kappa-\kappa_{2})} \to - \infty,
\end{cases}, \\
& \sum_{j=k}^{k_{\rm max}-1} 1_{\{k\geq N\kappa_2\}} N^{-1/2}e^{-c'{N(\kappa-\kappa_2)^{3/2}}} = \begin{cases}
\bigO(N^{-1/2}), & \mbox{if } {N^{\frac{2}{3}}(\kappa-\kappa_{2})} \mbox{ is bounded} \\
\bigO\Big(\frac{e^{-c'N(\kappa-\kappa_{2})^{3/2}}}{N^{1/2}}\Big), & \mbox{if } {N^{\frac{2}{3}}(\kappa-\kappa_{2})} \to + \infty. \\
\end{cases}
\end{align*}
Using also that $\log F_N^{m,k_{\rm max}}(a;\epsilon)- \log F_{N}(a)$ is exponentially small as $N\to\infty$ by Theorem \ref{thm:frozen}, we obtain
\begin{align*}
\log F_N^{m,k}(a;\epsilon)-\log F_{N}(a)&=\int_{s^*{N^{2/3}(\kappa-\kappa_2})}^{+\infty} q(s)ds + \mathcal{E}_{N}, \qquad \mbox{as } N \to + \infty,
\end{align*}
where $\mathcal{E}_{N}$ satisfies \eqref{estimate for Epsilon N}. Since $u^2(s)$ is the derivative of $q(s)$ (with $u(s)$ the Hastings-McLeod solution of Painlev\'e II), we can integrate by parts to obtain
\begin{align*}
\log F_N^{m,k}(a;\epsilon)-\log F_{N}(a)
&=-\int_{s^*{N^{2/3}(\kappa-\kappa_2)}}^{+\infty} (s-s^*N^{-1/3}(k-N\kappa_2))u^2(s)ds + \mathcal{E}_{N}\\
&=\log F^{\rm TW}\left(s^*N^{-1/3}(k-N\kappa_2)\right)+\mathcal{E}_{N},\qquad N\to\infty.
\end{align*}
\end{proof}
We finally observe that 
$s^*=(c^*)^{-1/3}$, by \eqref{def:phi0x0},     \eqref{def:phi}, \eqref{def:f0} and \eqref{asymp:s}.
This completes the proof of Theorem \ref{thm:boundary}.

\section{Almost maximal removed corner}\label{sec:max}
In this section, $k\in\{1,\ldots, M\}$ for some fixed $M\in \N$, and $\mu=m/N\in (0,1)$. 
\subsection{The RH problem for $Y$}
We start from the RH problem for $Y$ from Section \ref{section:RHP for Y}, and we deform the contour $\Sigma_{0}$ so that it coincides with $\Sigma_{1}$ (here we use the fact that $\Sigma_{0}$ is allowed to cross $-a$). In this way, $Y$ satisfies the following properties. Note that the jump matrix in \eqref{lol27} below is the product of the jump matrices on $\Sigma_1$ and $\Sigma_0$ from \eqref{eq:JY2}.

\subsubsection*{RH problem for $Y$}
\begin{itemize}
\item[(a)] $Y:\mathbb C\setminus \Sigma_1 \to \mathbb C^{2\times 2}$ is analytic, with $\Sigma_1$ positively oriented and enclosing $0,-a$ without enclosing $1/a$.
\item[(b)] $Y_+(z)=Y_-(z)J_Y(z)$ for $z\in \Sigma_1$, with
\begin{align}\label{lol27}
& J_Y(z)=
\begin{pmatrix}
0 & \frac{z}{(1-az)^{\epsilon}}z^{-k}e^{-N\phi(z;0)} \\
- \frac{(1-az)^{\epsilon}}{z}z^{k}e^{N\phi(z;0)} & 1
\end{pmatrix}, \qquad z\in\Sigma_1.
\end{align}
\item[(c)]{As $z\to \infty$, $Y(z) = I +\bigO(z^{-1})$}.
\end{itemize}

\subsection{Transformation $Y\mapsto T$}
With our next transformation, we reshuffle and renormalize the entries of the jump matrix by defining
\begin{align}\label{def of T birth}
T(z) = \begin{cases}
\ds -\sigma_{1} Y(z) \sigma_{3}(1-az)^{-(1-\mu)N\sigma_{3}}, & \mbox{inside } \Sigma_{1}, \\[0.1cm]
\ds \sigma_{1} Y(z) \sigma_{1}\Big(\frac{z+a}{z}\Big)^{\mu N \sigma_{3}}, & \mbox{outside } \Sigma_{1}.
\end{cases}
\end{align}
$T$ satisfies the following RH problem.
\subsubsection*{RH problem for $T$}
\begin{itemize}
\item[(a)] $T:\mathbb C\setminus \Sigma_1 \to \mathbb C^{2\times 2}$ is analytic.
\item[(b)] $T_+(z)=T_-(z)J_T(z)$ for $z\in \Sigma_1$, with
\begin{align*}
& J_T(z)=
\begin{pmatrix}
\frac{(1-az)^{\epsilon}}{z}z^{k} & e^{N\varphi_{0}(z)} \\
0 & \frac{z}{(1-az)^{\epsilon}}z^{-k}
\end{pmatrix}, \qquad z\in\Sigma_1,
\end{align*}
where
\begin{align*}
\varphi_{0}(z) := (1-\mu)\log(1-az)-\mu \log(z+a) + \mu \log z
\end{align*}
and the principal branch is used for the logarithms.
\item[(c)]As $z\to \infty$, $T(z) = I +\bigO(z^{-1})$.
\end{itemize} 

Note that $\varphi_0(z)$ is not equal to $\phi(z;0)$ defined in \eqref{def:phi}: indeed the signs of two terms are different. The new phase function $\varphi_{0}$ has two saddle points, which are given by
\begin{align*}
x_0=\frac{-a-\sqrt{a^2+4\mu(1-\mu)}}{2(1-\mu)}\in (-\infty,-a), \qquad \gamma_0=\frac{-a+\sqrt{a^2+4\mu(1-\mu)}}{2(1-\mu)} \in (0,\tfrac{1}{a}).
\end{align*}
The relevant saddle point for our analysis is $x_{0}$. Since $\varphi_{0}''(x_{0})>0$, there are four vertical trajectories emanating from $x_{0}$, making angles $\pm \frac{\pi}{4},\pm \frac{3\pi}{4}$ with respect to $\R$. Since
\begin{align}\label{blow up of varphi0}
\Re \varphi_{0}(\infty) = \Re \varphi_{0}(-a) = +\infty, \qquad \Re \varphi_{0}(0) = \Re \varphi_{0}(\tfrac{1}{a}) = -\infty,
\end{align}
the two trajectories (say $\Gamma_{1},\Gamma_{2}$) going into the upper half-plane cannot connect with $+\infty,-a,0,\frac{1}{a}$. By the max/min principle for harmonic functions, these trajectories can neither form closed loops in the upper half-plane nor intersect each other. Hence they must intersect $\R$ at two distinct points. A direct inspection of $x \to \Re \varphi_{0}(x)$ for $x\in \R$ then shows that one curve must intersect $(-a,0)$ and the other one must intersect $(\frac{1}{a},+\infty)$. The symmetry $\varphi_{0}(z) = \overline{\varphi_{0}(\overline{z})}$ then implies that $\Gamma_{1},\Gamma_{2}$ are symmetric with respect to $\R$ (hence they must be closed loops). The graph of $x \mapsto \Re \varphi_{0}(x)$ for $x\in \R$ also shows that
\begin{align*} 
\{x\in \R\setminus \{x_{0}\}: \Re \varphi_{0}(x) =  \varphi_{0}(x_{0}) \}
\end{align*}
consists of exactly two points; hence, by the max/min principle for harmonic functions, the level set
\begin{align*}
\{z\in \C: \Re \varphi_{0}(x) = \varphi_{0}(x_{0}) \}
\end{align*}
consists exclusively of $\Gamma_{1}\cup \Gamma_{2}$, see also Figure \ref{fig:level curve varphi0}. The set $\Gamma_{1}\cup \Gamma_{2}$ delimits three regions, and the sign of $\varphi_{0}$ within each region is determined by \eqref{blow up of varphi0}.

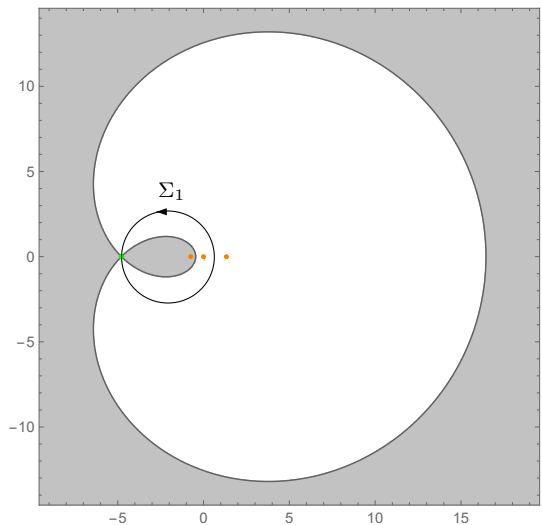
\begin{figure}
\begin{center}
\begin{tikzpicture}[master]
\node at (0,0) {\includegraphics[width=7cm]{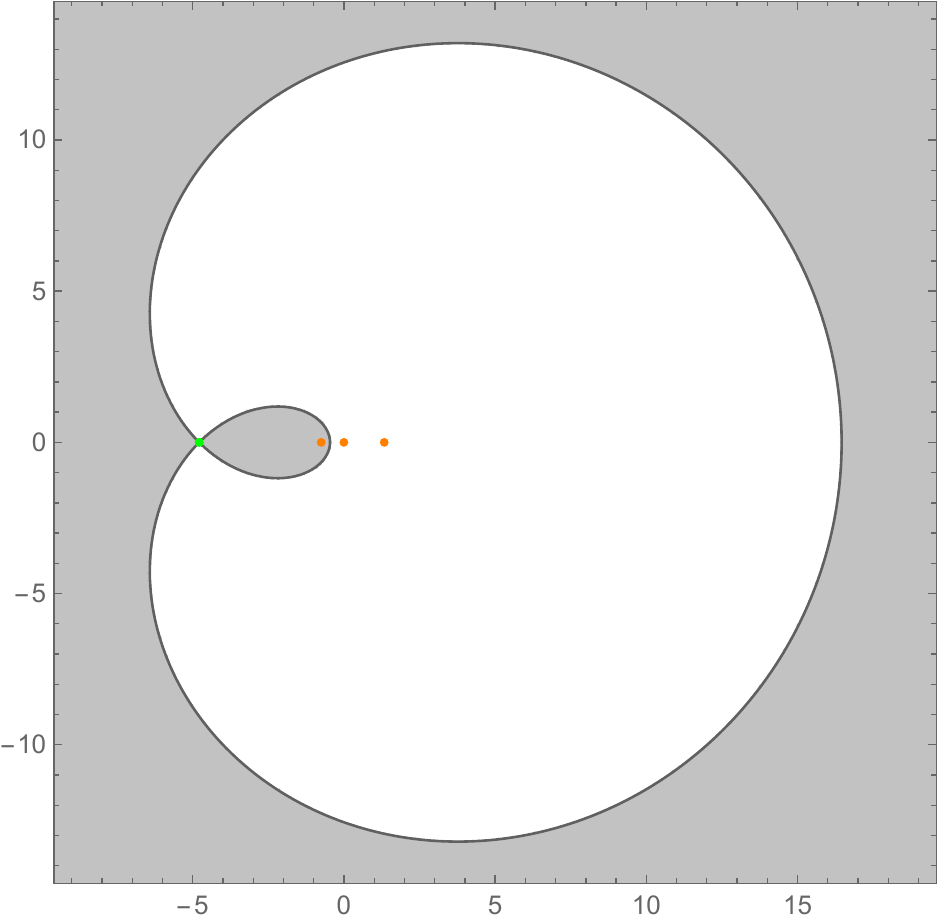}};

\draw (-1.4,0.13) circle (0.61cm);

\draw[black,arrows={-Triangle[length=0.135cm,width=0.09cm]}]
(-1.55,0.73) --  ++(180:0.001);
\node at (-1.35,1) {\footnotesize $\Sigma_{1}$};
\end{tikzpicture}
\end{center}
\caption{\label{fig:level curve varphi0}
The shaded regions correspond to $\{z:\Re \varphi_{0}(z)>h\}$ and the white regions to $\{z:\Re \varphi_{0}(z)<h\}$, where $h=\varphi_{0}(x_{0})$. The orange dots are $-a,0,1/a$, and the green dot is $x_{0}$. The figure is drawn for $a=0.75$ and $\mu=0.808 \in (\frac{a^{2}}{1+a^{2}},1)$.}
\end{figure}
 
\medskip The above discussion shows that we can deform the contour $\Sigma_{1}$ of the RH problem for $T$ so that 
\begin{align}\label{lol15}
\Sigma_{1}\setminus \{x_{0}\} \subset \{z\in \C: \Re \varphi_{0}(x) < \varphi_{0}(x_{0}) \},
\end{align}
see again Figure \ref{fig:level curve varphi0}.

\subsection{Transformation $T\mapsto S$}
Next, in order to create jump matrices close to the identity matrix, we define
\begin{align}\label{def of S birth}
S(z) = e^{-\frac{N}{2}\varphi_{0}(x_{0})\sigma_{3}}T(z)e^{\frac{N}{2}\varphi_{0}(x_{0})\sigma_{3}}.
\end{align}
$S$ satisfies the following RH problem.
\subsubsection*{RH problem for $S$}
\begin{itemize}
\item[(a)] $S:\mathbb C\setminus \Sigma_1 \to \mathbb C^{2\times 2}$ is analytic.
\item[(b)] $S_+(z)=S_-(z)J_S(z)$ for $z\in \Sigma_1$, with
\begin{align*}
& J_S(z)=
\begin{pmatrix}
\frac{(1-az)^{\epsilon}}{z}z^{k} & e^{N(\varphi_{0}(z)-\varphi_{0}(x_{0}))} \\
0 & \frac{z}{(1-az)^{\epsilon}}z^{-k}
\end{pmatrix}, \qquad z\in\Sigma_1.
\end{align*}
\item[(c)]As $z\to \infty$, $S(z) = I +\bigO(z^{-1})$.
\end{itemize} 
By \eqref{lol15}, the $(1,2)$-entry of $J_S$ is uniformly exponentially small on $\Sigma_{1}$ as $N\to + \infty$, except within a fixed neighborhood of $x_{0}$.

\subsection{Global parametrix}
We now construct a global parametrix $\widehat P^{(\infty)}$ satisfying the following properties.
\subsection*{RH problem for $\widehat P^{(\infty)}$}
\begin{itemize}
\item[(a)] $\widehat P^{(\infty)}:\C\setminus \Sigma_{1} \to \mathbb C^{2\times 2}$ is analytic.
\item[(b)] $\widehat P^{(\infty)}_+(z)=\widehat P^{(\infty)}_-(z) \begin{pmatrix}
\frac{(1-az)^{\epsilon}}{z}z^{k} & 0 \\
0 & \frac{z}{(1-az)^{\epsilon}}z^{-k}
\end{pmatrix} $ for $z\in \Sigma_{1}\setminus \{x_{0}\}$.
\item[(c)]As $z\to \infty$, $\widehat P^{(\infty)}(z)=I+\bigO(z^{-1})$.
\end{itemize}
A solution to the above RH problem is given by
\begin{align}\label{def of Pinf hat}
\widehat P^{(\infty)}(z) = \begin{cases}
\ds \begin{pmatrix}
\frac{(1-az)^{\epsilon}}{z-x_{0}}(z-x_{0})^{k} & 0 \\
0 & \frac{z-x_{0}}{(1-az)^{\epsilon}}(z-x_{0})^{-k}
\end{pmatrix}, & \mbox{if } z \mbox{ is inside } \Sigma_{1}, \\[0.35cm]
\ds \begin{pmatrix}
\frac{(z-x_{0})^{k-1}}{z^{k-1}} & 0 \\
0 & \frac{z^{k-1}}{(z-x_{0})^{k-1}}
\end{pmatrix}, & \mbox{otherwise}.
\end{cases}
\end{align}
Note that $\widehat P^{(\infty)}$ has a singularity at $x_{0}$. 
\subsection{Local parametrix}
We need to construct a local parametrix $P$ near $x_{0}$ satisfying the following conditions. We denote $\mathcal D$ for the open disk centered at $x_{0}$ with small but fixed radius $\eta>0$.

\subsubsection*{RH problem for $P$}
\begin{itemize}
\item[(a)] $P:\mathcal{D}\setminus \Sigma_1 \to \mathbb C^{2\times 2}$ is analytic.
\item[(b)] $P_+(z)=P_-(z)J_{S}(z)$ for $z\in \Sigma_1 \cap\mathcal D$.
\item[(c)]$P(z)=\widehat{P}^{(\infty)}(z)\big(I+\bigO(N^{-1/2})\big)$ as $N\to \infty$ uniformly for $z\in \partial \mathcal D$.
\end{itemize}

We will construct $P$ using the following model RH problem, whose solution can be constructed using Hermite polynomials.
\subsubsection*{RH problem for $\Upsilon$}
\begin{itemize}
\item[(a)] $\Upsilon:\mathbb C\setminus\mathbb R\to\mathbb C^{2\times 2}$ is analytic, where $\R$ is oriented from left to right.
\item[(b)] For $x\in\mathbb R$,
\[\Upsilon_+(x)=\Upsilon_-(x)\begin{pmatrix}1&e^{-x^2/2}\\0&1\end{pmatrix}.\]
\item[(c)] As $\zeta\to\infty$, we have
\begin{align}\label{Upsilon at infty}
\Upsilon(\zeta)=\left(I+\frac{1}{\zeta}\begin{pmatrix}0&\frac{i \, (k-1)!}{\sqrt{2\pi}} \\
\frac{-i\sqrt{2\pi}}{(k-2)!}&0\end{pmatrix}+\mathcal O(\zeta^{-2})\right)\zeta^{(k-1)\sigma_3}.
\end{align}
\end{itemize}
This is a special case of the Fokas-Its-Kitaev RH problem characterizing orthogonal polynomials \cite{FokasItsKitaev}, corresponding to the weight function $e^{-x^{2}/2}$. The relevant orthogonal polynomials are then Hermite polynomials. Writing $p_j$ for the normalized Hermite polynomial of degree $j$ and $\kappa_j=\frac{1}{(2\pi)^{1/4}\sqrt{j!}}$ for its leading coefficient, and the solution $\Upsilon$ can be constructed as follows,
\begin{align*}
\Upsilon(\zeta) = \begin{pmatrix}
\frac{1}{\kappa_{k-1}}p_{k-1}({\zeta})&\frac{1}{2\pi i\kappa_{k-1}}\int_{\mathbb R}p_{k-1}(s)e^{-s^2/2}\frac{ds}{s-\zeta}\\ 
-2\pi i\kappa_{k-2}p_{k-2}({\zeta})& -\kappa_{k-2}\int_{\mathbb R}p_{k-2}(s)e^{-s^2/2}\frac{ds}{s-\zeta} \end{pmatrix}.
\end{align*}
 Then $\widehat{\Upsilon}(z) := \Upsilon(z)e^{-\frac{z^{2}}{4}\sigma_{3}}$ has constant jumps. We construct the local parametrix $P(z)$ for $z\in \mathcal{D}\setminus \Sigma_{1}$ as follows:
\begin{align*}
& P(z) = E(z)\widehat{\Upsilon}(\sqrt{N}f(z))e^{-\frac{N}{2}(\varphi_{0}(z)-\varphi_{0}(x_{0}))\sigma_{3}}\widehat{\mathcal{G}}(z)^{\sigma_{3}}, 
\end{align*}
where $f$ is the conformal map from $\mathcal{D}$ to a neighborhood of $0$ given by
\begin{align}\label{def of f birth}
f(z) = i\sqrt{2(\varphi_{0}(z)-\varphi_{0}(x_{0}))},
\end{align}
and $E$ and $\widehat{\mathcal{G}}$ are defined by
\begin{align}
& E(z) = \widehat{P}^{(\infty)}(z)\widehat{\mathcal{G}}(z)^{-\sigma_{3}}f(z)^{-(k-1)\sigma_{3}}, \label{def of E birth of a cut} \\
& \widehat{\mathcal{G}}(z) = \begin{cases}
(1-az)^{\frac{\epsilon}{2}} i^{-k+1}(-z)^{\frac{k-1}{2}}, & \mbox{if } z \mbox{ is inside } \Sigma_{1}, \\
(1-az)^{-\frac{\epsilon}{2}}i^{k-1}(-z)^{-\frac{k-1}{2}}, & \mbox{if } z \mbox{ is outside } \Sigma_{1}, \label{def of Gcal birth}
\end{cases}
\end{align}
In \eqref{def of Gcal birth}, the principal branch is used for the roots, and in \eqref{def of f birth}, the branch for the root is chosen so that
\begin{align*}
f(z) = i\sqrt{\varphi_{0}''(x_{0})}(z-x_{0})\big( 1+\bigO(z-x_{0}) \big), \qquad \mbox{as } x\to x_{0}.
\end{align*}
We also choose the contour $\Sigma_{1}$ so that $f(\Sigma_{1})\subset \R$. It is easy to check that $E(z)$ has no jumps for $z\in \mathcal{D}$. Moreover, by \eqref{def of Pinf hat}, $E(z)=\bigO(1)$ as $z\to x_{0}$. Hence $E$ is analytic in $\mathcal{D}$. Using \eqref{Upsilon at infty}, we obtain
\begin{align}\label{matching condition for R birth of a cut}
P(z)\widehat{P}^{(\infty)}(z)^{-1} = N^{-\frac{k-1}{2}\sigma_{3}}\bigg\{I + \frac{E(z)}{\sqrt{N}f(z)} 
\begin{pmatrix}
0 & \frac{i \, (k-1)!}{\sqrt{2\pi}} \\
\frac{-i\sqrt{2\pi}}{(k-2)!} & 0
\end{pmatrix}E(z)^{-1} + \bigO(N^{-1})\bigg\}N^{\frac{k-1}{2}\sigma_{3}},
\end{align}
as $N\to \infty$ uniformly for $z\in \partial \mathcal{D}$.

\subsection{Small norm RH problem}

Define
\begin{align}\label{def of RHermite}
R(z) = \begin{cases}
N^{\frac{k-1}{2}\sigma_{3}} S(z)\widehat P^{(\infty)}(z)^{-1} N^{-\frac{k-1}{2}\sigma_{3}}, & z \in \mathbb C\setminus \overline{\mathcal{D}} , \\
N^{\frac{k-1}{2}\sigma_{3}} S(z) P(z)^{-1} N^{-\frac{k-1}{2}\sigma_{3}}, & z \in \mathcal{D}.
\end{cases}
\end{align}
$R$ satisfies the following RH problem.
\subsection*{RH problem for $R$}
\begin{itemize}
\item[(a)] $R:\C\setminus \Sigma_{R}\to \C^{2\times 2}$ is analytic, where
\begin{align*}
\Sigma_{R} = \partial \mathcal{D} \cup \big( \Sigma_1 \setminus \mathcal{D} \big)
\end{align*}
and where $\partial \mathcal{D}$ is oriented in the clockwise direction.
\item[(b)] For $z \in \Sigma_{R}$, we have $R_{+}(z) = R_{-}(z)J_{R}(z)$, where
\begin{align*}
& J_{R}(z) = P(z)\widehat{P}^{(\infty)}(z)^{-1}, & & z \in \partial \mathcal{D}, \\
& J_{R}(z) =\widehat P^{(\infty)}(z) J_{S}(z)\widehat P^{(\infty)}(z)^{-1}, & & z \in \Sigma_{R} \setminus \partial \mathcal{D}.
\end{align*}
\item[(c)] $R(z)$ remains bounded as $z$ tends to the points of self-intersections of $\Sigma_{R}$.

\noindent As $z\to \infty$, $R(z) = I + \frac{R_{1}}{z} + \bigO(z^{-2})$ for some $R_{1}$ independent of $z$.
\end{itemize}
From the matching condition \eqref{matching condition for R birth of a cut}, we obtain
\begin{align*}
J_R(z)=I + \frac{E(z)}{\sqrt{N}f(z)} 
\begin{pmatrix}
0 & \frac{i \, (k-1)!}{\sqrt{2\pi}} \\
\frac{-i\sqrt{2\pi}}{(k-2)!} & 0
\end{pmatrix}E(z)^{-1} + \bigO(N^{-1}),
\end{align*}
for $z\in\partial\mathcal D$.
It follows from standard theory for small norm RH problems \cite{DeiftZhou, DKMVZ} that
\begin{align}\label{eq:asR birth}
R(z)=I+\frac{R^{(1)}(z)}{\sqrt{N}}+\bigO\bigg(\frac{1}{N(1+|z|)}\bigg),\qquad N\to\infty,
\end{align}
uniformly for $z\in \C\setminus \Sigma_{R}$. For $z \in \C \setminus \overline{\mathcal{D}}$, we obtain by residue computation that
\begin{align}\label{Rp1p birth}
R^{(1)}(z) = \frac{1}{z-x_{0}} \frac{1}{\sqrt{\varphi_{0}''(x_{0})}} 
\begin{pmatrix}
0 & E_{11}(x_{0})^{2}\frac{(k-1)!}{\sqrt{2\pi}} \\
\frac{-\sqrt{2\pi}}{(k-2)!E_{11}(x_{0})^{2}} & 0
\end{pmatrix}.
\end{align}
The quantity $E_{11}(x_{0})$ can be computed explicitly using \eqref{def of E birth of a cut}:
\begin{align*}
E_{11}(x_{0}) = \frac{|1-a x_{0}|^{\frac{\epsilon}{2}}}{|x_{0}|^{\frac{k-1}{2}}\varphi_{0}''(x_{0})^{\frac{k-1}{2}}}>0.
\end{align*}

\subsection{Ratio asymptotics and proof of Theorem \ref{thm:birth of a cut}}

\begin{proposition}[Ratio asymptotics for a near-maximal {removed corner}]\label{prop:ratio asymp in max region}
For any fixed $\delta, M>0$, we have
\begin{align}
&\log F_N^{m,k+1}(a;\epsilon)-\log F_N^{m,k}(a;\epsilon)
=N\varphi_{0}(x_{0})-\Big(k-\frac{1}{2}\Big)\log N + \log\big((k-1)!\big)\nonumber \\ &\qquad\qquad\qquad - \Big( k-\frac{1}{2} \Big) \log\big(|x_{0}|^{2}\varphi_{0}''(x_{0})\big) 
 +\epsilon \log(1-ax_{0}) - \frac{1}{2}\log (2\pi)+\bigO(N^{-1/2}), \label{asymp ratio birth}
\end{align}
uniformly for $k\in\{1,\ldots, M\}$ and for $\delta \leq {\mu}\leq 1-\delta$.
\end{proposition}
\begin{proof}
By \eqref{def of T birth}, \eqref{def of S birth}, \eqref{def of Pinf hat}, \eqref{def of RHermite}, we have
\begin{align*}
Y_{22}(0) = T_{12}(0) = e^{N\varphi_{0}(x_{0})}S_{12}(0) = e^{N\varphi_{0}(x_{0})}N^{-k+1}|x_{0}|^{-k+1}R_{12}(0).
\end{align*}
Using the large $N$ asymptotics \eqref{eq:asR birth} of $R$, we obtain
\begin{align*}
\log Y_{22}(0) = N\varphi_{0}(x_{0}) - \Big(k-\frac{1}{2}\Big) \log N  + \log \frac{R_{12}^{(1)}(0)}{|x_{0}|^{k-1}} +\mathcal O(N^{-1/2}).
\end{align*}
Substituting \eqref{Rp1p birth} in the above expansion and using \eqref{eq:ratioid2}, we obtain the claim.
\end{proof}

Let $K>0$ be any fixed number.
Summing \eqref{asymp ratio birth} with $k$ running from $1$ to $K$ yields
\begin{align}
\log \frac{F_N^{m,K+1}(a;\epsilon)}{F_N^{m,1}(a;\epsilon)} & = KN \varphi_{0}(x_{0}) - \frac{K}{2} \log N + \log G({K}+1) - \frac{K^{2}}{2} \log\big( |x_{0}|^{2}\varphi_{0}''(x_{0}) \big) \nonumber \\
&\quad + K \epsilon \log(1-ax_{0}) - \frac{K}{2}\log(2\pi) + \bigO(N^{-1/2}) \label{lol16}
\end{align}
as $N \to +\infty$. By \eqref{eq:formulamirror},
\begin{align}\label{lol17}
\log F_N^{m,1}(a;\epsilon)=\bigg( \frac{N^{2}}{2} + \Big( \frac{1}{2}-m \Big)N + m^{2} - \epsilon \, m \bigg) \log (1+a^2).
\end{align}
Summing \eqref{lol16} and \eqref{lol17} and using $m=\mu N$, we obtain 
\begin{align*}
& \log F_N^{m,K+1}(a;\epsilon) = F_2(\mu)N^{2} + F_1(K,\mu) N - \frac{K^2}{2}\log N + F_0(K,\mu) + \bigO(N^{-1/2})
\end{align*}
as $N\to + \infty$, where
$F_2,F_1,F_0$ are given {by \eqref{def:F1}--\eqref{def:F4}.} This proves Theorem \ref{thm:birth of a cut}.

\section{Large removed corner}
\label{sec:large}

In this section, $\mu=m/N\in\big(0,1\big)$, and $\kappa =k/N\in(0,\kappa_2)$. We observe from \eqref{def:kappa2} that $\kappa_2<\mu$ for $\mu\in(0,1)\setminus\{\frac{a^2}{1+a^2}\}$.

\subsection{Construction and analysis of the $g$-function}
Here, we cannot simply deform the jump contours $\Sigma_0,\Sigma_1$ to arrive at a small-norm RH problem. Instead, we need to construct a $g$-function, with a branch cut on a curve $\mathcal S$ connecting two points $z_0, \overline{z_0}$, depending on $\mu,\kappa$ and to be determined later. It has to satisfy the following properties.

\subsubsection*{RH problem for $g$}
\begin{itemize}
\item[(a)] $g:\C\setminus \mathcal{S} \to \C$ is analytic.
\item[(b)] There exists a constant $\ell \in \C$ such that
\begin{align}\label{lol9}
g_{+}(z) + g_{-}(z) + \ell = \phi(z;\kappa), \qquad z \in \mathcal{S}.
\end{align}
\item[(c)] $g(z) = g_{1}z^{-1} + \bigO(z^{-2})$ as $z\to\infty$, for some $g_{1}\in \C$.
\end{itemize}
Recall the definition of $\phi$ from \eqref{def:phi}.
Since $\phi$ contains logarithmic terms, it is convenient to first solve the RH problem for $g'$ and afterwards construct $g$ by integrating. The derivative of $g$ needs to satisfy the following properties.
\subsubsection*{RH problem for $g'$}
\begin{itemize}
\item[(a)] $g':\C\setminus \mathcal{S} \to \C$ is analytic.
\item[(b)] $g_{+}'(z) + g_{-}'(z) = \phi'(z;\kappa)$ for $z \in \mathcal{S}$.
\item[(c)] $g'(z) = -g_{1}z^{-2} + \bigO(z^{-3})$ as $z\to\infty$, for some $g_{1}\in \C$.
\end{itemize}
Note that $g$ depends on $\mu,\kappa$, although this is not visible in our notation.
Let $\mathcal{R}(z) =\left((z-z_{0})(z-\overline{z_{0}})\right)^{1/2}$, where the branch cut is on $\mathcal{S}$ and such that $\mathcal{R}(z)>0$ for $z\in (\Re z_{0},+\infty)$. Then
\begin{align*}
\frac{g'_{+}(z)}{\mathcal{R}_{+}(z)} - \frac{g'_{-}(z)}{\mathcal{R}_{-}(z)} = \frac{\phi'(z;\kappa)}{\mathcal{R}_{+}(z)}, \qquad z\in \mathcal{S}.
\end{align*}
We can construct $g'$ in such a way that $\frac{g'(z)}{\mathcal{R}(z)}=\bigO(1)$ as $z\to z_{0}$ and as $z\to\overline{z_{0}}$, by setting
\begin{align}
g'(z) & = \frac{\mathcal{R}(z)}{2\pi i} \int_{\overline{z_{0}}}^{z_{0}} \frac{\phi'(s;\kappa)}{\mathcal{R}_{+}(s)}\frac{ds}{s-z},\qquad  z\in \C \setminus \mathcal{S}.\nonumber 
\end{align}
We make the ansatz that $\mathcal{S}$ crosses the real axis between $0$ and $\frac{1}{a}$. With this choice, we have
\begin{align*}
\mathcal{R}\left(1/a\right)=|z_0-1/a|>0,\qquad \mathcal{R}(-a)=-|z_0+a|<0,\qquad \mathcal{R}(0)=-|z_0|<0.
\end{align*}
We can compute $g'$ explicitly in terms of $z_0$ using residue calculus and obtain
\begin{align}
g'(z)&=
 \mathcal{R}(z) \bigg( ( \kappa - \mu) \frac{\frac{1}{\mathcal{R}(0)}-\frac{1}{\mathcal{R}(z)}}{2(0-z)} + (1-\mu) \frac{\frac{1}{\mathcal{R}(\frac{1}{a})}-\frac{1}{\mathcal{R}(z)}}{2(\frac{1}{a}-z)} + \mu\frac{\frac{1}{\mathcal{R}(-a)}-\frac{1}{\mathcal{R}(z)}}{2(-a-z)} \bigg) \nonumber \\
& = \frac{\phi'(z;\kappa)}{2} - \frac{\mathcal{R}(z)}{2} \bigg( \frac{\mu-\kappa}{|z_0|z} + \frac{1-\mu}{|z_0-\frac{1}{a}|(z-\frac{1}{a})} -  \frac{\mu}{|z_0+a|(z+a)} \bigg). \label{def of gprim}
\end{align}
We must choose $z_{0}$ such that $g'(z)=\bigO(z^{-2})$ as $z\to \infty$, i.e. $z_{0}$ must be such that
\begin{align}\label{def of z0 intro}
\begin{cases}
\ds \frac{1-\mu}{|z_{0}-\frac{1}{a}|} + \frac{\mu-\kappa}{|z_{0}|} - \frac{\mu}{|z_{0}+a|} = 0, \\
\ds \kappa+1-\mu - \frac{a\mu}{|z_{0}+a|} - \frac{1-\mu}{a |z_{0}-\frac{1}{a}|} = 0.
\end{cases}
\end{align}
\paragraph{Existence and uniqueness of $z_{0}$.}
The equations in \eqref{def of z0 intro} can be rewritten as 
\begin{align}\label{def:z0gamma}
\begin{cases}
\ds |z_0+a|=\frac{ (1+a^2) \mu}{(a+\gamma) ( \kappa+1-\mu)}, \\
\ds |z_0-1/a|=\frac{ (1+a^2) (1-\mu)}{a (1-a \gamma) (\kappa+1-\mu)},
\end{cases}
\end{align}
where 
\begin{align}\label{def of gamma}
\gamma:=\frac{\mu-\kappa}{1-\mu+\kappa}\frac{1}{|z_0|}.
\end{align}
This parametrization in terms of $\gamma$ turns out to be convenient. Given $\gamma>0$, the equations in \eqref{def:z0gamma} describe the intersection of two circles with center on the real line, and thus uniquely fix the value of $z_0\in\mathbb C$ with $\Im z_0\geq 0$, provided that 
\begin{align}\label{lol28}
|z_0+a|+|z_0-\tfrac{1}{a}|\geq a+\tfrac{1}{a},\qquad \left(|z_0+a|-|z_0-\tfrac{1}{a}|\right)^2 \leq (a+\tfrac{1}{a})^2.
\end{align}
For $0<\kappa <\kappa_{2} \leq \mu < 1$, the next lemma shows that both inequalities hold strictly and simultaneously if and only if $\gamma \in (0,\frac{1}{a})$.
\begin{lemma}
Let $\mu \in (0,1)$, and suppose that $z_{0}\in \C$ satisfies the two equations in \eqref{def:z0gamma} for some $\gamma \in (0,\frac{1}{a})$. Then
\begin{align}
& |z_0+a|+|z_0-\tfrac{1}{a}|> a+\tfrac{1}{a} , & & \mbox{for all } \kappa \in (0,\kappa_{2}), \label{lol30} \\
& (a+\tfrac{1}{a})^2 > (|z_0+a|-|z_0-\tfrac{1}{a}|)^2, & & \mbox{for all } \kappa \in (0,\kappa_{2}). \label{lol31}
\end{align}
Moreover, \eqref{lol30} does not hold if $\gamma \notin (0,\frac{1}{a})$.
\end{lemma}
\begin{proof}
By \eqref{def:z0gamma},
\begin{align}\label{two circles}
|z_0+a|+|z_0-\tfrac{1}{a}| - (a+\tfrac{1}{a}) = \frac{(1+a^{2})p_{1}(\gamma)}{a(a+\gamma)(1-a\gamma)(1-\mu+\kappa)},
\end{align}
where
\begin{align*}
p_{1}(\gamma) = a(1+\kappa-\mu)\gamma^{2} + \big( a^{2}(1+\kappa-2\mu)-\kappa \big) \gamma + a(\mu-\kappa).
\end{align*}
The discriminant of $p_{1}$ is $<0$ for $\kappa \in (0,\kappa_{2})$. Hence $p_{1}(\gamma)>0$ for all $\gamma \in \R$ and $\kappa\in (0,\kappa_{2})$, and therefore \eqref{lol30} holds if and only if $\gamma \in (0,\frac{1}{a})$. Similarly, using the equations in \eqref{def:z0gamma}, we obtain
\begin{align*}
(a+\tfrac{1}{a})^2 - (|z_0+a|-|z_0-\tfrac{1}{a}|)^2 = \frac{(1+a^{2})^{2}p_{2}(\gamma)p_{3}(\gamma)}{a^{2}(a+\gamma)^{2}(1-a\gamma)^{2}(1+\kappa-\mu)^{2}},
\end{align*}
where
\begin{align*}
& p_{2}(\gamma) = a(1+\kappa-\mu)\gamma^{2} + \big( a^{2}(1+\kappa)-\kappa \big)\gamma - a(\kappa + \mu), \\
& p_{3}(\gamma) = a(1+\kappa-\mu)\gamma^{2} + \big( a^{2}(1+\kappa - 2\mu) + 2\mu - 2 - \kappa \big) \gamma - a(2+\kappa - 3\mu).
\end{align*}
A direct computation shows that the discriminants of $p_{2}$ and $p_{3}$ are negative for any $\kappa \in \R$. Thus $p_{2}(\gamma)>0$ and $p_{3}(\gamma)>0$ for all $\gamma \in \R$ and $\kappa\in \R$, and \eqref{lol31} follows.
\end{proof}

To finish our proof that $z_{0}$ is uniquely determined by the condition $\Im z_{0} \geq 0$ and \eqref{def:z0gamma}, it remains to prove that there exists a unique $\gamma \in (0,\frac{1}{a})$ which is compatible with \eqref{def:z0gamma} and \eqref{def of gamma}. To show this, we apply the law of cosines to the triangles $(-a,0,z_{0})$ and $(0,\frac{1}{a},z_{0})$, and obtain
\begin{align*}
\begin{cases}
|z_{0}-\tfrac{1}{a}|^{2} = |z_{0}|^{2}+\tfrac{1}{a^{2}} -  \tfrac{2}{a} \Re z_{0}, \\
|z_{0}+a|^{2} = |z_{0}|^{2}+a^{2} +2a \Re z_{0}.
\end{cases}
\end{align*}
Using the first equality to eliminate $\Re z_{0}$ from the second, we find
\begin{align*}
(1+a^2) (1+|z_0|^2)=a^2 |z_0-1/a|^2+|z_0+a|^2.
\end{align*}
Substituting the expressions for $|z_0|,|z_0+a|,|z_0-1/a|$ from \eqref{def:z0gamma}--\eqref{def of gamma} in terms of $\gamma$, we obtain the following equation for $\gamma$,
\begin{equation}\label{eqgamma1}
 \frac{(1 + a^2) (1-\mu)^2}{(1-a \gamma)^2} + \frac{(1 + 
    a^2) \mu^2}{(a + \gamma)^2 } - \frac{(\mu- \kappa)^2}{
 \gamma^2}=(\kappa+1-\mu)^2.
\end{equation}
Since $0<\kappa<\kappa_2\leq\mu<1$, we have $\mu-\kappa\in (0,1)$. Hence, the left hand side of \eqref{eqgamma1} is $0$ at $\gamma=\pm\infty$, $-\infty$ at $\gamma=0$ and $+\infty$ at $\gamma=-a$ and $\gamma=1/a$, and therefore this equation has at least one solution $\gamma_1<-a$, one solution $-a<\gamma_2<0$, one solution $0<\gamma_3<1/a$, and one solution $\gamma_4>1/a$. 
We can also rewrite \eqref{eqgamma1} as a polynomial equation $
p(\gamma)q(\gamma)=0$, with
\begin{align}
& p(\gamma)=a \gamma^4 (\kappa+1-\mu)+ \gamma^3 \left(a^2 \kappa+a^2-\kappa+2\mu-2\right)-3 a \gamma^2+\gamma \left( a^2 (\kappa-1)+ (2\mu-\kappa)\right)+a (\mu-\kappa), \label{def:p} \\ 
&q(\gamma)=a(\kappa+1-\mu)\gamma^2+(a^2(\kappa-2\mu+1)-\kappa)\gamma+a(\mu-\kappa). \nonumber
\end{align}
The discriminant of the quadratic equation $q(\gamma)=0$ is negative for $\kappa_1<\kappa<\kappa_2$, such that $\gamma_1,\gamma_2,\gamma_3,\gamma_4$ are the (only) roots of the degree $4$ polynomial $p$. If $\kappa_{1}>0$ and $\kappa \in (0,\kappa_{1})$, then $q$ has two distinct roots, which can be computed explicitly, and both are bigger than $\frac{1}{a}$. For $\kappa=\kappa_{1}>0$, $q$ has a double root, which coincides with the root of $p$ lying on $(\frac{1}{a},+\infty)$. 

We conclude that there is a unique $\gamma \in (0,\frac{1}{a})$ which is compatible with \eqref{def:z0gamma} and \eqref{def of gamma}. 

We obtain the following result.

\begin{proposition}\label{prop:z0}
Let $\mu\in\big(0,1\big)$ and $\kappa\in (0,\kappa_2)$. There exists a unique  $z_0=z_0(\kappa)\in\mathbb C$ with $\Im z_0> 0$, smooth in $\mu\in(0,1)$ and $\kappa\in(0,\kappa_2)$, such that \eqref{def of z0 intro} holds. We have the expansions
\begin{align}
z_0 & = \tilde{z}_{0}^{(0)} + i\tilde{z}_{0}^{(1/2)}\sqrt{\kappa_{2}-\kappa} + \tilde{z}_{0}^{(1)}(\kappa_{2}-\kappa) + i\tilde{z}_{0}^{(3/2)}(\kappa_{2}-\kappa)^{3/2} \nonumber \\
& + \tilde{z}_{0}^{(2)}(\kappa_{2}-\kappa)^{2} + \mathcal O((\kappa_2-\kappa)^{5/2}), & &  \mbox{as $\kappa\to\kappa_2$,} \label{z0kappa2}
\end{align}
for some $\tilde{z}_{0}^{(0)}, \tilde{z}_{0}^{(1/2)},\tilde{z}_{0}^{(1)},\tilde{z}_{0}^{(3/2)},\tilde{z}_{0}^{(2)} \in \R$, and
\begin{align}
z_0 & = z_{0}^{(0)} + iz_{0}^{(1/2)}\sqrt{\kappa} + z_{0}^{(1)}\kappa + iz_{0}^{(3/2)}\kappa^{3/2}  + z_{0}^{(2)}\kappa^{2} + \mathcal O(\kappa^{5/2}), & & \mbox{as $\kappa\to 0$,} \label{z00}
\end{align}
for some $z_{0}^{(0)}, z_{0}^{(1/2)},z_{0}^{(1)},z_{0}^{(3/2)},z_{0}^{(2)} \in \R$. Both \eqref{z0kappa2} and \eqref{z00} hold uniformly for $\mu\in[\delta,1-\delta]$ for any fixed $\delta>0$. The coefficients of the leading terms are given by
\begin{align*}
& \tilde{z}_{0}^{(0)} = \frac{a-(1+a^{2})\sqrt{\mu(1-\mu)}}{(1+a^{2})\mu-1} \in (-a,\tfrac{1}{a}), \qquad z_{0}^{(0)} = \frac{-a-\sqrt{a^2+4\mu(1-\mu)}}{2(1-\mu)} \in (-\infty,-a), \\
& \tilde{z}_{0}^{(1/2)} = \frac{\sqrt{2}(1+a^{2})^{3/2}\big((1-\mu)\mu\big)^{1/4}}{\sqrt{a} \big( 1-(1-a^{2})\mu+2a\sqrt{(1-\mu)\mu} \big)},
\end{align*}
and moreover $z_{0}^{(1/2)}>0$. \end{proposition}
\begin{remark}
All constants appearing in \eqref{z0kappa2} and \eqref{z00} can be computed explicitly. However, since these expressions are rather lengthy, we omit them.
\end{remark}
\begin{remark}
If $\mu-\kappa_{2}\geq\delta$ for some fixed $\delta>0$, then $\tilde{z}_{0}^{(0)}$ remains bounded away from $0$. Thus, by \eqref{z0kappa2},
\begin{align}\label{lol32}
z_0 = \tilde{z}_{0}^{(0)}(1+\bigO(\sqrt{\kappa_{2}-\kappa})), \qquad \mbox{as } \kappa \to \kappa_{2},
\end{align}
uniformly for $\mu\in[\delta,1-\delta]$ with $\mu-\kappa_{2}\geq\delta$.
\end{remark}
\begin{proof}
The existence and uniqueness of $z_0$ follows from the discussion above. The smoothness directly follows from \eqref{def of z0 intro}. The expansions \eqref{z0kappa2} and \eqref{z00} require a technical and involved analysis, which we omit; in particular, we first need to show that $z_{0}(\kappa,\mu)-\tilde{z}_{0}^{(0)}$ and $z_0(\kappa,\mu)-z_{0}^{(0)}$ are small as $\kappa \to \kappa_{2}$ and as $\kappa \to 0$, respectively; the coefficients in the expansions can then be obtained recursively using a long but straightforward computation.
\end{proof}

\begin{remark}
The expansions \eqref{z0kappa2} and \eqref{lol32} imply that 
\begin{align*}
|z_{0}| = \tilde{z}_{0}^{(0)} + \bigg( \frac{(\tilde{z}_{0}^{(1/2)})^{2}}{2\tilde{z}_{0}^{(0)}}+\tilde{z}_{0}^{(1)} \bigg)(\kappa_{2}-\kappa) + \bigO((\kappa_{2}-\kappa)^{2}), \qquad \mbox{as } \kappa \to \kappa_{2}
\end{align*}
uniformly for $\mu \in [\delta,1-\delta]$ and $\mu-\kappa_{2}\geq\delta$ for any fixed $\delta >0$. Using the fact that $\gamma(\kappa_2)=z_0(\kappa_2)=\tilde{z}_{0}^{(0)}$, we then obtain 
\begin{equation}\label{eq:asgamma}
\gamma=\frac{\mu-\kappa}{(1-\mu+\kappa)|z_0|}=
{\tilde{z}_{0}^{(0)}}+C_2(\kappa_2-\kappa) {+ \bigO((\kappa_{2}-\kappa)^{2})},\qquad \mbox{as } \kappa\to\kappa_2,
\end{equation}
for some $C_2\neq 0$ uniformly for $\mu \in [\delta,1-\delta]$ and $\mu-\kappa_{2}\geq\delta$. We will use this later on. 
\end{remark}

\paragraph{Construction of the $g$-function}
Integrating \eqref{def of gprim} from $\infty$ to $z \in\mathbb C\setminus\left( \mathcal{S}\cup [\overline{z_{0}},\overline{z_{0}}-i\infty)\right)$ yields
\begin{align*}
g(z) = \int_{\infty}^{z}g'(s)ds, \qquad z \in \mathbb C\setminus\left( \mathcal{S}\cup [\overline{z_{0}},\overline{z_{0}}-i\infty)\right).
\end{align*}
A residue computation, which uses \eqref{def of z0 intro}, then shows that $g$ can be analytically continued across $(\overline{z_{0}},\overline{z_{0}}-i\infty)$, as desired.

We define,
with the constant $\ell\in\mathbb C$ as in the RH conditions for $g$,
\begin{equation}\label{def:xi}
\xi(z)=g(z)-\frac{1}{2}\phi(z;\kappa)+\frac{\ell}{2},\qquad z\in\mathbb C\setminus\big(\mathcal S\cup (-\infty,0]\cup [\tfrac{1}{a},+\infty)\big).
\end{equation}
By \eqref{def of gprim}, we have
\begin{equation}\label{lol29}
\xi'(z)=-\frac{\mathcal{R}(z)}{2} \bigg(\frac{\mu-\kappa}{|z_0|z} +  \frac{1-\mu}{|z_0-\frac{1}{a}|(z-\frac{1}{a})} -  \frac{\mu}{|z_0+a|(z+a)} \bigg)
\end{equation}
We can rewrite this as 
\begin{equation}\label{def:xiprime}
\xi'(z)=-Q(z)\mathcal R(z),\qquad Q(z)=\frac{\kappa+1-\mu}{2}\frac{z-\gamma}{z(z+a)(z-1/a)},
\end{equation}
where $\gamma$ is given by \eqref{def of gamma}.
As $z\to -a,0,1/a$, we have that $g'(z)$ is bounded, such that $\xi'(z)\sim -\frac{1}{2}\phi'(z)$. 
Moreover, by \eqref{lol29}, $\xi'(x)/\mathcal{R}(x)$ increases from $-\infty$ to $+\infty$ as $x \in (0,\frac{1}{a})$ increases, so we must have $\gamma\in(0,1/a)$, which is consistent with the discussion surrounding \eqref{two circles}.

By \eqref{def:xiprime}, we have
\begin{equation}
\xi'(z)^2=
\frac{\left(\kappa+1-\mu\right)^2}{4}\frac{(z-\gamma)^2(z-z_0)(z-\overline{z_0})}{z^2(z+a)^2(z-1/a)^2}.
\end{equation}
In a similar way as in Section \ref{sec:small frozen region}, we now analyze the vertical trajectories of the quadratic differential \[\frac{(z-\gamma)^2(z-z_0)(z-\overline{z_0})}{z^2(z+a)^2(z-1/a)^2}dz^2,\]
in terms of the parameters $\gamma\in(0,1/a)$ and $z_0$ with $\Im z_0>0$, i.e.\ the curves on which 
\begin{align*}
\Re\left(\frac{(z-\gamma)\mathcal R(z)}{z(z+a)(z-1/a)}dz\right)=0.
\end{align*}
Instead of a quadratic differential with two double zeros, we now have a quadratic differential with one double zero $\gamma$ and two simple zeros $z_0,\overline{z_0}$.
Note that the trajectories are independent of the choice of the branch cut $\mathcal S$ for $\mathcal R$. Since $\xi_{+}+\xi_{-}=0$ on $\mathcal{S}$, we have $\xi(z_{0})=\xi(\overline{z_{0}})=0$, and thus we can write
\begin{align}\label{lol8}
\xi(z)=\int_{z_0}^z\xi'(z)dz, \qquad z \in \C \setminus \big( \mathcal{S} \cup (-\infty,0] \cup [\tfrac{1}{a},+\infty) \big),
\end{align}
where the path of integration does not cross $\mathcal{S} \cup (-\infty,0] \cup [\frac{1}{a},+\infty)$.
Note that for any $h\in \R$, the level set
\begin{align*}
\mathcal T_{h}:=\{z\in\mathbb C:\Re \xi(z)=h\}
\end{align*}
depends on the choice of the branch cut. (However, for $h=0$, the connected component of $\mathcal{T}_{0}$ containing $z_{0}$ is independent of the branch cut, since $\xi_+=-\xi_-$ on $\mathcal S$.) It is therefore important to specify $\mathcal{S}$. We will define it in terms of a horizontal trajectory. For this, first recall that the \textit{horizontal} trajectories of $\xi'(z)^{2}dz^{2}$ are the curves on which 
\begin{align*}
\Im\left(\frac{(z-\gamma)\mathcal R(z)}{z(z+a)(z-1/a)}dz\right)=0.
\end{align*}
Since $\xi'(x)\in \R$ for $x\in \R\setminus \{-a,0,\frac{1}{a}\}$, $\Im \xi$ is piecewise constant on $\R$. In particular, the two horizontal trajectories of $\xi'(z)^{2}dz^{2}$ emanating from any $x\in \R\setminus \{-a,0,\gamma,\frac{1}{a}\}$ must follow the real line. Since $\gamma$ is a simple zero of $\xi'$, there are four horizontal trajectories emanating from $\gamma$, two of which are subsets of $\R$, and the other two make angles $\pm\frac{\pi}{2}$ with $\R$. Let $\Gamma_{h,+}$ and $\Gamma_{h,-}$ be the trajectories going into the upper and lower half-plane, respectively. By the above discussion, $\Gamma_{h,+}$ cannot intersect $\R\setminus \{-a,0,\frac{1}{a}\}$. Moreover, for any $0 \leq \alpha_{1} \leq \alpha_{2} \leq \pi$, we easily check that as $\epsilon \to 0_{+}$,
\begin{align}
& \xi(-a+e^{\alpha_{1}i}\epsilon)-\xi(-a+e^{\alpha_{2}i}\epsilon) = \frac{\mu (\alpha_{2}-\alpha_{1})i}{4}(1+o(1)), \nonumber \\
& \xi(e^{\alpha_{1}i}\epsilon)- \xi(e^{\alpha_{2}i}\epsilon) = \frac{(\kappa-\mu) (\alpha_{2}-\alpha_{1})i}{4}(1+o(1)), \nonumber \\ 
& \xi(\tfrac{1}{a}+e^{\alpha_{1}i}\epsilon)- \xi(\tfrac{1}{a}+e^{\alpha_{2}i}\epsilon) = \frac{(1-\mu) (\alpha_{2}-\alpha_{1})i}{4}(1+o(1)). \label{lol7}
\end{align}
Equations \eqref{lol7} with $\alpha_{1}=0$ and $\alpha_{2}=\pi$ show that for $0<\kappa<\kappa_{2} \leq \mu < 1$, $\im \xi_{+}$ assumes four distinct values on the four connected components of $\R\setminus \{-a,0,\frac{1}{a}\}$. Equations \eqref{lol7} also imply that $\Gamma_{h,+}$ cannot intersect $\{-a,0,\frac{1}{a}\}$. Hence $\Gamma_{h,+}$ must connect with $z_{0}$. (By \eqref{lol8}, this also shows that $\im \xi=0$ on $(0,\frac{1}{a})$.) Since $\xi(z) = \overline{\xi(\overline{z})}$ holds for all $z\in \C$, we have $\Gamma_{h,-} = \overline{\Gamma_{h,+}}$. We define $\mathcal{S}:=\Gamma_{h,+}\cup \Gamma_{h,-}$ with the orientation going from $\overline{z_{0}}$ to $z_{0}$, see also the orange curves in Figure \ref{fig:level curve xi}. The next result establishes the structure of $\mathcal{T}_{0}$.

\begin{figure}
\begin{center}
\begin{tikzpicture}[master]
\node at (0,0) {\includegraphics[width=7cm]{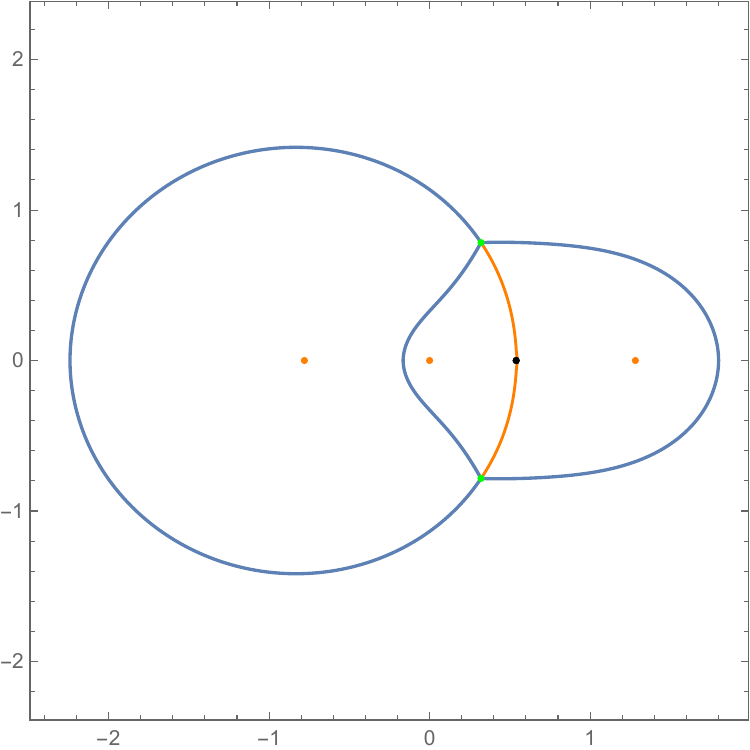}};
\node at (-1.5,0.63) {\small $+$};
\node at (0.85,0.33) {\small $-$};
\node at (1.8,0.7) {\small $+$};
\node at (1.8,1.65) {\small $-$};
\node at (1.4,-0.3) {\small $\mathcal{S}$};
\node at (-1.3,2.25) {\small $\Gamma_{1}$};
\node at (0.37,-0.4) {\small $\Gamma_{2}$};
\node at (1.7,-1.2) {\small $\Gamma_{3}$};
\end{tikzpicture} \begin{tikzpicture}[slave]
\node at (0,0) {\includegraphics[width=7cm]{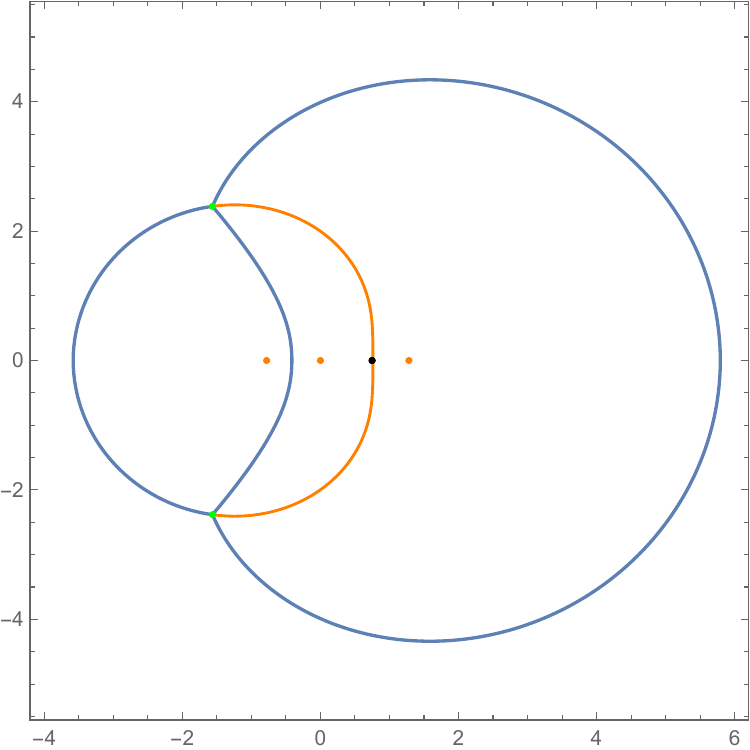}};
\node at (-2,0.63) {\small $+$};
\node at (-0.55,0.63) {\small $-$};
\node at (1,0.7) {\small $+$};
\node at (-2.2,2.2) {\small $-$};
\node at (0.05,-0.5) {\small $\mathcal{S}$};
\node at (-2.4,-1.1) {\small $\Gamma_{1}$};
\node at (-1.3,-0.6) {\small $\Gamma_{2}$};
\node at (0.7,-2.7) {\small $\Gamma_{3}$};
\end{tikzpicture}
\end{center}
\caption{\label{fig:level curve xi}
In each image, the blue curves are $\Gamma_{1},\Gamma_{2},\Gamma_{3}$, and the orange curve is $\mathcal{S}$. Together, these curves delimit four regions in the complex plane, with the sign of $\Re \xi$ in each region indicated by $\pm$. The orange dots are $-a$, $0$ and $\frac{1}{a}$; the black dot is $\gamma$; and the green dots are $z_{0}$ and $\overline{z_{0}}$. The figures are drawn for $a=0.78$, $\mu=0.76$, and with $\kappa=0.44$ (left) and $\kappa=0.08$ (right).}
\end{figure}
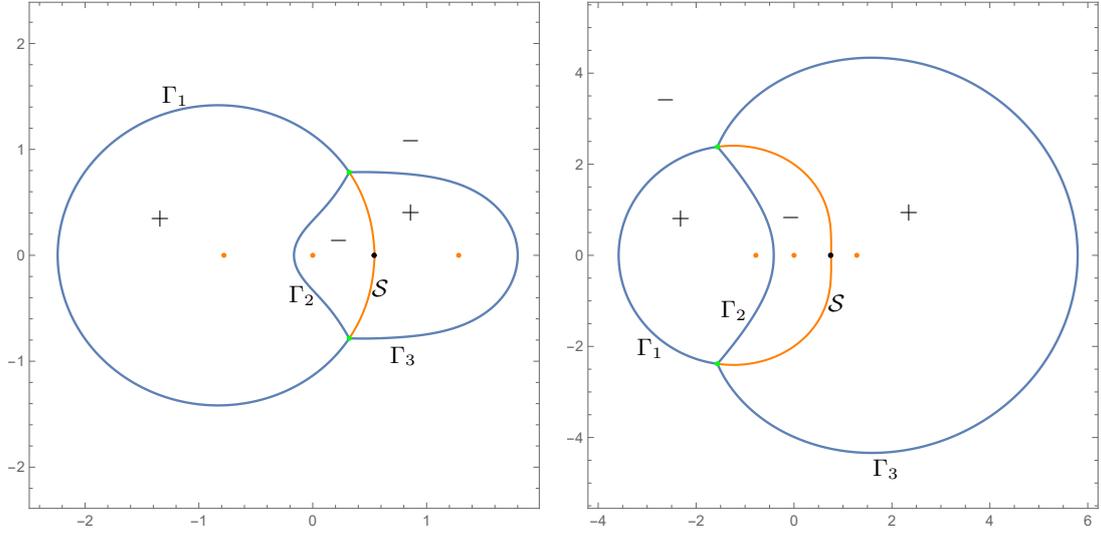
\begin{proposition}\label{prop:crittraj2}
Let $\mu\in\big(0,1\big)$ and $\kappa\in (0,\kappa_2)$. Then the level set $\mathcal T_0$ satisfies
\begin{align*}
\mathcal T_0=\Gamma_1\sqcup\Gamma_2\sqcup\Gamma_3\sqcup\{z_0,\overline{z_0}\},
\end{align*}
where the contours $\Gamma_1,\Gamma_2,\Gamma_3$ are disjoint, have no self-intersections, and
\begin{itemize}
\item $\Gamma_1$ connects $z_0$ with $\overline z_0$ and crosses the real line at a point smaller than $-a$;
\item $\Gamma_2$ connects $z_0$ with $\overline z_0$ and crosses the real line at a point between $-a$ and $0$;
\item $\Gamma_3$ connects $z_0$ with $\overline z_0$ and crosses the real line at a point bigger than $1/a$.
\end{itemize}
Moreover, 
\begin{itemize}
\item $\Re \xi >0$ in the region delimited by $\Gamma_1$ and $\Gamma_2$, and in the region delimited by $\mathcal S$ and $\Gamma_3$;
\item $\Re \xi <0$ in the region delimited by $\Gamma_2$ and $\mathcal S$, and in the outer region.
\end{itemize}
\end{proposition}
\begin{proof}
Since $\xi'(z)$ vanishes as a square root at $z_{0}$, there are three vertical trajectories emanating at $z_0$, which we denote by $\Gamma_1,\Gamma_2,\Gamma_3$. By the max/min principle for harmonic functions, the curves $\Gamma_1,\Gamma_2,\Gamma_3$ cannot form a closed loop in the upper-half plane. They also cannot connect with infinity, since $\Re \xi(\infty)=-\infty$. So they must intersect $\R\setminus \{-a,0,\frac{1}{a}\}$ at three distinct points. Since $\xi(z) = \overline{\xi(\overline{z})}$ holds for all $z\in \C$, the branches must obey the symmetries $\Gamma_{j} = \overline{\Gamma_{j}}$, $j=1,2,3$, and therefore they must connect $z_0$ with $\overline{z_{0}}$.

By inspection of the sign of ${\Re}\xi'(x)$ for $x$ on the real line, we also see that the level set $\mathcal T_{0}$ must contain exactly one point $x_1<-a$, exactly one point $-a<x_2<0$ and exactly one point $x_3>1/a$. {In addition, since $\Re\xi'(x)>0$ for all $x\in (0,\frac{1}{a})\setminus\{\gamma\}$ and $\Re \xi$ is possibly discontinuous at $x=\gamma$, there must be $0$, $1$ or $2$ points of the level set between $0$ and $1/a$ (if there are two points, one must be on $(0,\gamma)$ and the other one must be on $(\gamma,\frac{1}{a})$). In fact there cannot be $1$ point, otherwise we again find a contradiction using the max/min principle for harmonic functions. There also cannot be $2$ points. To see this, note that $\im \xi$ is constant along $\mathcal{S}$. Since $z_{0}$ is the only critical point of $\xi'(z)^{2}dz^{2}$ in the upper half-plane $\C^{+}$, $\Re \xi_{\pm}$ must be monotone on $\mathcal{S}\cap \C^{+}$. This, together with a local analysis of $\xi'$ around $z_{0}$, implies that $\Re \xi_{+}(z)$ (resp. $\Re \xi_{-}$) is decreasing (resp. increasing) as $z$ follows $\mathcal{S}$ from $z_{0}$ to $\gamma$. In particular, 
\begin{align*}
\Re \xi(\gamma-0_{+})<0, \qquad \Re \xi(\gamma+0_{+})>0,
\end{align*}
and therefore $\Gamma_{1},\Gamma_{2},\Gamma_{3}$ cannot intersect $(0,\frac{1}{a})$. By construction, the curves $\Gamma_{1},\Gamma_{2},\Gamma_{3}$ and $\mathcal{S}$ divide $\C$ into four regions where the sign of $\Re \xi$ remains constant. These signs can be determined from
\begin{align*}
\Re \xi (-a)=+\infty, \qquad \Re \xi (0)=-\infty, \qquad \Re \xi(\tfrac{1}{a})=+\infty, \qquad \Re \xi(\infty)=-\infty.
\end{align*}
}
\end{proof}

\begin{figure}
\begin{center}
\begin{tikzpicture}[master]
\node at (0,0) {\includegraphics[width=7cm]{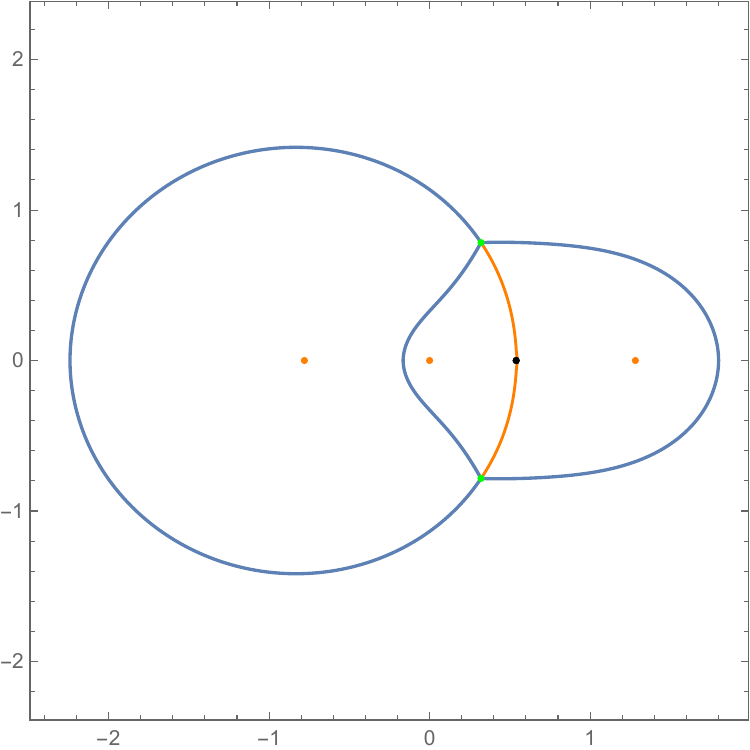}};

\draw[black,arrows={-Triangle[length=0.12cm,width=0.08cm]}]
(-0.95,3) --  ++(180:0.001);
\node at (1.3,2.1) {\small $\tilde{\Sigma}_{1}$};
\draw ($(0,0.13)+(0.99,1.11)$) to [out=60, in=0]
(-1,3) to [out=180, in=90] (-3.05,0.13) to [out=-90, in=180] (-1,-2.74) to [out=0, in=-60] ($(0,0.13)+(0.99,-1.11)$);

\draw[black,arrows={-Triangle[length=0.12cm,width=0.08cm]}]
(-0,1.21) --  ++(180:0.001);
\node at (-0.6,1.4) {\small $\tilde{\Sigma}_{0}$};
\draw ($(0,0.13)+(0.99,1.11)$) to [out=180, in=0]
(-0.7,1.2) to [out=180, in=90] (-1.7,0.13) to [out=-90, in=180] (-0.9,-0.96) to [out=0, in=-180] ($(0,0.13)+(0.99,-1.11)$);

\node at (1.45,-0.3) {\small $\mathcal{S}$};
\draw[black,arrows={-Triangle[length=0.12cm,width=0.08cm]}]
(1.3,-0.16) --  ++(82:0.001);
\end{tikzpicture} \begin{tikzpicture}[slave]
\node at (0,0) {\includegraphics[width=7cm]{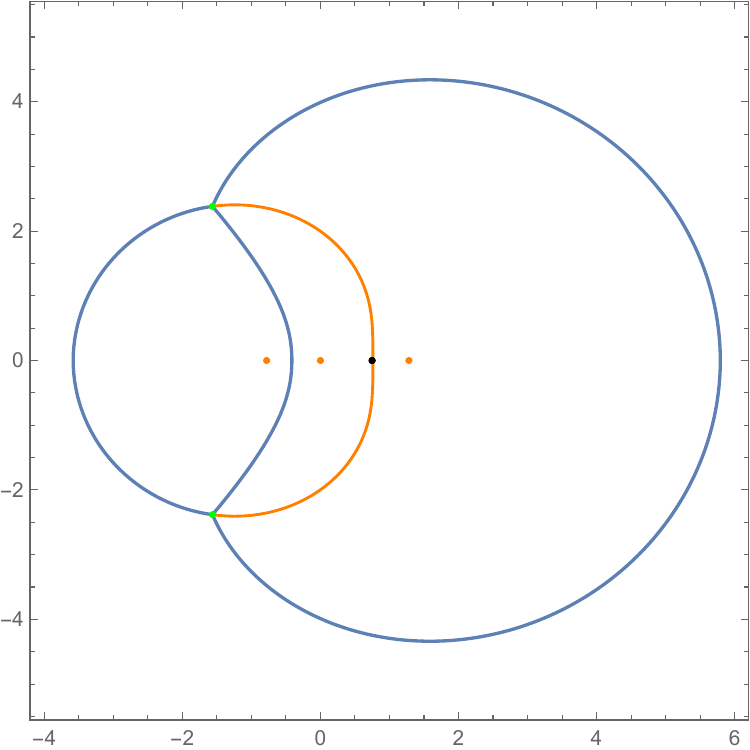}};

\draw[black,arrows={-Triangle[length=0.12cm,width=0.08cm]}]
(-2.25,1.8) --  ++(180:0.001);
\node at (-2.25,2.05) {\small $\tilde{\Sigma}_{1}$};
\draw ($(0,0.13)+(-1.51,1.43)$) to [out=120, in=0]
(-2.2,1.8) to [out=180, in=90] (-3.05,0.13) to [out=-90, in=180] (-2.2,-1.56) to [out=0, in=-120] ($(0,0.13)+(-1.51,-1.43)$);

\draw[black,arrows={-Triangle[length=0.12cm,width=0.08cm]}]
(-1.9,0.13) --  ++(-90:0.001);
\node at (-1.65,0.6) {\small $\tilde{\Sigma}_{0}$};
\draw ($(0,0.13)+(-1.51,1.43)$) to [out=240, in=90]
(-1.9,0.13) to [out=-90, in=-240] ($(0,0.13)+(-1.51,-1.43)$);

\node at (0.1,-0.5) {\small $\mathcal{S}$};
\draw[black,arrows={-Triangle[length=0.12cm,width=0.08cm]}]
(-0.045,-0.35) --  ++(84:0.001);
\end{tikzpicture}
\end{center}
\caption{\label{fig:contour Sigma T}
In each image, the black curves are $\tilde \Sigma_0,\tilde\Sigma_1$, the blue curves are $\mathcal{T}_{0}$, and the orange curve is $\mathcal{S}$. The orange dots are $-a$, $0$ and $\frac{1}{a}$; the black dot is $\gamma$; and the green dots are $z_{0}$ and $\overline{z_{0}}$. The figures are drawn for $a=0.78$, $\mu=0.76$, and with $\kappa=0.44$ (left) and $\kappa=0.08$ (right).}
\end{figure}

The following is a direct corollary of Proposition \ref{prop:crittraj2}.
\begin{corollary}\label{cor:contours2}
Let $\mu\in\big(0,1\big)$ and $\kappa\in (0,\kappa_2)$.
For any $\delta>0$, there exist $c,\eta>0$ and a contour $\Sigma_T=\tilde \Sigma_0\cup\tilde\Sigma_1\cup\mathcal S\cup\{z_0,\overline{z_0}\}$ as shown in Figure \ref{fig:contour Sigma T}, such that the following holds for $z\in\Sigma_T$:
\begin{align*}
&\Re \xi(z;\kappa)<-c , &z\in\tilde\Sigma_1,\ |z-z_0|\geq \eta,\ |z-\overline{z_0}|\geq \eta,\\
&\Re \xi(z;\kappa)>c , &z\in\tilde\Sigma_0,\ |z-z_0|\geq \eta,\ |z-\overline{z_0}|\geq \eta,
\\
&\Re \xi_+(z;\kappa)<-c, &z\in\mathcal S,\ |z-z_0|\geq \eta,\ |z-\overline{z_0}|\geq\eta,
\end{align*}
uniformly for $\kappa \in [\delta,\kappa_2-\delta]$ and $\mu\in[\delta,1-\delta]$.
\end{corollary}

We can compute $g$ and $\ell$ explicitly in terms of $z_0$ as follows.
Integrating \eqref{def of gprim} from $\infty$ to $z\in \C \setminus \mathcal{S}$, we find
\begin{align}
& g(z) = \int_{\infty}^{z}g'(s)ds \nonumber\\
& = \frac{{\kappa-\mu}}{{2}|z_{0}|} \bigg\{  - (\Re z_{0}) \log \bigg( \frac{z + \mathcal{R}(z) - \Re z_{0}}{i \, \im z_{0}} \bigg) + |z_{0}| \log \bigg( \frac{(\Re z_{0})z - |z_{0}|^{2} + |z_{0}| \mathcal{R}(z)}{i \, \im z_{0}} \bigg) \bigg\} - \frac{{1-\mu}}{{2}|\frac{1}{a}-z_{0}|} \bigg\{ \nonumber\\
&\quad  \big( \tfrac{1}{a}-\Re z_{0} \big) \log \bigg( \frac{z+\mathcal{R}(z) - \Re z_{0}}{i \, \im z_{0}} \bigg) - |\tfrac{1}{a}-z_{0}| \log \bigg(  \frac{(\frac{1}{a}-\Re z_{0})z - \frac{\Re z_{0}}{a}+|z_{0}|^{2} + |\frac{1}{a}-z_{0}|\mathcal{R}(z)}{i\frac{1}{a}\im z_{0}} \bigg) \bigg\} \nonumber\\
&\quad  + \frac{{\mu}}{{2}|a+z_{0}|} \bigg\{ - (a+\Re z_{0}) \log \bigg( \frac{z+\mathcal{R}(z) - \Re z_{0}}{i \, \im z_{0}} \bigg)  \nonumber \\
&\quad  + |a+z_{0}| \log \bigg( \frac{(a+\Re z_{0})z - a \, \Re z_{0} - |z_{0}|^{2} + |a+z_{0}|\mathcal{R}(z)}{i\im z_{0}} \bigg) \bigg\} - \frac{\ell}{2},\label{eq:exprg}
\end{align} 
where
\begin{multline}\label{def of ell}
\ell = {(1-\mu)} \log a - {(\kappa+1-\mu)} \log 2 + {(1-\mu)} \log \big( \tfrac{1}{a} - \Re z_{0} + |\tfrac{1}{a}-z_{0}| \big) \\ + {(\kappa-\mu)} \log \big( \Re z_{0} + |z_{0}| \big) + {\mu} \log\big( a+\Re z_{0} + |a+z_{0}| \big).
\end{multline}
In \eqref{eq:exprg} and \eqref{def of ell}, the principal branch is used for the logarithms. Note that
the first line of equation \eqref{eq:exprg} defines a single-valued function, because the integral of $g'$ along any closed loop surrounding $\mathcal{S}$ is $0$.

\subsection{First transformation: $Y \mapsto T$}
Our first transformation consists in deforming the contours: according to the results obtained in Corollary \ref{cor:contours2}, we deform $\Sigma_0,\Sigma_1$ in such a way that they coincide along $\mathcal S$, and we denote $\tilde\Sigma_0,\tilde\Sigma_1$ for the parts of these deformed contours away from $\mathcal S$, as shown in Figure \ref{fig:contour Sigma T}.
We construct the solution of the resulting RH problem by analytically continuing the RH solution $Y$ from the region inside $\Sigma_0$ and from the region outside $\Sigma_1$ respectively. As a result of the fact that the two jump contours coincide on $\mathcal S$, the jump matrix for $T$ on $\mathcal S$ is the product of the jump matrices for $Y$ on $\Sigma_1$ and on $\Sigma_0$. We obtain the following RH problem.

\subsubsection*{RH problem for $T$}
\begin{itemize}
\item[(a)] $T:\mathbb C\setminus\Sigma_T\to \mathbb C^{2\times 2}$ is analytic, with
\[\Sigma_T=\left(\tilde \Sigma_0\cup\tilde \Sigma_1\cup\mathcal S\right)\cup\{z_0,\overline{z_0}\}.\]
\item[(b)] $T_+(z)=T_-(z)J_{\widehat{Y}}(z)$ for $z\in \tilde \Sigma_0\cup\tilde \Sigma_1\cup\mathcal S$, with
\begin{align}\label{eq:JYhat}
J_T(z)=\begin{cases}
\begin{pmatrix}
1& e^{-N\phi(z)}\frac{z}{(1-az)^\epsilon}\\ 0& 1
\end{pmatrix},&z\in\tilde\Sigma_1,\\
\begin{pmatrix}
1& 0\\ 
-e^{N\phi(z)}\frac{(1-az)^\epsilon}{z}& 1
\end{pmatrix},&z\in\tilde \Sigma_0, \\
\begin{pmatrix}
0& e^{-N\phi(z)}\frac{z}{(1-az)^\epsilon}\\ 
-e^{N\phi(z)}\frac{(1-az)^\epsilon}{z}& 1
\end{pmatrix},&z\in\mathcal{S}.
\end{cases}
\end{align}
\item[(c)]{As $z\to \infty$, $T(z) = I +\bigO(z^{-1})$}.

As $z\to z_0$ and as $z\to \overline{z_0}$, we have $T(z) = \bigO(1)$.
\end{itemize}

\subsection{Second transformation: $T\mapsto S$}\label{subsection: S liquid}

We now define
\begin{align}\label{def of S}
S(z) := e^{\frac{N\ell}{2}\sigma_{3}} T(z) e^{-Ng(z)\sigma_{3}} e^{-\frac{N\ell}{2}\sigma_{3}}.
\end{align}
Using \eqref{lol9}, we infer that $S$ satisfies the following RH problem.
\subsubsection*{RH problem for $S$}
\begin{itemize}
\item[(a)] $S:\C\setminus\Sigma_T\to \mathbb C^{2\times 2}$ is analytic.
\item[(b)] $S_+(z)=S_-(z)J_S(z)$ for $z\in \tilde\Sigma_0\cup\tilde \Sigma_1\cup\mathcal S$, with
\begin{align}\label{eq:JS}
J_S(z)=\begin{cases}
\begin{pmatrix}
1& e^{-N(\phi(z)-2g(z)-\ell)} \frac{z}{(1-az)^\epsilon}
\\ 0& 1
\end{pmatrix}, & z\in\tilde\Sigma_1,\\
\begin{pmatrix}
1& 0\\ 
-e^{N(\phi(z)-2g(z)-\ell)}\frac{(1-az)^\epsilon}{z} & 1
\end{pmatrix},&z\in\tilde\Sigma_0, \\
\begin{pmatrix}
0& \frac{z}{(1-az)^\epsilon} \\ 
-\frac{(1-az)^\epsilon}{z} & e^{N(g_{+}(z)-g_{-}(z))}
\end{pmatrix},&z\in\mathcal{S}.
\end{cases}
\end{align}
\item[(c)]As $z\to \infty$, $S(z)=I +\bigO(z^{-1})$.

As $z\to z_0$ and as $z\to \overline{z_0}$, $S(z) = \bigO(1)$.
\end{itemize}
In terms of the function $\xi$ defined in \eqref{def:xi}, we can rewrite the jump matrix $J_S$ as 
\begin{align}\label{eq:JS2}
J_S(z)=\begin{cases}
\begin{pmatrix}
1& e^{2N\xi(z)} \frac{z}{(1-az)^\epsilon}
\\ 0& 1
\end{pmatrix},&z\in\tilde\Sigma_1,\\
\begin{pmatrix}
1& 0\\ 
-e^{-2N\xi(z)}\frac{(1-az)^\epsilon}{z} & 1
\end{pmatrix},&z\in\tilde\Sigma_0, \\
\begin{pmatrix}
0& \frac{z}{(1-az)^\epsilon} \\ 
-\frac{(1-az)^\epsilon}{z} & e^{2N\xi_+(z)}
\end{pmatrix},&z\in\mathcal{S}.
\end{cases}
\end{align}
By Corollary \ref{cor:contours2}, we now easily see that for $z$ bounded away from $z_0,\overline{z_0}$, the jump matrix $J_S(z)$ converges as $N\to\infty$ to the identity matrix on $\tilde\Sigma_1$ and on $\tilde\Sigma_0$, and to the $N$-independent matrix $\begin{pmatrix}
0 & \frac{z}{(1-az)^\epsilon} \\ -\frac{(1-az)^\epsilon}{z} & 0
\end{pmatrix}$ on $\mathcal S$.
We will deal with this $N$-independent matrix by constructing a global parametrix which has exactly this jump matrix, and we will afterwards deal with small neighborhoods of $z_0,\overline{z_0}$ by constructing local parametrices there.

\subsection{Global parametrix}\label{subsection: Pinf liquid}
First, we will construct the global parametrix $P^{(\infty)}$, which needs to satisfy the following conditions.
\subsection*{RH problem for $P^{(\infty)}$}
\begin{itemize}
\item[(a)] $P^{(\infty)}:\C\setminus \mathcal{S}\to \mathbb C^{2\times 2}$ is analytic.
\item[(b)] $P^{(\infty)}_+(z)=P^{(\infty)}_-(z) \begin{pmatrix}
0 & \frac{z}{(1-az)^\epsilon} \\ -\frac{(1-az)^\epsilon}{z} & 0
\end{pmatrix} $ for $z\in \mathcal{S}$.
\item[(c)]As $z\to \infty$, $P^{(\infty)}(z)=I+\bigO(z^{-1})$.
\end{itemize}
A solution is given by
\begin{align}\label{def:Pinf}
P^{(\infty)}(z) = \mathcal{G}_{\infty}^{-\sigma_{3}}  \begin{pmatrix}
\frac{\beta(z)+\beta(z)^{-1}}{2} & \frac{\beta(z)-\beta(z)^{-1}}{2i} \\
\frac{\beta(z)-\beta(z)^{-1}}{-2i} & \frac{\beta(z)+\beta(z)^{-1}}{2}
\end{pmatrix}\mathcal{G}(z)^{\sigma_{3}}, \qquad z \in \C \setminus \mathcal{S},
\end{align}
where $\beta(z)$, $\mathcal{G}(z)=\mathcal{G}(z;\epsilon)$ and $\mathcal G_\infty=\mathcal G_\infty(\epsilon)$ are given by
\begin{align}
& \beta(z) = \frac{(z-z_{0})^{\frac{1}{4}}}{(z-\overline{z_{0}})^{\frac{1}{4}}}, \label{def of beta} \\
& \mathcal{G}(z) = \exp \bigg(  \frac{\mathcal{R}(z)}{2\pi i} \int_{\mathcal{S}} \frac{{\epsilon}\log(1-as)-\log s}{\mathcal{R}_{+}(s)} \frac{ds}{s-z} \bigg), \label{def of Gcal} \\
& \mathcal{G}_{\infty} = \exp \bigg(  \frac{-1}{2\pi i} \int_{\mathcal{S}} \frac{{\epsilon}\log(1-as)-\log s}{\mathcal{R}_{+}(s)} ds \bigg)>0, \label{def of Gcal inf}
\end{align}
and the branch cuts in \eqref{def of gamma} are such that $\beta:\C\setminus \mathcal{S}\to \C$ is analytic and satisfies $\lim_{z\to+\infty}\beta(z)=1$. The branch of the logarithms in \eqref{def of Gcal}--\eqref{def of Gcal inf} is the principal one.
Note that this solution is not unique because we did not specify how $P^{(\infty)}$ must behave around $z_0$ and $\overline{z_0}$. 

The following lemma simplifies the expressions of $\mathcal{G}(0)$ and $\mathcal{G}_{\infty}$. 
\begin{lemma}\label{lemma:simplif of G0 and Ginf}
We have the identities
\begin{align}
\mathcal{G}(0) & = \bigg( \frac{|z_{0}|+\Re z_{0}}{2|z_{0}|^{2}} \bigg)^{\frac{1}{2}} \bigg( \frac{|z_{0}|+\frac{1}{a}-|z_{0}-\frac{1}{a}|}{|z_{0}|+|z_{0}-\frac{1}{a}|-\frac{1}{a}} \; \frac{|z_{0}|-\Re z_{0}}{|z_{0}|+\Re z_{0}} \bigg)^{\frac{\epsilon}{2}}, \label{simplif of Gcal at 0}\\
\mathcal{G}_{\infty} & = \bigg( \frac{2}{|z_{0}|+\Re z_{0}} \bigg)^{\frac{1}{2}}  \bigg( \frac{a (\im z_{0})^{2}}{2(|z_{0}-\frac{1}{a}|+\Re z_{0} - \frac{1}{a})} \bigg)^{\frac{\epsilon}{2}} \label{simplif of Gcal at infty}.
\end{align}
\end{lemma}
\begin{proof}
By \eqref{def of Gcal} and the definition of $\mathcal{R}$,
\begin{align*}
\log \mathcal{G}(0) = - \frac{|z_{0}|}{2\pi i} \int_{\mathcal{S}} \frac{\epsilon\log(1-as)-\log s}{\mathcal{R}_{+}(s)} \frac{ds}{s} = \frac{|z_{0}|}{4\pi i} \int_{\mathcal{C}_\mathcal{S}} \frac{\epsilon\log(1-as)-\log s}{\mathcal{R}(s)} \frac{ds}{s},
\end{align*}
where $\mathcal{C}_\mathcal{S}$ is a closed loop surrounding $\mathcal{S}$ once in the positive direction, but not enclosing $0$ and $\frac{1}{a}$. Let $M>\frac{1}{a}$ be a large but fixed constant, and let $\delta>0$ be a small but fixed constant. By deforming $\mathcal{C}_\mathcal{S}$ into $\mathcal{C}(0,M)\cup (-\mathcal{C}(0,\delta)) \cup (-M+i0_{+},-\delta+i0_{+}) \cup (-\delta-i0_{+},-M-i0_{+}) \cup (M-i0_{+},\frac{1}{a}-i0_{+})\cup (\frac{1}{a}+i0_{+},M+i0_{+})$, where $\pm \mathcal{C}(c',R')$ denotes the positively (resp. negatively) oriented circle centered at $c'$ of radius $R'>0$, we obtain
\begin{align*}
\log \mathcal{G}(0) = \frac{|z_{0}|}{2} \bigg\{ \int_{\mathcal{C}(0,M)} \frac{\epsilon\log(1-as)-\log s}{\mathcal{R}(s)} \frac{ds}{2\pi i s} + \int_{\mathcal{C}(0,\delta)} \frac{\log s}{\mathcal{R}(s)} \frac{ds}{2\pi i s} \\
+ \int_{-M}^{-\delta} \frac{-1}{\mathcal{R}(x)} \frac{dx}{x} + \int_{\frac{1}{a}}^{M} \frac{- \epsilon}{\mathcal{R}(x)} \frac{dx}{x} \bigg\}.
\end{align*}
Letting $M\to + \infty$ in the above expression, we find
\begin{align}\label{lol25}
\log \mathcal{G}(0) = \frac{|z_{0}|}{2} \bigg\{ \int_{\mathcal{C}(0,\delta)} \frac{\log s}{\mathcal{R}(s)} \frac{ds}{2\pi i s} + \int_{-\infty}^{-\delta} \frac{1}{|\mathcal{R}(x)|} \frac{dx}{x} + \int_{\frac{1}{a}}^{+\infty} \frac{- \epsilon}{\mathcal{R}(x)} \frac{dx}{x} \bigg\}.
\end{align}
Using primitives, we obtain
\begin{align*}
\int_{-\infty}^{-\delta} \frac{1}{|\mathcal{R}(x)|} \frac{dx}{x} & = \frac{1}{|z_{0}|} \log \bigg( \frac{\sqrt{(\im z_{0})^{2} + (\delta + \Re z_{0})^{2}}+\delta - |z_{0}|}{\sqrt{(\im z_{0})^{2} + (\delta + \Re z_{0})^{2}}+\delta + |z_{0}|} \bigg) \\
& = \frac{\log \delta}{|z_{0}|} + \frac{1}{|z_{0}|}\log \bigg( \frac{\Re z_{0} + |z_{0}|}{2|z_{0}|^{2}} \bigg) + \bigO(\delta), \qquad \mbox{as } \delta \to 0.
\end{align*}
Since
\begin{align*}
\int_{\mathcal{C}(0,\delta)} \frac{\log s}{\mathcal{R}(s)} \frac{ds}{2\pi i s} = -\frac{\log \delta}{|z_{0}|} + \bigO(\delta), \qquad \mbox{as } \delta \to 0,
\end{align*}
by letting $\delta \to 0$ in \eqref{lol25} we obtain
\begin{align}\label{lol26}
\log \mathcal{G}(0) = \frac{|z_{0}|}{2} \bigg\{ \frac{1}{|z_{0}|}\log \bigg( \frac{\Re z_{0} + |z_{0}|}{2|z_{0}|^{2}} \bigg) + \int_{\frac{1}{a}}^{+\infty} \frac{- \epsilon}{\mathcal{R}(x)} \frac{dx}{x} \bigg\}.
\end{align}
The remaining integral can also be evaluated explicitly using primitives:
\begin{align*}
\int_{\frac{1}{a}}^{+\infty} \frac{- \epsilon}{\mathcal{R}(x)} \frac{dx}{x} = \frac{\epsilon}{|z_{0}|} \log \bigg( \frac{|z_{0}|+\frac{1}{a}-|z_{0}-\frac{1}{a}|}{|z_{0}|+|z_{0}-\frac{1}{a}|-\frac{1}{a}} \; \frac{|z_{0}|-\Re z_{0}}{|z_{0}|+\Re z_{0}} \bigg).
\end{align*}
Substituting the above in \eqref{lol26}, we find \eqref{simplif of Gcal at 0}. The proof of \eqref{simplif of Gcal at infty} is similar (and slightly simpler) and we omit it.
\end{proof}

\subsection{Local parametrix near $z_{0}$}
We let $\mathcal D_{z_0}$ be a sufficiently small disk around $z_0$, with radius at most $\frac{1}{2}\Im z_0$, such that $\mathcal D_{z_0}$ lies entirely in the upper half plane.
We will construct a local parametrix $P^{(z_0)}$ in $\mathcal D_{z_0}\setminus\Sigma_T$, which satisfies exactly the same jump relations as $S$ on $\Sigma_T\cap\mathcal D_{z_0}$, and which is close to the global parametrix $P^{(\infty)}$ on the boundary $\partial\mathcal D_{z_0}$ of the disk.

\subsubsection*{RH problem for $P^{(z_{0})}$}
\begin{itemize}
\item[(a)] $P^{(z_{0})}:\mathcal{D}_{z_{0}}\setminus\Sigma_T\to \mathbb C^{2\times 2}$ is analytic.
\item[(b)] $P^{(z_{0})}_+(z)=P^{(z_{0})}_-(z)J_{T}(z)$ for $z\in \Sigma_T\cap\mathcal D_{z_0}$.
\item[(c)]$P^{(z_{0})}(z)=P^{(\infty)}(z)\big(I+o(1)\big)$ as $N\to +\infty$ uniformly for $z\in \partial \mathcal D_{z_0}$.
\end{itemize}
We will construct the local parametrix using a model RH solution related to the Airy function. Such a construction of a local Airy parametrix is rather standard, see e.g. \cite{DKMVZ}. 
We recall this construction here, but without going into much detail.

\begin{figure}
\begin{center}
\begin{tikzpicture}[master]
\node at (0,0) {\includegraphics[width=7cm]{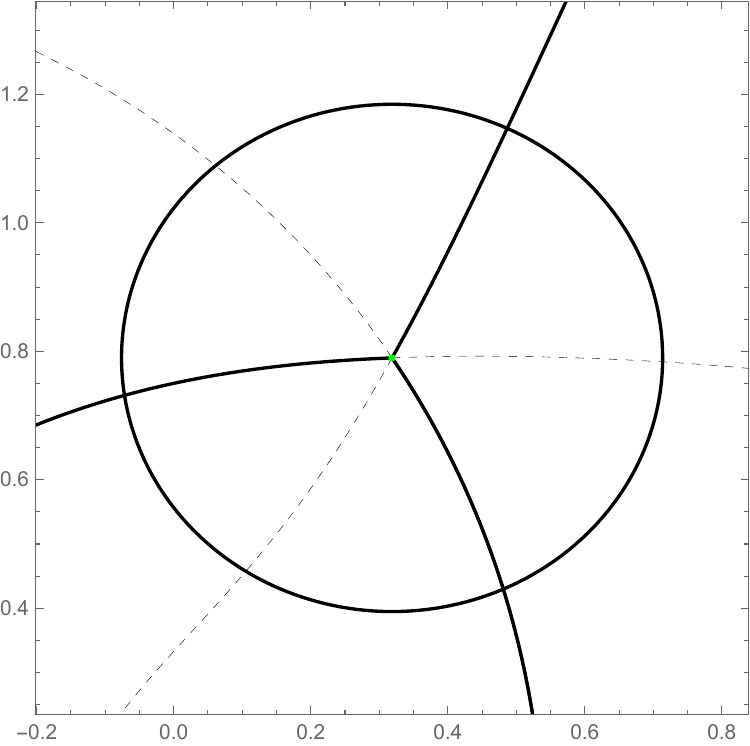}};


\draw[black,arrows={-Triangle[length=0.21cm,width=0.14cm]}]
($(0.16,0.14)+(62.4:1.35)$) --  ++(62.4:0.001);
\node at (1.8,2.9) {\small $\tilde{\Sigma}_{1}$};

\draw[black,arrows={-Triangle[length=0.21cm,width=0.14cm]}]
($(0.16,0.14)+(185:1.35)$) --  ++(185:0.001);
\node at (-2.8,-0.65) {\small $\tilde{\Sigma}_{0}$};

\draw[black,arrows={-Triangle[length=0.21cm,width=0.14cm]}]
($(0.16,0.14)+(-60.2:1.2)$) --  ++(180-63:0.001);
\node at (1.55,-2.6) {\small $\mathcal{S}$};

\node at ($(0.16,0.14)+(5:1.35)$) {\small $[1]$};
\node at ($(0.16,0.14)+(125:1.35)$) {\small $[2]$};
\node at ($(0.16,0.14)+(245:1.35)$) {\small $[3]$};

\end{tikzpicture}
\end{center}
\caption{\label{fig:Airy local param}
The solid curves are $\tilde \Sigma_0,\tilde\Sigma_1, \mathcal{S},\partial \mathcal{D}_{z_{0}}$ and the thin dashed curves are $\mathcal{T}_{0}$. The curves $\tilde \Sigma_0,\tilde\Sigma_1, \mathcal{S}$ divide $\mathcal{D}_{z_{0}}$ in three regions which are denoted by $[1]$, $[2]$ and $[3]$. The green dot is $z_{0}$. The figure is drawn for $a=0.78$, $\mu=0.76$ and $\kappa=0.44$.}
\end{figure}
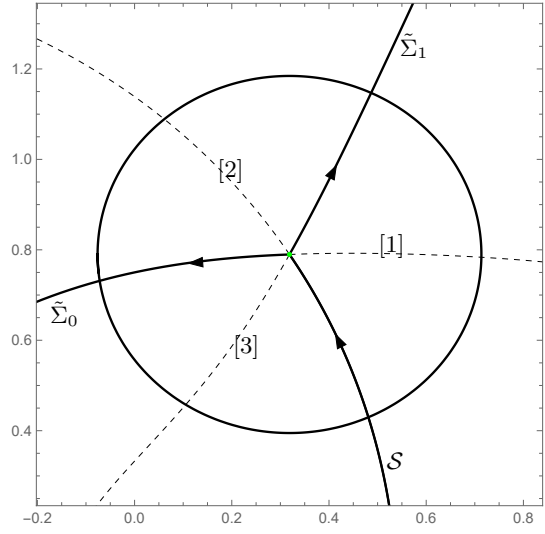

First, define the complex-valued functions $\{y_{j}(\zeta)\}_{j=0}^2$ by
\begin{align*}
y_{j}(\zeta) = e^{\frac{2\pi i j}{3}} \mbox{Ai}(e^{\frac{2\pi i j}{3}}\zeta), \qquad j = 0,1,2,
\end{align*}
where $\mbox{Ai}$ is the Airy function, 
and let the $2\times 2$-matrix valued functions $\{A_j(\zeta)\}_{j=1}^3$ be given by
\begin{align}
& A_{1}(\zeta) = -2 i \sqrt{\pi} \begin{pmatrix}
-y_{2}(\zeta) & -y_{0}(\zeta) \\
-y_{2}'(\zeta) & -y_{0}'(\zeta)
\end{pmatrix}, \label{def of A1} \\
& A_{2}(\zeta) = -2 i \sqrt{\pi} \begin{pmatrix}
-y_{2}(\zeta) & y_{1}(\zeta) \\
-y_{2}'(\zeta) & y_{1}'(\zeta)
\end{pmatrix}, \label{def of A2} \\
& A_{3}(\zeta) = -2 i \sqrt{\pi} \begin{pmatrix}
y_{0}(\zeta) & y_{1}(\zeta) \\
y_{0}'(\zeta) & y_{1}'(\zeta)
\end{pmatrix}. \label{def of A3}
\end{align}
These functions satisfy
\begin{align*}
& A_{1}(\zeta) = A_{2}(\zeta) \begin{pmatrix}
1 & -1 \\
0 & 1 
\end{pmatrix}, \qquad A_{2}(\zeta) = A_{3}(\zeta) \begin{pmatrix}
1 & 0 \\
1 & 1
\end{pmatrix}, \qquad 
A_{1}(\zeta) = A_{3}(\zeta) \begin{pmatrix}
1 & -1 \\
1 & 0
\end{pmatrix}.
\end{align*}
Moreover, 
\begin{align}\label{weak asymp for Ak}
A_{k}(\zeta) = \zeta^{-\frac{\sigma_{3}}{4}} \begin{pmatrix}
1 & i \\
1 & -i
\end{pmatrix} \Big[I + \frac{1}{8 \, \zeta^{3/2}} \begin{pmatrix}
-\frac{1}{6} & -i \\ -i & \frac{1}{6}
\end{pmatrix} + \bigO(\zeta^{-3}) \Big]e^{\frac{2}{3}\zeta^{3/2}\sigma_{3}}
\end{align}
as $\zeta \to \infty$ in the sector $S_{k}$ for $k = 1,2,3$, with
\begin{align}\label{Skdef}
S_{k} = \left\{ \zeta \in \mathbb{C}: \frac{2k - 3}{3}\pi + \delta \leq \arg \zeta \leq  \frac{2k+1}{3}\pi - \delta \right\}, \qquad k = 1,2,3,
\end{align}
and the branches of the complex powers in (\ref{weak asymp for Ak}) are such that $\zeta^u = e^{u\log|\zeta| + iu \arg \zeta}$ where $\arg \zeta$ belongs to $(-\pi/3,\pi)$, $(\pi/3,5\pi/3)$, and $(\pi,7\pi/3)$ for $\zeta$ in $S_1, S_2, S_3$, respectively.

We construct the local parametrix $P^{(z_0)}$ for $\zeta \in \mathcal{D}_{z_{0}}\setminus\Sigma_T$ of the following form:
\begin{align}\label{def of P}
P^{(z_{0})}(z) = E(z)A_k\big(-N^{\frac{2}{3}} f(z)\big) e^{-N \xi(z) \sigma_{3}} \bigg(\frac{(1-az)^{\epsilon}}{z}\bigg)^{\frac{\sigma_{3}}{2}}, \qquad z \in [k], \;\; k = 1,2,3,
\end{align}
where the principal branch is used for the roots. Here, the regions $[k]$, $k = 1,2,3$, denote the three components of $\mathcal{D}_{z_{0}} \setminus \Sigma_T$ as shown in Figure \ref{fig:Airy local param}, $f$ is the conformal map from $\mathcal{D}_{z_{0}}$ to a neighborhood of $0$ defined by
\begin{align} \label{def of q}
f(z) = -\bigg( \frac{3}{2}\xi(z) \bigg)^{\frac{2}{3}},
\end{align}
and $E$ denotes the $2 \times 2$-matrix valued function analytic on $\mathcal{D}_{z_{0}}$ defined by
\begin{align}\label{def of E local param}
E(z) = P^{(\infty)}(z)\bigg( \frac{(1-az)^{\epsilon}}{z}\bigg)^{-\frac{\sigma_{3}}{2}}\begin{pmatrix}
1 & i \\ 1 & -i
\end{pmatrix}^{-1} \Big( -N^{\frac{2}{3}} f(z) \Big)^{\frac{\sigma_{3}}{4}}.
\end{align}
As $z\to z_{0}$, we have
\begin{align}\label{eq:expf}
f(z) = -c_{0}(z-z_{0})\big( 1+c_{1}(z-z_{0}) + \bigO((z-z_{0})^{2}) \big),
\end{align}
where
\begin{equation}\label{def of c0 and c1}
c_0=(2\Im z_0)^{1/3}(-e^{\frac{\pi i}{4}}Q(z_0))^{2/3},\qquad c_1=\frac{2}{5}(\log Q)'(z_0)-\frac{i}{10\Im z_0},
\end{equation}
and the principal branch is used for $(-e^{\frac{\pi i}{4}}Q(z_0))^{2/3}$.
Note that we need to take $\mathcal D_{z_0}$ sufficiently small so that $f$ is conformal on $\mathcal D_{z_0}$.

From \eqref{weak asymp for Ak} and \eqref{def of P}, we obtain the matching condition
\begin{multline}
P^{(z_0)}(z)P^{(\infty)}(z)^{-1}  \\
=P^{(\infty)}(z)
\Big[I + \frac{1}{8N \, (-f(z))^{3/2}} \begin{pmatrix}
-\frac{1}{6} & -i{\frac{z}{(1-az)^{\epsilon}}} \\ -i{\frac{(1-az)^{\epsilon}}{z}} & \frac{1}{6}
\end{pmatrix} + \bigO(N^{-2}) \Big]P^{(\infty)}(z)^{-1}, \label{eq:matchingz0}
\end{multline}
as $N\to\infty$ uniformly for $z\in\partial D_{z_0}$.
This matching condition is also uniform in $\kappa\in[\delta,\kappa_2-\delta]$ and in $\mu\in[\delta,1-\delta]$, for any fixed $\delta>0$. As $\kappa\to 0$, we need to replace the error term $\bigO(N^{-2})$, using \eqref{z00}, \eqref{weak asymp for Ak}, \eqref{eq:expf} and \eqref{def of c0 and c1}, by 
\begin{align}\label{lol10}
\bigO(N^{-2}f(z)^{-3})=\bigO\left(N^{-2}(\Im z_0)^{-4}\right)=\bigO\left(N^{-2}\kappa^{-2}\right).
\end{align}
On the other hand, as $\kappa\to \kappa_2$ with $\mu-\kappa_{2} \geq \delta$, we use \eqref{z0kappa2}, \eqref{eq:asgamma}, \eqref{weak asymp for Ak}, \eqref{eq:expf} and \eqref{def of c0 and c1} to replace the error term $\bigO(N^{-2})$ by
\begin{align}\label{lol11}
\bigO(N^{-2}f(z)^{-3})=\bigO\left(N^{-2}|z_0-\gamma|^{-2}(\Im z_0)^{-4}\right)=\bigO\left(N^{-2}(\kappa_2-\kappa)^{-3}\right).
\end{align}
This extends the analysis to the range 
\begin{align*}
\frac{M}{N} \leq \kappa\leq \min\{\kappa_2-\frac{M}{N^{2/3}},\mu-\delta\}, \qquad {\delta \leq \mu \leq 1-\delta}
\end{align*}
for sufficiently large but fixed $M>0$, and for any fixed $\delta>0$.

%
%


\subsection{Local parametrix near $\overline{z_{0}}$}
Near $\overline{z_0}$, we need to construct a local parametrix satisfying the following conditions.
\subsubsection*{RH problem for $P^{(\overline{z_{0}})}$}
\begin{itemize}
\item[(a)] $P^{(\overline{z_{0}})}:\mathcal{D}_{\overline{z_{0}}}\setminus {\Sigma_{T}}\to \mathbb C^{2\times 2}$ is analytic.
\item[(b)] $P^{(\overline{z_{0}})}_+(z)=P^{(\overline{z_{0}})}_-(z)J_{T}(z)$ for $z\in {\Sigma_{T}\cap \mathcal{D}_{\overline{z_{0}}}}$.
\item[(c)]$P^{(\overline{z_{0}})}(z)=P^{(\infty)}\big(I+\bigO(N^{-1})\big)$ as $N\to +\infty$ uniformly for $z\in \partial {\mathcal{D}_{\overline{z_{0}}}}$.
\end{itemize}
We can simply use the local parametrix near $z_0$ for this, and define
\begin{align}\label{sym of local param}
P^{(\overline{z_{0}})}(z) = \overline{P^{(z_{0})}(\overline{z})}, \qquad z \in \mathcal{D}_{\overline{z_{0}}}\setminus {\Sigma_{T}}
\end{align}

By \eqref{eq:matchingz0} and by the symmetry $\overline{P^{(\infty)}({\overline{z}})}=P^{(\infty)}(z)$, we obtain the matching condition
\begin{multline}
P^{(\overline{z_0})}(z)P^{(\infty)}(z)^{-1}=P^{(\infty)}(z)\Big[I + \frac{1}{8N \, {\overline{(-f(\overline{z}))^{3/2}}}} \begin{pmatrix}
-\frac{1}{6} & i {\frac{z}{(1-az)^{\epsilon}}} \\ i {\frac{(1-az)^{\epsilon}}{z}} & \frac{1}{6}
\end{pmatrix} + \bigO(N^{-2}) \Big]P^{(\infty)}(z)^{-1}, \label{eq:matchingz0bar}
\end{multline}
as $N\to\infty$ uniformly for $z\in\overline{\partial D_{z_0}}$ and for $\kappa\in [\delta,\kappa_2-\delta]$ for any fixed $\delta>0$. If $\kappa\to 0$ (respectively $\kappa\to \kappa_{2}$ with $\kappa \leq \mu-\delta$), then the error term $\bigO(N^{-2})$ in \eqref{eq:matchingz0bar} must be replaced by \eqref{lol10} (respectively \eqref{lol11}).

\subsection{Small norm RH problem}

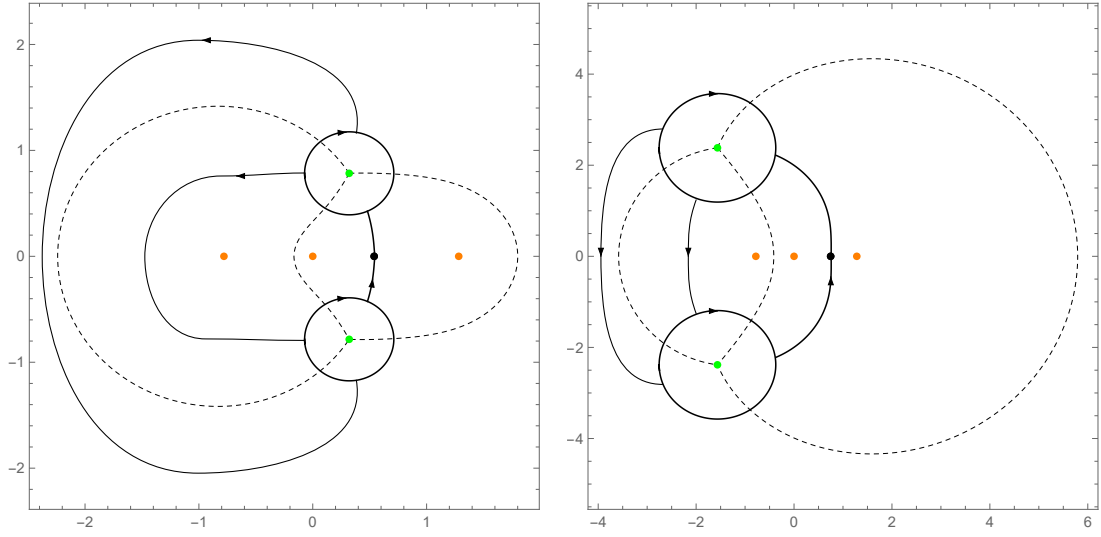
\begin{figure}
\begin{center}
\begin{tikzpicture}[master]
\node at (0,0) {\includegraphics[width=7cm]{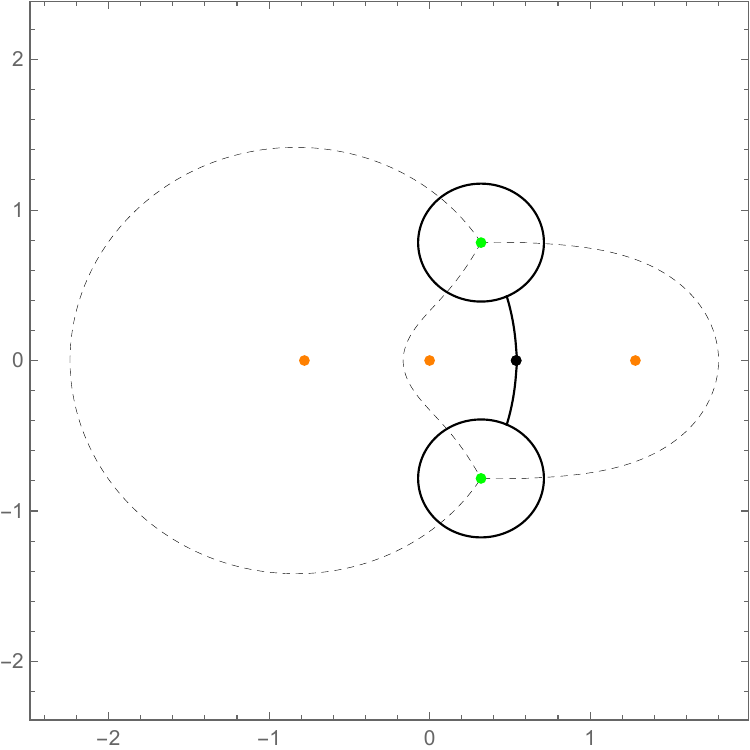}};

\draw[black,arrows={-Triangle[length=0.12cm,width=0.08cm]}]
(-0.95,3) --  ++(180:0.001);
\draw ($(0,0.13)+(0.99,1.11)+(80:0.53)$) to [out=80, in=0]
(-1,3) to [out=180, in=90] (-3.05,0.13) to [out=-90, in=180] (-1,-2.74) to [out=0, in=-80] ($(0,0.13)+(0.99,-1.11)+(-80:0.53)$);

\draw[black,arrows={-Triangle[length=0.12cm,width=0.08cm]}]
(-0.5,1.20) --  ++(180:0.001);
\draw ($(0,0.13)+(0.99,1.11)+(180:0.58)$) to [out=180, in=0]
(-0.7,1.2) to [out=180, in=90] (-1.7,0.13) to [out=-90, in=180] (-0.9,-0.96) to [out=0, in=-180] ($(0,0.13)+(0.99,-1.11)+(180:0.58)$);

\draw[black,arrows={-Triangle[length=0.12cm,width=0.08cm]}]
(1.3,-0.16) --  ++(82:0.001);

\draw[black,arrows={-Triangle[length=0.12cm,width=0.08cm]}]
($(0,0.14)+(1.05,1.11)+(100:0.54)$) --  ++(5:0.001);

\draw[black,arrows={-Triangle[length=0.12cm,width=0.08cm]}]
($(0,0.14)+(1.05,-1.11)+(100:0.56)$) --  ++(5:0.001);
\end{tikzpicture} \begin{tikzpicture}[slave]
\node at (0,0) {\includegraphics[width=7cm]{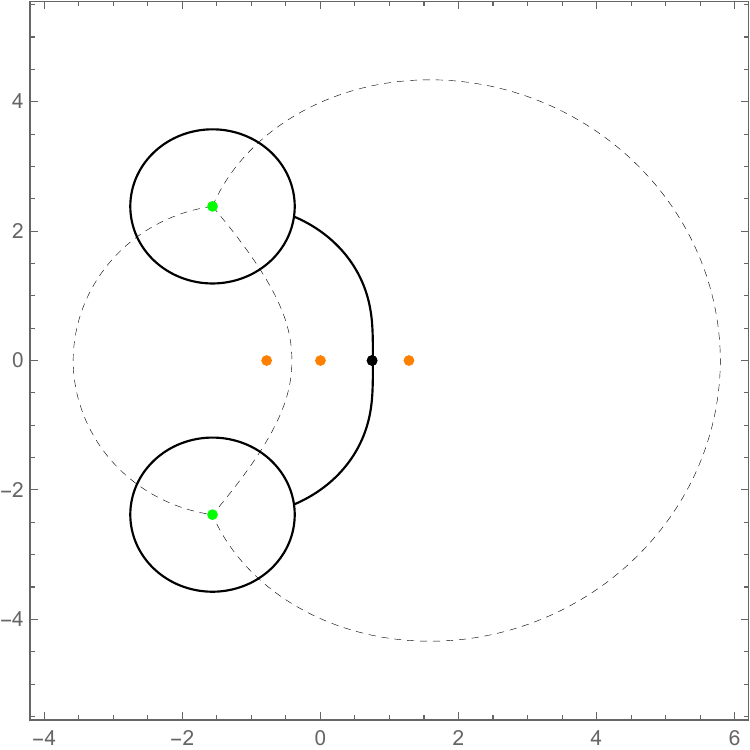}};

\draw[black,arrows={-Triangle[length=0.12cm,width=0.08cm]}]
(-3.05,0.13) --  ++(-90:0.001);
\draw ($(0,0.13)+(-1.51,1.43)+(160:0.77)$) to [out=180, in=90] (-3.05,0.13) to [out=-90, in=180] ($(0,0.13)+(-1.51,-1.43)+(-160:0.77)$);

\draw[black,arrows={-Triangle[length=0.12cm,width=0.08cm]}]
(-1.9,0.13) --  ++(-90:0.001);
\draw ($(0,0.13)+(-1.51,1.43)+(-113:0.73)$) to [out=-110, in=90]
(-1.9,0.13) to [out=-90, in=110] ($(0,0.13)+(-1.51,-1.43)+(113:0.73)$);

\draw[black,arrows={-Triangle[length=0.12cm,width=0.08cm]}]
($(0,0.13)+(-1.51,1.43)+(90:0.73)$) --  ++(0:0.001);

\draw[black,arrows={-Triangle[length=0.12cm,width=0.08cm]}]
($(0,0.13)+(-1.51,-1.43)+(90:0.71)$) --  ++(0:0.001);

\draw[black,arrows={-Triangle[length=0.12cm,width=0.08cm]}]
(-0.025,-0.13) --  ++(90:0.001);
\end{tikzpicture}
\end{center}
\caption{\label{fig:contour Sigma R}
In each image, the solid black curves are $\Sigma_{R}$ and the dashed curves are $\mathcal{T}_{0}$. The orange dots are $-a$, $0$ and $\frac{1}{a}$; the black dot is $\gamma$; and the green dots are $z_{0}$ and $\overline{z_{0}}$. The figures are drawn for $a=0.78$, $\mu=0.76$, and with $\kappa=0.44$ (left) and $\kappa=0.08$ (right).}
\end{figure}

Define
\begin{align}\label{def of R}
R(z) = \begin{cases}
S(z) P^{(\infty)}(z)^{-1}, & z \in \C\setminus (\overline{\mathcal{D}_{z_{0}}  \cup  \mathcal{D}_{\overline{z_{0}}}  }), \\
S(z) P^{(z_{0})}(z)^{-1}, & z \in \mathcal{D}_{z_{0}}, \\
S(z) P^{(\overline{z_{0}})}(z)^{-1}, & z \in \mathcal{D}_{\overline{z_{0}}}.
\end{cases}
\end{align}
$R$ satisfies the following RH problem.
\subsection*{RH problem for $R$}
\begin{itemize}
\item[(a)] $R:\C\setminus \Sigma_{R}\to \C^{2\times 2}$ is analytic, where
\begin{align*}
\Sigma_{R} = \partial \mathcal{D}_{z_{0}} \cup \partial \mathcal{D}_{\overline{z_{0}}} \cup \big( \Sigma_T \setminus (\mathcal{D}_{z_{0}} \cup \mathcal{D}_{\overline{z_{0}}}) \big)
\end{align*}
The contour $\Sigma_{R}$ is oriented as shown in Figure \ref{fig:contour Sigma R}. In particular, we orient the circles $\partial \mathcal{D}_{z_{0}}$ and
$\partial \mathcal{D}_{\overline{z_{0}}}$ in the clockwise direction.
\item[(b)] For $z \in \Sigma_{R}$, we have $R_{+}(z) = R_{-}(z)J_{R}(z)$, where
\begin{align*}
& J_{R}(z) = P^{(z_{0})}(z)P^{(\infty)}(z)^{-1}, & & z \in \partial \mathcal{D}_{z_{0}}, \\
& J_{R}(z) = P^{(\overline{z_{0}})}(z)P^{(\infty)}(z)^{-1}, & & z \in \partial \mathcal{D}_{\overline{z_{0}}}, \\
& J_{R}(z) = P^{(\infty)}(z)J_{S}(z)P^{(\infty)}(z)^{-1}, & & z \in \Sigma_{R} \setminus \big( \partial \mathcal{D}_{z_{0}} \cup \partial \mathcal{D}_{\overline{z_{0}}} \big).
\end{align*}
\item[(c)] $R(z)$ remains bounded as $z$ tends to the points of self-intersections of $\Sigma_{R}$.

\noindent As $z\to \infty$, $R(z) = I + \frac{R_{1}}{z} + \bigO(z^{-2})$ for some $R_{1}$ independent of $z$.
\end{itemize}

Observe that we have the symmetry $R(z) = \overline{R(\overline{z})}$, $z\in \C \setminus \Sigma_{R}$. 

On $\partial \mathcal{D}_{z_{0}} \cup \partial \mathcal{D}_{\overline{z_{0}}}$, by \eqref{eq:matchingz0} and \eqref{eq:matchingz0bar}, we have 
\begin{align}\label{lol12}
J_{R}(z) = I + \frac{J_{R}^{(1)}(z)}{N} + \bigO(N^{-2}), \qquad \mbox{as } N\to + \infty
\end{align}
uniformly for $z\in \partial \mathcal{D}_{z_{0}} \cup \partial \mathcal{D}_{\overline{z_{0}}}$, $\kappa\in [\delta,\kappa_2-\delta]$ and $\mu \in [\delta,1-\delta]$, for any fixed $\delta>0$ and for some matrix $J_{R}^{(1)}$ that can be explicitly computed. On the rest of the contour, using Corollary \ref{cor:contours2}, we obtain
\begin{align}\label{small jump liquid}
J_{R}(z) = I + \bigO(e^{-cN}), \qquad \mbox{as } N \to +\infty,
\end{align}
uniformly for $z\in \Sigma_{R}\setminus (\partial \mathcal{D}_{z_{0}} \cup \partial \mathcal{D}_{\overline{z_{0}}})$, $\kappa\in [\delta,\kappa_2-\delta]$ and $\mu \in [\delta,1-\delta]$, for some $c=c(\delta)>0$ and any fixed $\delta>0$.
By small-norm theory for RH problems \cite{DeiftZhou, DKMVZ}, we thus have
\begin{align}\label{eq:asR}
R(z) = I + \frac{R^{(1)}(z)}{N} + \bigO\bigg(\frac{N^{-2}}{1+|z|}\bigg), \qquad \mbox{as } N \to + \infty,
\end{align}
uniformly for $z\in \C \setminus  \Sigma_{R}$, where
\begin{align*}
R^{(1)}(z) = \frac{1}{2\pi i} \int_{\partial \mathcal{D}_{z_{0}} \cup \partial \mathcal{D}_{\overline{z_{0}}}} \frac{J_{R}^{(1)}(s)}{s-z}ds,
\end{align*}
and we recall that the orientation of the circles is clockwise. Therefore, for $z \in \C \setminus (\overline{\mathcal{D}_{z_{0}} \cup  \mathcal{D}_{\overline{z_{0}}}})$, we have
\begin{align}
R^{(1)}(z) = & \, \frac{1}{z-z_{0}} \mbox{Res}\bigg( J_{R}^{(1)}(s); s=z_{0} \bigg) + \frac{1}{(z-z_{0})^{2}} \mbox{Res}\bigg( (s-z_{0})J_{R}^{(1)}(s); s=z_{0} \bigg) \nonumber \\
& + \frac{1}{z-\overline{z_{0}}} \mbox{Res}\bigg( J_{R}^{(1)}(s); s = \overline{z_{0}} \bigg) + \frac{1}{(z-\overline{z_{0}})^{2}} \mbox{Res}\bigg( (s-\overline{z_{0}})J_{R}^{(1)}(s); s=\overline{z_{0}} \bigg). \label{Rp1p liquid}
\end{align}
The above residues can be computed explicitly using \eqref{eq:expf} and \eqref{sym of local param}:
\begin{align*}
& \mbox{Res}\bigg( (s-z_{0})J_{R}^{(1)}(s); s=z_{0} \bigg) = \frac{5\sqrt{\im z_{0}} \, e^{\frac{\pi i}{4}}}{48\sqrt{2} \,c_{0}^{\frac{3}{2}}} \begin{pmatrix}
1 & -i \mathcal{G}_{\infty}^{-2} \\ -i \mathcal{G}_{\infty}^{2} & -1
\end{pmatrix}, \\
& \mbox{Res}\bigg( (s-\overline{z_{0}})J_{R}^{(1)}(s); s=\overline{z_{0}} \bigg) = \overline{\mbox{Res}\bigg( (s-z_{0})J_{R}^{(1)}(s); s=z_{0} \bigg)}, \\
& \mbox{Res}\bigg( J_{R}^{(1)}(s); s=z_{0} \bigg) = \frac{e^{-\frac{\pi i}{4}}}{64\sqrt{2} \, c_{0}^{\frac{3}{2}}  \sqrt{\im z_{0}} } \begin{pmatrix}
-3-10i \, c_{1} \im z_{0} & (- \frac{19}{3}i - 10 \, c_{1} \im z_{0})\mathcal{G}_{\infty}^{-2} \\
(-\frac{19}{3}i - 10 \, c_{1} \im z_{0})\mathcal{G}_{\infty}^{2} & 3+10i \, c_{1} \im z_{0}
\end{pmatrix} \\
& \hspace{3.25cm} + \frac{\mathfrak{g}_{1}}{4\sqrt{2}c_{0}^{\frac{3}{2}}} \begin{pmatrix}
e^{\frac{\pi i}{4}} \sqrt{\im z_{0}} \, \mathfrak{g}_{1} & \frac{\sqrt{2} + e^{\frac{\pi i}{4}}\sqrt{\im z_{0}} \,\mathfrak{g}_{1}}{i \mathcal{G}_{\infty}^{2}} \\
\frac{\sqrt{2} - e^{\frac{\pi i}{4}}\sqrt{\im z_{0}} \,\mathfrak{g}_{1}}{-i \mathcal{G}_{\infty}^{-2}} & -e^{\frac{\pi i}{4}} \sqrt{\im z_{0}} \, \mathfrak{g}_{1}
\end{pmatrix} \\
& \mbox{Res}\bigg( J_{R}^{(1)}(s); s = \overline{z_{0}} \bigg) = \overline{\mbox{Res}\bigg( J_{R}^{(1)}(s); s=z_{0} \bigg)},
\end{align*}
where $\mathcal{G}_{\infty}$ is given by \eqref{simplif of Gcal at infty} and
\begin{align*}
\mathfrak{g}_{1} = - \frac{e^{-\frac{\pi i}{4}}}{\sqrt{2}} \frac{\sqrt{\im z_{0}}}{ 2\pi} \oint_{\mathcal{C}} \frac{\epsilon\log(1-as)- \log s}{\mathcal{R}(s)} \frac{ds}{s-z_{0}}.
\end{align*}
Here, the principal branch is used for the logarithms, and $\mathcal{C}\subset \C \setminus \big( (-\infty,0]\cup [\frac{1}{a},+\infty) \big)$ is a closed contour surrounding $\mathcal{S}$ once in the positive direction. {Using contour deformation, it is possible to simplify $\mathfrak{g}_{1}$ as in \eqref{gfrak1 simplif} below; the proof is similar to the proof of Lemma \ref{lemma:simplif of G0 and Ginf} and we omit it.}

We can again extend the range of validity of our asymptotic analysis as $\kappa\to\kappa_2$ with $\kappa \leq \mu-\delta$ and as $\kappa\to 0$: since $P^{(\infty)}$ remains bounded near $z_0,\overline{z_0}$ in these limits, we have
that \eqref{small jump liquid} holds uniformly for $M/N\leq \kappa\leq \min\{\kappa_2-MN^{-2/3},\mu-\delta\}$ with $M$ large, if we replace the error term $\bigO(e^{-cN})$ by $\bigO(e^{-c\kappa N})$ as $\kappa\to 0$, and by $\bigO(e^{-c(\kappa_{2}-\kappa)^{3/2} N})$ as $\kappa\to\kappa_2$ with $\kappa \leq \mu-\delta$. Similarly, \eqref{lol12} and \eqref{eq:asR} hold uniformly for $M/N\leq \kappa\leq \min\{\kappa_2-MN^{-2/3},\mu-\delta\}$ with $M$ large, if we replace the error terms $\bigO(N^{-2})$ by 
\begin{align}\label{lol13}
\bigO(N^{-2}\xi(z)^{-2})=\bigO\left(N^{-2}|\Im z_0|^{-4}\right)=\bigO\left(N^{-2}\kappa^{-2}\right),
\end{align}
as $\kappa\to 0$, and by
\begin{align}\label{lol14}
\bigO(N^{-2}\xi(z)^{-2})=\bigO\left(N^{-2}(z_0-\gamma)^{-2}|\Im z_0|^{-4}\right)=\bigO\left(N^{-2}(\kappa_2-\kappa)^{-3}\right),
\end{align}
as $\kappa\to\kappa_2$ {with $\kappa \leq \mu-\delta$}. All of these error terms are also uniform in $\mu\in[\delta,1-\delta]$, for any fixed $\delta>0$.

\subsection{Ratio asymptotics}
\begin{proposition}[Ratio asymptotics in the rough region]\label{prop:ratio asymp in rough region}
For any $\delta>0$ and for sufficiently large $M>0$, we have
\begin{multline*}
\log F_N^{m,k+1}(a;\epsilon)-\log F_N^{m,k}(a;\epsilon)
=NG(\kappa,\mu)+H(\kappa,\mu)+\frac{1}{N}F(\kappa,\mu)\\+\bigO\left(N^{-2}(\kappa_2-\kappa)^{-3}\right)+\bigO\left(N^{-2}\kappa^{-2}\right),
\end{multline*}
as $N\to\infty$, uniformly for
$\delta\leq \mu \leq 1-\delta$ and for
 $M/N\leq \kappa \leq \min\{\kappa_2 -M N^{-2/3},\mu-\delta\}$,
with $\mu=\frac{m}{N}$, $\kappa=k/N$. Here, 
\begin{align}
&\label{def:G1} G(\kappa,\mu) = -g(0;\kappa,\mu),\\
&\label{def:H1}H(\kappa,\mu)=\log \frac{\mathcal G_\infty}{\mathcal G(0)}+\log\cos\left(\frac{\arg z_0}{2}\right),\\
&\label{def:F1 F}
F(\kappa,\mu)=R^{(1)}_{22}(0)+\frac{R^{(1)}_{21}(0)}{\mathcal G_\infty^2}\tan\left(\frac{\arg z_0}{2}\right).
\end{align}
\end{proposition}
\begin{proof}
By the construction of $T$, \eqref{def of S}, and \eqref{def of R} (note that $0\in \C\setminus (\overline{\mathcal{D}_{z_{0}}\cup \mathcal{D}_{\overline{z_{0}}}})$), we have
\begin{align*}
Y_{22}(0)&=e^{-Ng(0)}S_{22}(0)\\
&=e^{-Ng(0)}\left(R_{21}(0)P^{(\infty)}_{12}(0)+R_{22}(0)P^{(\infty)}_{22}(0)\right).
\end{align*}
By \eqref{def:Pinf} and \eqref{def of gamma}, we have
\begin{align*}
P^{(\infty)}_{22}(0)=\frac{\mathcal G_\infty}{\mathcal G(0)}\cos\left(\frac{\arg z_0}{2}\right),\qquad P^{(\infty)}_{12}(0)=\frac{1}{\mathcal G(0) \mathcal G_\infty}\sin\left(\frac{\arg z_0}{2}\right).
\end{align*}
Using the large $N$ asymptotic expansion \eqref{eq:asR} of $R$, we obtain\begin{align*}
Y_{22}(0)
&=e^{-Ng(0)} \frac{\mathcal G_\infty}{\mathcal G(0)}\cos\left(\frac{\arg z_0}{2}\right)\left(1+\frac{1}{N}\left[R^{(1)}_{22}(0)+\frac{R^{(1)}_{21}(0)}{\mathcal G_\infty^{2}}\tan\left(\frac{\arg z_0}{2}\right)\right]+\mathcal O(N^{-2})\right),
\end{align*}
where we need to replace the error term by $\bigO\left(N^{-2}(\kappa_2-\kappa)^{-3}\right)$ or by $\bigO\left(N^{-2}\kappa^{-2}\right)$ as $\kappa\to\kappa_2$ or as $\kappa\to 0$, respectively (see \eqref{lol13} and \eqref{lol14}).
So
\begin{align*}
\log Y_{22}(0) =-Ng(0)+\log\frac{\mathcal G_\infty}{\mathcal G(0)}+\log\cos\left(\frac{\arg z_0}{2}\right)  +\frac{1}{N}\left[R^{(1)}_{22}(0)+\frac{R^{(1)}_{21}(0)}{\mathcal G_\infty^{2}}\tan\left(\frac{\arg z_0}{2}\right)\right]+\mathcal O(N^{-2}),
\end{align*}
again with modified error term as $\kappa\to \kappa_2$ or as $\kappa\to 0$. 
The result now follows from \eqref{eq:ratioid2}.
\end{proof}

We can evaluate $G,H,F$ explicitly in terms of $z_0$. For $H$, we use \eqref{simplif of Gcal at 0}--\eqref{simplif of Gcal at infty} and obtain
\begin{align}
H(\kappa,\mu)&= \frac{\epsilon}{2}\log\left(\bigg( \frac{a (\im z_{0})^{2}}{2(|z_{0}-\frac{1}{a}|+\Re z_{0} - \frac{1}{a})} \bigg)
\bigg( \frac{|z_{0}|+|z_{0}-\frac{1}{a}|-\frac{1}{a}}{|z_{0}|+\frac{1}{a}-|z_{0}-\frac{1}{a}|} \; \frac{|z_{0}|{+}\Re z_{0}}{|z_{0}|{-}\Re z_{0}} \bigg)\right)\nonumber\\
&\quad +\log\left(\frac{2|z_0|}{|z_0|+\Re z_0}\right)+\log \cos\left(\frac{\arg z_0}{2}\right)  .\label{def:H}
\end{align}
For $G$, we evaluate \eqref{eq:exprg} at $z=0$ and use \eqref{def of ell} to obtain
\begin{align}
G(\kappa,\mu) & = \frac{\mu-\kappa}{2} \bigg\{ \log \bigg( \frac{4|z_{0}|^{2}}{(\im z_{0})(|z_{0}|+\Re z_{0})} \bigg) - \frac{\Re z_{0}}{|z_{0}|} \log \bigg( \frac{|z_{0}|+\Re z_{0}}{\im z_{0}} \bigg) \bigg\} \nonumber \\
& \quad - \frac{1-\mu}{2}\bigg\{ \log \bigg( \frac{|z_{0}| \, |z_{0}-\frac{1}{a}| + \frac{1}{a}\Re z_{0} - |z_{0}|^{2}}{\frac{\im z_{0}}{2}(|z_{0}-\frac{1}{a}|+\frac{1}{a} - \Re z_{0})} \bigg) + \frac{\Re z_{0}-\frac{1}{a}}{|z_{0}-\frac{1}{a}|} \log \bigg( \frac{|z_{0}|+\Re z_{0}}{\im z_{0}} \bigg) \bigg\} \nonumber \\
& \quad - \frac{\mu}{2}\bigg\{ \log \bigg( \frac{|z_{0}| \, |z_{0}+a| + |z_{0}|^{2} + a \, \Re z_{0}}{\frac{\im z_{0}}{2}(|z_{0}+a|+a+\Re z_{0})} \bigg) - \frac{\Re z_{0}+a}{|z_{0}+a|} \log \bigg( \frac{|z_{0}|+\Re z_{0}}{\im z_{0}} \bigg) \bigg\} \label{def:G}.
\end{align} 
Finally, for $F$, we use \eqref{Rp1p liquid} and obtain
\begin{align}
& F(\kappa) = \frac{\sqrt{\im z_{0}}}{2\sqrt{2}} \bigg\{ \tan \Big( \frac{\arg z_{0}}{2} \Big) \im \bigg( \frac{5}{12} \frac{1}{z_{0}^{2}e_{0}}  + \frac{5}{8} \frac{e_{1}}{z_{0}e_{0}} - \frac{\mathfrak{g}_{1}^{2}}{z_{0}e_{0}} \bigg) - \re \bigg( \frac{5}{12} \frac{1}{z_{0}^{2}e_{0}}  + \frac{5}{8} \frac{e_{1}}{z_{0}e_{0}} - \frac{\mathfrak{g}_{1}^{2}}{z_{0}e_{0}} \bigg)  \bigg\} \nonumber \\
& - \frac{\im(\frac{3}{16}\frac{1}{z_{0}e_{0}})}{2\sqrt{2}\sqrt{\im z_{0}}} + \tan \Big( \frac{\arg z_{0}}{2} \Big) \bigg\{ \frac{\im(\frac{19}{48}\frac{1}{z_{0}e_{0}})}{2\sqrt{2}\sqrt{\im z_{0}}} + \im \bigg[ \frac{e^{-\frac{\pi i}{4}}\mathfrak{g}_{1}}{2z_{0}e_{0}} \bigg] \bigg\}, \label{def:F}
\end{align}
where $e_{0}$, $e_{1}$ and $\mathfrak{g}_{1}$ are given by
\begin{align}
& e_{0} = -\sqrt{2}\sqrt{\im z_{0}}Q(z_{0}), \qquad e_{1} = \frac{2}{5}\frac{Q'(z_{0})}{Q(z_{0})} - \frac{i}{10 \, \im z_{0}}, \nonumber \\
& Q(z)=\frac{\kappa+1-\mu}{2}\frac{z-\gamma}{z(z+a)(z-1/a)},\qquad \gamma=\frac{\mu-\kappa}{1-\mu+\kappa}\frac{1}{|z_0|}, \nonumber \\
& \mathfrak{g}_{1} = e^{\frac{\pi i}{4}}\sqrt{2} \sqrt{\im z_{0}} \bigg( \frac{1}{z_{0}+|z_{0}|} + \frac{\epsilon}{\frac{1}{a}-z_{0}+|\frac{1}{a}-z_{0}|} \bigg). \label{gfrak1 simplif}
\end{align}
Note that all the quantities defined above have a highly nontrivial dependence on $\kappa$ and $\mu$, via $z_0=z_0(\kappa,\mu)$.

It is crucial for us to understand the asymptotics of the quantities $H(\kappa,\mu)$, $F(\kappa,\mu)$ and $G(\kappa,\mu)$ as 
$\kappa\to 0$ and as $\kappa\to\kappa_2$.

\begin{proposition}\label{prop:GHF as kappa tends to 0}
Let $\mu\in(0,1)$. 
As $\kappa \to 0$, we have
\begin{align*}
& G(\kappa,\mu) = G_0 +G_* \kappa\log \kappa + G_1\kappa +  \bigO(\kappa^{2}), \qquad \partial_{\kappa}G(\kappa,\mu) = G_* \log \kappa + G_*+G_1+\bigO(\kappa), \\
& \partial_{\kappa}^{j} G(\kappa,\mu) = \bigO(\kappa^{-j+1}) \qquad \mbox{for } j=2,3,4, \\
& H(\kappa,\mu) = -\frac{1}{2}\log \kappa + H_0 + \bigO(\kappa), \qquad F(\kappa,\mu) =  \frac{F_{-1}}{\kappa} + \bigO(1),\\
&\partial_{\kappa}^jH(\kappa,\mu) = \bigO(\kappa^{-j}), \quad \partial_\kappa^j F(\kappa,\mu) = \bigO(\kappa^{-j-1}) \qquad \mbox{for }j=1,2,3,4,\\
\end{align*}
with \begin{align}
&G_0=\varphi_0(x_0), \qquad G_*=1, \qquad G_1=-1-\log\big(|x_{0}|^{2}\varphi_{0}''(x_{0})\big).\label{exprconstantslast} \\
& H_0=\frac{1}{2}\log\big(|x_{0}|^{2}\varphi_{0}''(x_{0})\big)
+\epsilon\log(1-ax_{0}),\qquad F_{-1}=\frac{1}{12}, \label{exprconstantsfirst}
\end{align}
The error terms are moreover uniform for $\mu\in[\delta,1-\delta]$ for any fixed $\delta >0$.
\label{prop:kappatozero}
\end{proposition}
\begin{proof}
Using \eqref{def:H} and Proposition \ref{prop:z0}, we obtain after a straightforward computation that 
\[H(\kappa,\mu)=-\frac{1}{2}\log\kappa +H_0 + \bigO(\kappa)\qquad \mbox{as $\kappa\to 0$},\]
for some $H_0\in\mathbb R$, and that $\partial_{\kappa}H(\kappa,\mu) = \bigO(\kappa^{-j})$ as $\kappa \to 0$ for $j=1,2,3,4$.
Similarly, by \eqref{def:F} and Proposition \ref{prop:z0}, we can prove that
\[F(\kappa,\mu)=\frac{F_{-1}}{\kappa}+\bigO(1)\qquad \mbox{as $\kappa\to 0$},\]
for some $F_{-1}\in\mathbb R$, and that $\partial_\kappa^j F(\kappa,\mu)=\bigO(\kappa^{-j-1})$ as $\kappa\to 0$ for $j=1,2,3,4$.

Finally, by \eqref{def:G} and Proposition \ref{prop:z0}, we find after a lengthy computation that 
\[G(\kappa,\mu) = G_0 +G_* \kappa\log \kappa + G_1\kappa +  \bigO(\kappa^{2}),\]
and that $\partial_{\kappa}G(\kappa,\mu) = G_1 \log \kappa  + G_*+G_1+\bigO(\kappa)$, $\partial_{\kappa}^{j} G(\kappa,\mu) = \bigO(\kappa^{-j+1})$ for $j=2,3,4,$ as $\kappa\to 0$, for some constants $G_0,G_*,G_1\in\mathbb R$.

To compute the constants $H_0, F_{-1}, G_0, G_*, G_1$, we can now use Proposition \ref{prop:ratio asymp in max region} for $k$ fixed but sufficiently large: this gives\begin{align*}
\log F_N^{m,k+1}(a;\epsilon)-\log F_N^{m,k}(a;\epsilon)
& =N\varphi_{0}(x_{0})-\Big(k-\frac{1}{2}\Big)\log N + \log\big((k-1)!\big)  \nonumber \\
& \hspace{-1.5cm} - \Big( k-\frac{1}{2} \Big) \log\big(|x_{0}|^{2}\varphi_{0}''(x_{0})\big)+\epsilon \log(1-ax_{0}) - \frac{1}{2}\log (2\pi)+\bigO(N^{-1/2}),
\end{align*}
as $N\to\infty$,
which we can rewrite in terms of $\kappa=k/N$ as
\begin{align}
\log F_N^{m,k+1}(a;\epsilon)-\log F_N^{m,k}(a;\epsilon)
& =N\varphi_{0}(x_{0})+N\kappa\log\kappa 
-\kappa N-\kappa N\log\big(|x_{0}|^{2}\varphi_{0}''(x_{0})\big)-\frac{1}{2}\log \kappa \nonumber\\
&+\frac{1}{2}\log\big(|x_{0}|^{2}\varphi_{0}''(x_{0})\big)+\epsilon\log(1-ax_{0}) + S(k)+\bigO(N^{-1/2})
,\label{compat1}
\end{align}
with
\[S(k)=\log\big((k-1)!\big)-(k-\frac{1}{2})\log k+k- \frac{1}{2}\log (2\pi), \]
which is $\frac{1}{12{k}}+\bigO(1/k^3)$ as $k\to\infty$ by Stirling's approximation.

Now we compare this with Proposition \ref{prop:ratio asymp in rough region} in which we substitute the above expansions for $H(\kappa,\mu)$, $F(\kappa,\mu)$, $G(\kappa,\mu)$: this gives
\begin{align*}
\log F_N^{m,k+1}(a;\epsilon)-\log F_N^{m,k}(a;\epsilon)
&=NG_0+G_* N \kappa \log\kappa+G_1N\kappa
-\frac{1}{2}\log\kappa+ H_0
+\frac{F_{-1}}{k}\\
&
+\bigO(k^2N^{-1})
+\bigO(kN^{-1})
+\bigO(N^{-1}) +\bigO(k^{-2}),
\end{align*}
as $N\to\infty$ for $k$ large.
Comparing
the latter with \eqref{compat1}, we obtain \eqref{exprconstantslast}--\eqref{exprconstantsfirst}.
\end{proof}

\begin{proposition}
As $\kappa \to \kappa_2$, we have
\begin{align*}
& G(\kappa,\mu) = \tilde G_2(\kappa_{2}-\kappa)^{2}  +  \bigO((\kappa_{2}-\kappa)^{3}),   \\
&\partial_{\kappa} G(\kappa,\mu) = \bigO(\kappa_{2}-\kappa),\qquad \partial_{\kappa}^{j} G(\kappa,\mu) = \bigO(1) \qquad \mbox{for }j=2,3,4, \\
& H(\kappa,\mu) = \bigO(\kappa_{2}-\kappa), \qquad  F(\kappa) = \frac{\tilde F_{-1}}{\kappa_{2}-\kappa} + \bigO(1), \\
&\partial_{\kappa}^jH(\kappa,\mu) = \bigO(1),\qquad \partial_\kappa^j F(\kappa,\mu) = \bigO((\kappa_{2}-\kappa)^{-j-1}) \qquad \mbox{for }j=1,2,3,4,
\end{align*}
with
\begin{equation}\label{exprconstants}
\tilde G_2=\frac{\left(s^*\right)^3}{4},\qquad \tilde F_{-1}=\frac{1}{8}.
\end{equation}
The error terms are moreover uniform for $\mu\in[\delta,1-\delta]\setminus (\frac{a^{2}}{1+a^{2}}-\delta,\frac{a^{2}}{1+a^{2}}+\delta)$ for any fixed $\delta>0$.
\label{prop:kappatokappa2}
\end{proposition}
\begin{proof}
We proceed similarly as in the proof of Proposition \ref{prop:GHF as kappa tends to 0}.
Using \eqref{def:H} and Proposition \ref{prop:z0}, we obtain  
\begin{align*}
H(\kappa,\mu)=\bigO(\kappa_2-\kappa),\qquad\partial_{\kappa}^jH(\kappa,\mu) = \bigO(1), \qquad \mbox{as }\kappa\to \kappa_2, \quad j=1,2,3,4.
\end{align*}
By \eqref{def:F}, Proposition \ref{prop:z0} and \eqref{Rp1p liquid}, we find
\[F(\kappa,\mu)=\frac{\tilde F_{-1}}{\kappa_{2}-\kappa}+\bigO(1)\qquad \mbox{as $\kappa\to \kappa_2$},\]
for some $\tilde F_{-1}\in\mathbb R$, and $\partial_\kappa^j F(\kappa,\mu)=\bigO((\kappa_2-\kappa)^{-j-1})$ as $\kappa\to \kappa_2$ for $j=1,2,3,4$.
By \eqref{def:G} and Proposition \ref{prop:z0}, we obtain after a long computation that 
\[G(\kappa,\mu) = \tilde G_2(\kappa_2-\kappa)^2 +  \bigO((\kappa_2-\kappa)^{3}),\]
and that $\partial_{\kappa}G(\kappa,\mu) = \bigO(\kappa_2-\kappa)$ as $\kappa\to \kappa_{2}$, for some constant $\tilde G_2\in\mathbb R$, as well as the fact that $\partial_\kappa^j G(\kappa,\mu)$ is bounded for $j=2,3,4$.

To compute the constant $\tilde G_2$, we use Proposition \ref{prop:ratio asymp in TW region} and \eqref{asymp of u and q update} for $\kappa_2-\kappa=N^{-\frac{2}{3}+\eta}$ with $\eta\in (\frac{1}{15},\frac{1}{6})$. We get
\begin{align}
& \log F_N^{m,k+1}(a;\epsilon)-\log F_N^{m,k}(a;\epsilon)
= \frac{\left(s^*\right)^3}{4}N^{-\frac{1}{3}+2\eta} \nonumber \\
& \hspace{2cm} +\frac{1}{8}N^{-\frac{1}{3}-\eta}+\bigO\left(N^{-\frac{2}{3}+4\eta}+N^{-\frac{1}{3}-\frac{5}{2}\eta}\right),\label{compat2}
\end{align}
as $N\to\infty$.
We compare this with Proposition \ref{prop:ratio asymp in rough region} in which we substitute the above expansions of $H(\kappa,\mu)$, $F(\kappa,\mu)$, $G(\kappa,\mu)$,
\begin{align}\label{compat2 bis}
\log F_N^{m,k+1}(a;\epsilon)-\log F_N^{m,k}(a;\epsilon)
&=\tilde G_2N^{-\frac{1}{3}+2\eta}+\tilde F_{-1}N^{-\frac{1}{3}-\eta}
+\bigO(N^{-1+3\eta}+N^{-\frac{2}{3}+\eta}  + N^{-3\eta}),
\end{align}
as $N\to\infty$.
Comparing 
the latter with \eqref{compat2}, we find the expression for $\tilde G_2$ in \eqref{exprconstants}. Note that $\tilde F_{-1}$ cannot be obtained by comparing \eqref{compat2} with \eqref{compat2 bis}, because $\eta$ cannot be chosen in such a way that the $\bigO$-terms in both \eqref{compat2} and \eqref{compat2 bis} are simultaneously small. However, a long and tedious computation, which uses the explicit expressions of the constants $\tilde{z}_{0}^{(0)}$, $\tilde{z}_{0}^{(1/2)}$ from Proposition \ref{prop:z0}, shows that $\tilde F_{-1}=\frac{1}{8}$. \end{proof}

\subsection{Proof of Theorem \ref{thm:rough}}

We will now sum the result from Proposition \ref{prop:ratio asymp in rough region} over $k$ running from $k_-$ to $k_+$, where $k_{-},k_{+}\in[M+1, N(\kappa_2-MN^{-2/3})]\cap \N$. We will explicitly evaluate the summation of the main terms, and estimate the summation of the error terms.
The key will be to recognize the sums as Riemann sums of integrals. We record the following elementary lemma to estimate the difference between Riemann sums and integrals.

\begin{lemma}\label{lemma:Riemannsum}
Let $k_-<k_+$ be integers and let $N\in\mathbb N$.
Write 
\begin{align}\label{def of kappa pm}
\kappa_-=\frac{k_--\frac{1}{2}}{N}, \qquad \kappa_+=\frac{k_++\frac{1}{2}}{N}.
\end{align}
Let $h$ be $C^4$ on $[\kappa_-,\kappa_+]$.
We have
\begin{align*}
\sum_{k=k_-}^{k_+}h(k/N)=N
\int_{\kappa_-}^{\kappa_+}h(x)dx-\frac{1}{24 N}(h'(\kappa_+)-h'_-(\kappa_-))+{\bigO(N^{-4}\Sigma_{h^{(4)}})},\qquad N \to\infty,
\end{align*}
where
\[\Sigma_f:=\sum_{k=k_-}^{k_+}\max_{\kappa \in [k/N,(k+1)/N]} |f(\kappa)|.\]
\end{lemma}
\begin{proof}
Using a mid-point rule for the integrals of $h$ and $h''$, we obtain that 
\begin{align*}
& \int_{\kappa_-}^{\kappa_+}h(x)dx=\frac{1}{N}\sum_{k=k_-}^{k_+}h(k/N)+ \frac{1}{24 N^3}\sum_{k=k_-}^{k_+}h''(k/N)+\bigO(N^{-5}\Sigma_{h^{(4)}}), \\
& \int_{\kappa_-}^{\kappa_+}h''(x)dx=\frac{1}{N}\sum_{k=k_-}^{k_+}h''(k/N)+ \bigO(N^{-3}\Sigma_{h^{(4)}}),
\end{align*}
as $N\to\infty$.
This implies that
\begin{align*}
\int_{\kappa_-}^{\kappa_+}h(x)dx&=\frac{1}{N}\sum_{k=k_-}^{k_+}h(k/N)+ \frac{1}{24 N^2}\int_{\kappa_-}^{\kappa_+}h''(x)dx+\bigO(N^{-5}\Sigma_{h^{(4)}})\\
&=\frac{1}{N}\sum_{k=k_-}^{k_+}h(k/N)+ \frac{1}{24 N^2}(h'(\kappa_+)-h'(\kappa_-))+{\bigO(N^{-5}\Sigma_{h^{(4)}})}.
\end{align*}
Consequently,
\begin{align*}
\sum_{k=k_-}^{k_+}h(k/N)&=N
\int_{\kappa_-}^{\kappa_+}h(x)dx-\frac{1}{24 N}(h'(\kappa_+)-h'(\kappa_-))+\bigO(N^{-4}\Sigma_{h^{(4)}}).
\end{align*}
\end{proof}

\begin{lemma}\label{lemma:Riemann sums}
Let $M\in \N$ be fixed but sufficiently large, and let $\delta>0$ be fixed. As $N \to +\infty$, we have 
\begin{align}\label{lol18}
& \sum_{k=k_-}^{k_{+}} N G(\kappa=\tfrac{k}{N},\mu) = N^{2} \int_{\kappa_-}^{\kappa_{+}} G(\kappa,\mu)d\kappa  - \frac{ \partial_{\kappa}G(\kappa_{+},\mu)-\partial_{\kappa}G(\kappa_{-},\mu)}{24} + {\bigO(\kappa_{-}^{-2}N^{-2})}, 
\end{align}
uniformly for $k_{-}, k_+\in[M+1, \min\{N(\kappa_2-MN^{-2/3}),N(\mu-\delta)\}]{\cap \N}$ and $\mu\in [\delta,1-\delta]$, and where $\kappa_{+},\kappa_{-}$ are as in \eqref{def of kappa pm}.
In particular, the following hold.
\begin{itemize}
\item As $N \to +\infty$, we have
\begin{align}
& \sum_{k=M+1}^{k_{+}} N G(\tfrac{k}{N},\mu) = N^{2} \int_{0}^{\kappa_{+}} G(\kappa,\mu)d\kappa -
G_0\bigg(M+\frac{1}{2}\bigg)N -\frac{6M^2+6M+1}{12}\log \frac{2M+1}{2N}  \nonumber \\
& -\frac{ \partial_{\kappa}G(\kappa_{+},\mu)}{24} +\frac{5}{48}-\frac{G_1}{12} +\frac{1}{4}\left(M^2+M\right)(1-2G_1)
 +\bigO\bigg(\frac{1}{M^2} +  \frac{M^{3}}{N} \bigg), \label{lol19}
\end{align}
uniformly for $k_+\in[M+1, \min\{N(\kappa_2-MN^{-2/3}),N(\mu-\delta)\}]\cap \N$ and $\mu\in [\delta,1-\delta]$.
\item Let $k_{+} = \lfloor N(\kappa_{2}-\frac{M}{N^{2/3}})\rfloor$. As $N \to +\infty$, we have
\begin{align}
\sum_{k=k_-}^{k_{+}} N G(\tfrac{k}{N},\mu) = N^{2} \int_{\kappa_-}^{\kappa_{2}} G(\kappa,\mu)d\kappa  + \frac{\partial_{\kappa}G(\kappa_-,\mu)}{24} - \frac{\tilde G_2}{3}M^{3} + \bigO\bigg(\frac{1}{\kappa_{-}^{2}N^{2}}+ \frac{M^{2}}{N^{1/3}}\bigg).\label{lol19b}
\end{align}
uniformly for $k_-\in[M+1, N(\kappa_2-MN^{-2/3})]\cap \N$ and for $\mu \in [\delta,1-\delta]\setminus (\frac{a^{2}}{1+a^{2}}-\delta,\frac{a^{2}}{1+a^{2}}+\delta)$.
\end{itemize} 
\end{lemma}
\begin{proof}
The result \eqref{lol18} follows from Lemma \ref{lemma:Riemannsum} applied to the function $h(\kappa)=G(\kappa,\mu)$ and from the estimates for $\partial_\kappa^4 G(\kappa,\mu)$ obtained in Propositions \ref{prop:kappatozero} and \ref{prop:kappatokappa2}.

To derive \eqref{lol19}, we again use Proposition \ref{prop:kappatozero}, with $k_-=M+1$. We use it to obtain
\begin{align*}
&\sum_{k=k_-}^{k_{+}} N G(\tfrac{k}{N},\mu) = N^{2} \int_{\kappa_-}^{\kappa_{+}} G(\kappa,\mu)d\kappa  - \frac{ \partial_{\kappa}G(\kappa_{+},\mu)-\partial_{\kappa}G(\kappa_{-},\mu)}{24} + {\bigO(M^{-2})}\\
& =N^{2} \int_{0}^{\kappa_+} G(\kappa,\mu)d\kappa -N^{2} \int_0^{\kappa_-} G(\kappa,\mu)d\kappa  + \frac{1}{24}\left(- \partial_{\kappa}G(\kappa_{+},\mu)+\log\frac{2M+1}{2N}+1+G_1 \right)+ {\bigO(M^{-2})},
\end{align*}
which yields the result after a straightforward computation.
We obtain \eqref{lol19b} similarly using Proposition \ref{prop:kappatokappa2}.
\end{proof}

\begin{lemma}\label{lemma:Riemann sums1}
Let $M\in \N$ be fixed but sufficiently large, and let $\delta>0$ be fixed. As $N \to +\infty$, we have  
\begin{equation}\label{lol20}
\sum_{k=k_{-}}^{k_{+}}H(\kappa=\tfrac{k}{N}) = N \int_{\kappa_{-}}^{\kappa_{+}} H(\kappa,\mu)d\kappa + \bigO\bigg( \frac{1}{\kappa_{-}N} \bigg).
\end{equation}
uniformly for $k_{-}, k_+\in[M+1, \min\{N(\kappa_2-MN^{-2/3}),N(\mu-\delta)\}]{\cap \N}$  and $\mu\in [\delta,1-\delta]$, and where $\kappa_{+},\kappa_{-}$ are as in \eqref{def of kappa pm}.
In particular, the following hold.
\begin{itemize}
\item As $N \to +\infty$, we have
\begin{align*}
\hspace{-0.3cm} \sum_{k=M+1}^{k_{+}}H(\kappa=\tfrac{k}{N},\mu) = N \int_{0}^{\kappa_{+}} H(\kappa,\mu)d\kappa +
 \frac{2M+1}{4}\log \frac{2M+1}{2N}
- \frac{2M+1}{4}(2H_0+1) 
 + \bigO\bigg( \frac{1}{M} + \frac{M^{2}}{N} \bigg),
\end{align*}
uniformly for $k_+\in[M+1, \min\{N(\kappa_2-MN^{-2/3}),N(\mu-\delta)\}]\cap \N$ and $\mu\in [\delta,1-\delta]$.
\item Let $k_{+} = \lfloor N(\kappa_{2}-\frac{M}{N^{2/3}})\rfloor$. As $N \to +\infty$, we have 
\begin{align*}
\sum_{k=k_{-}}^{k_{+}}H(\kappa=\tfrac{k}{N},\mu) = N \int_{\kappa_{-}}^{\kappa_{2}} H(\kappa,\mu)d\kappa + \bigO\bigg( \frac{1}{\kappa_{-}N} + \frac{M^{2}}{N^{1/3}} \bigg)
\end{align*}
uniformly for $k_-\in[M+1, N(\kappa_2-MN^{-2/3})]\cap \N$ and for $\mu \in [\delta,1-\delta]\setminus (\frac{a^{2}}{1+a^{2}}-\delta,\frac{a^{2}}{1+a^{2}}+\delta)$.
\end{itemize} 
\end{lemma}
\begin{proof}Similarly as in the proof of Lemma \ref{lemma:Riemann sums}, we apply Lemma \ref{lemma:Riemannsum}, but now with $h(\kappa)=H(\kappa,\mu)$.
To estimate the error term, we use the bounds for $\partial_\kappa H(\kappa,\mu)$ and $\partial_\kappa^4 H(\kappa,\mu)$ from Propositions \ref{prop:kappatozero} and \ref{prop:kappatokappa2}. The two special cases are obtained in a similar way as in the proof of Lemma \ref{lemma:Riemann sums}, using Propositions \ref{prop:kappatozero} and \ref{prop:kappatokappa2}.
\end{proof}

\begin{lemma}
\label{lemma:Riemann sums2}
Let $M\in \N$ be fixed but sufficiently large, and let $\delta>0$ be fixed. As $N \to +\infty$, we have
\begin{align*}
\frac{1}{N}\sum_{k=k_{-}}^{k_{+}} F(\tfrac{k}{N},\mu) = \int_{\kappa_{-}}^{\kappa_{+}} F(\kappa,\mu) d\kappa + \bigO\bigg( \frac{1}{N^{2}\kappa_{-}^{2}} + \frac{1}{N^{2}(\kappa_{2}-\kappa_{+})^{2}} \bigg).
\end{align*}
uniformly for $k_{-},k_+\in[M+1, \min\{N(\kappa_2-MN^{-2/3}),N(\mu-\delta)\}]{\cap \N}$  and $\mu\in [\delta,1-\delta]$, and where $\kappa_{+},\kappa_{-}$ are as in \eqref{def of kappa pm}. In particular, the following hold.
\begin{itemize}
\item As $N \to +\infty$, we have 
\begin{align*}
\frac{1}{N}\sum_{k=M+1}^{k_{+}} F(\tfrac{k}{N},\mu) & = -\frac{1}{12} \log \frac{2M+1}{2N} + \int_{0}^{\kappa_{+}} \bigg( F(\kappa,\mu) -\frac{1}{12\kappa} \bigg)d\kappa \\
& + \frac{1}{12} \log \kappa_{+} + \bigO\bigg( \frac{1}{M^{2}} + \frac{M}{N} \bigg),
\end{align*}
uniformly for $k_+\in[M+1, \min\{N(\kappa_2-MN^{-2/3}),N(\mu-\delta)\}]\cap \N$ and $\mu\in [\delta,1-\delta]$.
\item Let $k_{+} = \lfloor N(\kappa_{2}-\frac{M}{N^{2/3}})\rfloor$. As $N \to +\infty$, we have 
\begin{multline*}
\frac{1}{N}\sum_{k=k_{-}}^{k_{+}} F(\tfrac{k}{N},\mu) =  \frac{1}{12} \log N - \frac{1}{8} \log M + \frac{1}{8} \log (\kappa_{2}-\kappa_{-}) \\
 + \int_{\kappa_{-}}^{\kappa_{2}} \bigg( F(\kappa,\mu) - \frac{1}{8}\frac{1}{\kappa_{2}-\kappa} \bigg)d\kappa   + \bigO\bigg( \frac{1}{N^{2}\kappa_{-}^{2}} + \frac{1}{M N^{1/3}} \bigg),
\end{multline*}
uniformly for $k_-\in[M+1, N(\kappa_2-MN^{-2/3})]\cap \N$ and for $\mu \in [\delta,1-\delta]\setminus (\frac{a^{2}}{1+a^{2}}-\delta,\frac{a^{2}}{1+a^{2}}+\delta)$.
\end{itemize} 
\end{lemma}
\begin{proof}
The proof is again a straightforward application of Lemma \ref{lemma:Riemannsum}, now with $h(\kappa)=F(\kappa,\mu)$.\end{proof}

Now we are going back to Proposition \ref{prop:ratio asymp in rough region} and sum that result from $k=M+1$ to $k_+$: this yields
\begin{multline*}\log F_N^{m,k_++1}=\log F_N^{m,M+1}+N\sum_{k=M+1}^{k_+}G(\tfrac{k}{N},\mu)
+\sum_{k=M+1}^{k_+}H(\tfrac{k}{N},\mu)
+\frac{1}{N}\sum_{k=M+1}^{k_+}F(\tfrac{k}{N},\mu)\\
+\bigO(M^{-1}+N^{-1}(\kappa_2-\kappa_{+})^{-2})
\end{multline*}
as $N\to\infty$ for $M$ large, where $\kappa_{+}$ is as in \eqref{def of kappa pm}.
Then we substitute the results from Lemmas \ref{lemma:Riemann sums}--\ref{lemma:Riemann sums2} as well as the expressions for $G_0,G_1,H_0$ from Proposition \ref{prop:kappatozero} and obtain
as $N \to +\infty$, 
\begin{align}
\log F_N^{m,k_++1}&=\log F_N^{m,M+1} + A_2(\kappa_+,\mu) N^{2} + A_1(\kappa_+,\mu) N+\left(\frac{M^2}{2}-\frac{1}{12}\right)\log N+A_0(\kappa_+,\mu) \nonumber \\
&\quad + \bigO\bigg( \frac{1}{M} + \frac{M^{3}}{N} + \frac{1}{N(\kappa_{2}-\kappa_{+})^{2}} \bigg), \label{lol33}
\end{align}
uniformly for $k_+\in[M+1, \min\{N(\kappa_2-MN^{-2/3}),N(\mu-\delta)\}]\cap \N$, where \begin{align*}
A_2(\kappa_+,\mu) &= 
\int_{0}^{\kappa_+} G(\kappa',\mu)d\kappa' ,  \\
A_1(\kappa_+,\mu) &= \int_{0}^{\kappa_+} H(\kappa',\mu) d\kappa'  -  (\tfrac{1}{2}+M)\varphi_0(x_0),  \\
A_0(\kappa_+,\mu) &= \int_{0}^{\kappa_+} \bigg( F(\kappa',\mu) -\frac{1}{12\kappa'} \bigg)d\kappa' + \frac{1}{12} \log \kappa_+ + \frac{1-6M^2}{12} \log (M+\tfrac{1}{2})  - \frac{ \partial_{\kappa_+}G(\kappa_+,\mu)}{24} \\
&\quad 
+\frac{M+3M^2}{4}
-\frac{1}{16}
+\frac{3M^2-1}{6}\log(|x_0|^2\varphi_{0}''(x_0))-\frac{\epsilon(1+2M)}{2}\log(1-ax_0)
.
\end{align*}

Now we substitute the asymptotics of Theorem \ref{thm:birth of a cut} for $\log F_N^{m,M+1}$ in \eqref{lol33}.
This yields
\begin{align*}
\log F_N^{m,k_{+}+1}= C_2(\kappa_+,\mu) N^{2} + C_1(\kappa_{+},\mu) N-\frac{1}{12}\log N+\widehat C_0(\kappa_{+},\mu)  + \bigO\bigg( \frac{1}{M} + \frac{M^{3}}{N} + \frac{1}{N(\kappa_{2}-\kappa_+)^{2}} \bigg),
\end{align*}
with 
\begin{align*}
&C_2(\kappa_+,\mu)=A_2(\kappa_+,\mu)+F_2(\mu),\\
&C_1(\kappa_+,\mu)=A_1(\kappa_+,\mu)+F_1(M,\mu),\\
&\widehat C_0(\kappa_+,\mu)=A_0(\kappa_+,\mu)+F_0(M,\mu)
\end{align*}
In other words,
\begin{align}
\label{def:C2}C_{2}(\kappa_+,\mu) &=\left(\frac{1}{2}-\mu+\mu^2\right)\log(1+a^2) +
\int_{0}^{\kappa_+} G(\kappa',\mu)d\kappa' ,  \\
\label{def:C1}C_{1}(\kappa_+,\mu) &= \int_{0}^{\kappa_+} H(\kappa',\mu) d\kappa'    -\frac{1}{2}\varphi_0(x_0)+\left(\frac{1}{2}-\epsilon\mu\right)\log(1+a^2),\\
\label{def:C0hat}\widehat C_{0}(\kappa_+,\mu) &= \int_{0}^{\kappa_+} \bigg( F(\kappa',\mu) -\frac{1}{12\kappa'} \bigg)d\kappa' + \frac{1}{12} \log \kappa_+ + \frac{1-6M^2}{12} \log (M+\tfrac{1}{2})  - \frac{ \partial_{\kappa_+}G(\kappa_+,\mu)}{24} \nonumber \\
&\hspace{-1.1cm} -\frac{M}{2}\log(2\pi)  +\frac{-1+4M+12M^2}{16}-\frac{1}{6}\log(|x_0|^2\varphi_{0}''(x_0))  -\frac{\epsilon}{2}\log(1-ax_0)+\log \mathcal G(M+1)
.
\end{align}
Here we observe that the asymptotics \eqref{asBarnes} of $\log \mathcal G(M+1)$ as $M\to\infty$ imply that 
\[\log \mathcal G(M+1)+\frac{1-6M^2}{12} \log (M+\tfrac{1}{2})  -\frac{M}{2}\log(2\pi)+\frac{-1+4M+12M^2}{16}=\zeta'(-1)+\bigO(1/M).\]
This implies that we can set
\begin{align}
C_0(\kappa_+,\mu)=\int_{0}^{\kappa_+} \bigg( F(\kappa',\mu) -\frac{1}{12\kappa'} \bigg)d\kappa' + \frac{1}{12} \log \kappa_+   - \frac{ \partial_{\kappa_+}G(\kappa_+,\mu)}{24} \nonumber \\
 -\frac{1}{6}\log(|x_0|^2\varphi_{0}''(x_0))  -\frac{\epsilon}{2}\log(1-ax_0)+\zeta'(-1), \label{def:C0}
\end{align}
and obtain 
\begin{align}\label{asympfinal0}
\log F_N^{m,k_++1}= C_2(\kappa_+,\mu) N^{2} + C_1(\kappa_+,\mu) N-\frac{1}{12}\log N+{C}_0(\kappa_+,\mu)  + \bigO\bigg( \frac{1}{M} + \frac{M^{3}}{N} + \frac{1}{N(\kappa_{2}-\kappa_+)^{2}} \bigg).
\end{align}
Since we can choose $M$ arbitrarily large, the error term is 
\begin{align*}
o(1)+\bigO\bigg( \frac{1}{N(\kappa_{2}-\kappa_+)^{2}} \bigg),
\end{align*}
which finishes the proof of \eqref{asymp main thm}. By Proposition \ref{prop:kappatozero}, we obtain the $\kappa_+\to 0$ asymptotics \eqref{asympCj0}. This completes the proof of Theorem \ref{thm:rough}.

\subsection{Consistency check and proof of \eqref{asympCjkappa2}}\label{subsection:end}

Alternatively, we can obtain asymptotics by summing the result of Proposition \ref{prop:ratio asymp in rough region} from $k_-=k+1$ to a value $k_+=k_+(N)$ defined by
\begin{align*}
k_++1=N\left(\kappa_2-\frac{M_N}{N^{2/3}}\right),
\end{align*}
for some $M_N$ satisfying $\lim_{N\to\infty}M_N=+\infty$, and which we will determine later.
By Proposition \ref{prop:ratio asymp in rough region}, we get
\begin{multline}\label{lol34}
\log F_N^{m,k+1}=\log F_N^{m,k_++1}-N\sum_{{j}=k+1}^{k_+}G({\tfrac{j}{N}},\mu)-\sum_{{j}=k+1}^{k_+}H({\tfrac{j}{N}},\mu) \\
-\frac{1}{N}\sum_{{j}=k+1}^{k_+}F({\tfrac{j}{N}},\mu)+\bigO\left(M_N^{-2}N^{1/3}\right)+\bigO\left(k^{-1}\right)
\end{multline}
as $N\to\infty$, uniformly for $k \in[M+1, N(\kappa_2-\delta)]\cap \N$ and for $\mu \in [\delta,1-\delta]\setminus (\frac{a^{2}}{1+a^{2}}-\delta,\frac{a^{2}}{1+a^{2}}+\delta)$, for some fixed $\delta>0$ and some fixed but large $M>0$.

Let $\kappa_{-} = \frac{k_{-}-\frac{1}{2}}{N} = \frac{k+\frac{1}{2}}{N}$. Now we use again Lemmas \ref{lemma:Riemann sums}--\ref{lemma:Riemann sums2} to obtain
\begin{align*}
& \log F_N^{m,k+1}=\log F_N^{m,k_++1}-N^2\int_{\kappa_-}^{\kappa_2}G(\kappa,\mu)d\kappa-N\int_{\kappa_-}^{\kappa_2}H(\kappa,\mu)d\kappa-\int_{\kappa_-}^{\kappa_2}\bigg(F(\kappa,\mu)-\frac{1}{8}\frac{1}{\kappa_{2}-\kappa}\bigg)d\kappa\\
&  -\frac{\partial_{\kappa}G(\kappa_-,\mu)}{24} +\frac{\tilde G_2}{3}M_{N}^3 -\frac{1}{12}\log N + \frac{1}{8}\log \frac{M_{N}}{\kappa_{2}-\kappa_{-}}+\bigO\left(M_N^{2}N^{-1/3}+ M_N^{-2}N^{1/3} + k^{-1}\right).
\end{align*}
Substituting the Tracy-Widom asymptotics \eqref{lol22}, we obtain
\begin{align*}
& \log F_N^{m,k+1}=\frac{N^2}{2}\log(1+a^2)
-N^2\int_{\kappa_-}^{\kappa_2}G(\kappa,\mu)d\kappa
+\frac{N}{2}\log(1+a^2) 
-N\int_{\kappa_-}^{\kappa_2}H(\kappa,\mu)d\kappa\\
&\quad -\int_{\kappa_-}^{\kappa_2}\bigg(F(\kappa,\mu)-\frac{1}{8}\frac{1}{\kappa_{2}-\kappa}\bigg)d\kappa-\frac{M_N^3}{12c^*}-\frac{1}{8}\log M_N +\frac{1}{24}\log (2c^*)+\zeta'(-1)
 -\frac{\partial_{\kappa}G(\kappa_-,\mu)}{24} \\
 &\quad +\frac{\tilde G_2}{3}M_N^3 -\frac{1}{12}\log N + \frac{1}{8}\log \frac{M_{N}}{\kappa_{2}-\kappa_{-}} +\bigO\left(M_N^{2}N^{-1/3} + M_N^{-2}N^{1/3} + k^{-1}\right).
\end{align*}
Here one should note that we need $M_N^2=o(N^{1/15})$ as $N\to\infty$ in order to apply \eqref{lol22}.
Let us set $M_N=N^{\frac{1}{15}-\frac{\eta}{2}}$ for a sufficiently small $\eta>0$.
We obtain
\begin{align}\label{asympfinalkappa2}
\log F_N^{m,k_-+1}&=N^2B_2(\kappa_-,\mu)+NB_1(\kappa_-,\mu)+\bigO(N^{\frac{1}{5}+\eta} + {k}^{-1}),
\end{align}
where
\begin{align}
&\label{def:B2}
B_1(\kappa_-,\mu)=\frac{1}{2}\log(1+a^2)
-\int_{\kappa_-}^{\kappa_2}G(\kappa,\mu)d\kappa,\\
&\label{def:B1}
B_2(\kappa_{{-}},\mu)=\frac{1}{2}\log(1+a^2)
-\int_{\kappa_-}^{\kappa_2}H(\kappa,\mu)d\kappa.
\end{align}
Comparing this with \eqref{asympfinal0}, we see that 
$B_1(\kappa,\mu)=C_1(\kappa,\mu)$ and that $B_2(\kappa,\mu)=C_2(\kappa,\mu)$.
In other words, we have the integral identities
\begin{subequations}\label{hard identities intro}
\begin{align}
& \int_{0}^{\kappa_{2}} G(\kappa,\mu)d\kappa = \mu(1-\mu)\log(1+a^{2}), \\
& \int_{0}^{\kappa_{2}}  H(\kappa,\mu)d\kappa = \frac{\varphi_{0}(x_{0})}{2} + \epsilon \, \mu  \log(1+a^{2}).
\end{align}
\end{subequations}
We can now substitute the results from Proposition \ref{prop:kappatokappa2} in \eqref{def:B2}--\eqref{def:B1} to obtain \eqref{asympCjkappa2}.

\begin{remark}
Both identities \eqref{hard identities intro} have been verified numerically. In fact, we even verified numerically that 
\begin{align*}
\int_{0}^{\kappa_{2}} \bigg( F(\kappa',\mu) -\frac{1}{12\kappa'} - \frac{1}{8}\frac{1}{\kappa_{2}-\kappa'} \bigg)d\kappa' = & - \frac{5}{24}\log \kappa_{2} + \frac{1}{6}\log\big( |x_{0}|^{2}\varphi_{0}''(x_{0}) \big) \\
& + \frac{\epsilon}{2}\log(1-ax_{0}) + \frac{1}{24}\log(2c^{*}).
\end{align*}
We can however not prove this because the error term in \eqref{asympfinalkappa2} is bigger than $\bigO(1)$.
\end{remark}

\paragraph{Acknowledgements.} CC is a Research Associate of the Fonds de la Recherche Scientifique - FNRS, and also acknowledges support from the European Research Council (ERC), Grant Agreement No. 101115687.
TC
acknowledges support by FNRS Research Project T.0028.23 and by the Fonds Sp\'ecial de Recherche
of UCLouvain. {The authors are grateful to Lennart H\"{u}bner and Max van Horssen for allowing us to use their fast Python program, which we used to generate Figure~\ref{fig:numerical confirmation} and to numerically verify our main results.}

\end{document}